\documentclass[11pt]{amsart}

\usepackage[utf8]{inputenc} 
\usepackage[T1]{fontenc} 
\usepackage{a4wide}
\usepackage{mathtools}
\usepackage{newtxtext}
\usepackage{newtxmath}
\usepackage{mathrsfs} 

\usepackage{tikz} 
\usepackage{tikz-cd}
\tikzcdset{
arrow style=tikz,
diagrams={>={To[scale=0.8]}}
}
\tikzcdset{every diagram/.append style = {sep=large}}

\usepackage{xspace} 
\usepackage{enumitem}
\usepackage{float}
\usepackage{old-arrows}
\usepackage{hyperref}

\swapnumbers
\theoremstyle{plain}
\newtheorem{theorem}{Theorem}[section]
\newtheorem{proposition}[theorem]{Proposition}
\newtheorem{lemma}[theorem]{Lemma}
\newtheorem{corollary}[theorem]{Corollary}
\newtheorem{observation}[theorem]{Observation}
\newtheorem{observations}[theorem]{Observations}
\newtheorem*{obs}{Observation}
\newtheorem*{intthm}{Theorem}
\newtheorem*{mainthm}{Theorem~\ref{mainthm}}
\newtheorem*{secmainthm}{Theorem~\ref{secmainthm}}
\theoremstyle{definition}
\newtheorem{definition}[theorem]{Definition}
\newtheorem{example}[theorem]{Example}
\theoremstyle{remark}
\newtheorem*{remark}{Remark}

\newcommand{\Fraisse}{Fra\"\i{}ss\'e\xspace}
\newcommand{\Konig}{K\H{o}nig\xspace}
\newcommand{\Kubis}{Kubi\'s\xspace}

\newcommand{\injto}{\hookrightarrow}
\newcommand{\injfrom}{\hookleftarrow}
\newcommand{\epito}{\twoheadrightarrow}
\newcommand{\Age}{\operatorname{Age}}
\newcommand{\restr}{\mathord{\upharpoonright}}
\newcommand{\down}{\!\downarrow\,}
\newcommand{\Aut}{\operatorname{Aut}}

\newcommand{\struc}[1]{\mathbf{#1}}	
\newcommand{\mstr}[1]{\mathcal{#1}}   
\newcommand{\class}[1]{\mathscr{#1}}	
\newcommand{\cat}[1]{\pmb{\mathscr{#1}}} 
\newcommand{\cats}[1]{\pmb{\mathscr{#1}}_{\!\!<\omega}} 
\newcommand{\func}{\mathrm}  
\newcommand{\cochain}[1]{\overleftarrow{\func{#1}}} 
\newcommand{\chain}[1]{\overrightarrow{\func{#1}}} 
\newcommand{\tup}[1]{\mathbf{#1}}
\newcommand{\ttup}[1]{\overline{\tup{#1}}}
\newcommand{\ul}{\underline}
\newcommand{\ol}{\overline}
\newcommand{\plim}{\varprojlim}
\newcommand{\comma}{\mathop{\downarrow}}
\newcommand{\lex}{{\operatorname{lex}}} 

\makeatletter
\newcommand{\genericprojlim}[1]{%
  \mathop{\mathpalette\myvarlim@{{#1}{\leftarrowfill@\textstyle}}}\nmlimits@
}
\def\myvarlim@#1#2{\myvarlim@@#1#2}
\def\myvarlim@@#1#2#3{%
  \vtop{\m@th\ialign{##\cr
    \hfil$#1\operator@font #2$\hfil\cr
    \noalign{\nointerlineskip\kern1.5\ex@}#3\cr
    \noalign{\nointerlineskip\kern-\ex@}\cr}}%
}
\makeatother
\newcommand{\mlim}{\genericprojlim{Lim}}

\newcommand{\defeq}{\vcentcolon\Leftrightarrow}

\newcommand{\JEP}{\ensuremath{\operatorname{JEP}}\xspace}
\newcommand{\HP}{\ensuremath{\operatorname{HP}}\xspace}
\newcommand{\AP}{\ensuremath{\operatorname{AP}}\xspace}
\newcommand{\HAP}{\ensuremath{\operatorname{HAP}}\xspace}
\newcommand{\AEP}{\ensuremath{\operatorname{AEP}}\xspace}

\newcommand{\ar}{\operatorname{ar}}
\newcommand{\op}{{\operatorname{op}}}
\newcommand{\Rho}{\mathrm{P}}
\newcommand{\bR}{\mathbb{R}}
\newcommand{\Seq}{\operatorname{Seq}}

\newcommand{\R}{\struc{R}} 
\newcommand{\G}{\mathbb G} 

\author[W.~Kubi\'s]{Wies\l aw Kubi\'s}
\address{Institute of Mathematics\\Czech Academy of Sciences\\\v{Z}itn\'a 25, 115 67 Prague, Czech Republic}
\email{kubis@math.cas.cz}
\urladdr{https://users.math.cas.cz/kubis/}

\author[Ch.\,Pech]{Christian Pech}
\email{pech@math.cas.cz}
\urladdr{https://www.researchgate.net/profile/Christian\_Pech2}

\thanks{The research of the first and the second author was supported by GA~\v{C}R (Czech Science Foundation) grant EXPRO 20-31529X}

\author[M.\,Pech]{Maja Pech} 
\address{Department of Mathematics and Informatics\\University of Novi Sad\\21000 Novi Sad\\ Serbia} 
\email{maja@dmi.uns.ac.rs}
\urladdr{http://people.dmi.uns.ac.rs/~maja/}

\thanks{The third author acknowledges financial support of the Ministry of Education, Science and Technological Development of the Republic of Serbia (Grant No. 451-03-68/2022-14/200125).}

\title{Homogeneous Ultrametric Structures}
\begin{document}

\begin{abstract}
	We develop the theory of homogeneous Polish ultrametric structures. Our starting point is a \Fraisse class of finite structures and the crucial tool is the universal homogeneous epimorphism. The new \Fraisse limit is an inverse limit, nevertheless its universality is with respect to embeddings and, contrary to the Polish metric \Fraisse theory of Ben Yaacov~\cite{Yaa15}, homogeneity is strict. Our development can be viewed as the third step of building a Borel-like hierarchy of \Fraisse limits, where the first step was the classical setting of \Fraisse and the second step is the more recent theory, due to Irwin and Solecki~\cite{IrvSol06}, of pro-finite \Fraisse limits.
\end{abstract}

\maketitle

\tableofcontents

\section{Introduction}

Trees and tree-like structures seem to be ubiquitous in pure mathematics and theoretical computer science.
A pure tree is typically understood as a connected simple graph with no cycles, or as a partially ordered set in which all sets of the form $\{x \colon x < p\}$ are well-ordered, or even finite linearly ordered. In the last case, we call it an $\omega$-tree, as its branches are of order type $\omega$---the natural numbers.
An $\omega$-tree can be called \emph{pruned} if every element (node) belongs to a branch, namely, there are no leaves.
As it happens, pruned $\omega$-trees are in one-to-one correspondence with inverse sequences of surjections between their levels.
Specifically, given an inverse sequence of surjections
\[\begin{tikzcd}
	L_0 & L_1 \ar[l, two heads] & L_2 \ar[l, two heads] & \cdots \ar[l, two heads]
\end{tikzcd}
\]
one obtains a pruned $\omega$-tree whose levels are the $L_i$s and the ordering is given by the bonding mappings, augmented by a smallest element $\bot$---the root. Conversely, a pruned $\omega$-tree induces an inverse sequence of surjections between its levels, determined by the ordering.

It is natural to add an extra structure on the levels of the tree and require that the bonding mappings (induced by the ordering) preserve this structure. By this way we may obtain, e.g.,  lexicographic trees, namely, where each level is linearly ordered. But there are many other possibilities, for example, graphs, groups, metric spaces, etc. The limit of such a tree is often called a \emph{pro-finite} structure. It is indeed a structure of the same type as the levels of the tree, declaring all the relations and operations in the obvious way.

There is yet another structure on the limit (the set of all branches) of a pruned $\omega$-tree. Specifically, one can naturally measure how far from each other are two distinct branches, saying their distance is $2^{-n}$  if $n$ is the number of levels, where they agree.
In this setting, all the levels of the tree are treated as discrete (equilateral, all pairs of points of the same distance) metric spaces and the limit is a complete ultrametric space in which the levels could often be reconstructed as closed balls of a fixed radius.

Combining both ideas, we can obtain complete ultrametric mathematical structures as inverse limits of ``structured'' trees and it is natural to ask how complicated they could be.

For example, it follows from the results of~\cite{Kub14} that there exists a graph epimorphism $\Omega \colon \R \to \R$, where $\R$ denotes the Rado graph, which is most complicated in the sense that it captures all homomorphisms into the Rado graph and is homogeneous with respect to finite homomorphisms of graphs. This special epimorphism is a retraction (i.e., it has a right inverse) and it is unique up to isomorphism. Repeating $\Omega$ infinitely many times, we obtain a sequence/tree
\[\begin{tikzcd}
	\R & \ar[l, two heads, "\Omega"'] \R & \ar[l, two heads, "\Omega"'] \R \ar[l, two heads, "\Omega"'] &  \ar[l, two heads, "\Omega"'] \cdots
\end{tikzcd}
\]
whose limit is a highly homogeneous graph on the Baire space, namely, the universal Polish space. The natural representation of the Baire space is the set of all infinite sequences with values in the natural numbers, or any other fixed countable set. Here we obtain a concrete graph representation of the Baire space, based on the Rado graph instead of the pure countable infinite set.

What is the role of the Rado graph here? It turns out that the class of finite graphs has very nice amalgamations, therefore the universal homomorphism $\Omega$ exists. In general, one needs to impose some restrictions onto the class of finite (or finitely generated) structures in order to get an analogue of $\Omega$ that would be the main building block of the most complicated ultrametric structure.

The purpose of this paper is to present a theory explaining highly homogeneous ultrametric structures that could be called \Fraisse limits, perhaps two-dimensional, as they rely first on the existence of a universal homogeneous homomorphism $\Omega$ and then showing that the inverse limit of repeating $\Omega$ infinitely often leads to an ultrametric structure with high homogeneity and universality.

On the one hand, we are using inverse limits, so our theory resembles projective \Fraisse theory of Irwin and Solecki~\cite{IrvSol06}, while on the other hand we obtain metric structures directly, resembling the metric \Fraisse theory~\cite{Yaa15}, however we obtain strict homogeneity, contrary to what is typically happening in continuous \Fraisse theory.
Furthermore, our research could also be viewed as the first attempt to lift the theory of \Fraisse limits to a higher level, thus possibly initiating a hierarchy, in the same spirit as the Borel hierarchy in descriptive set theory. The first step is the classical \Fraisse theory, the second one is Irwin -- Solecki projective \Fraisse theory~\cite{IrvSol06} that could be understood as the compact (profinite) theory. We are proposing the next step, that could be called the \emph{$G_\delta$ theory}, as discrete structures could be thought of as open and their limits are $G_\delta$ structures. Closed structures could be inverse limits of finite ones, so it is tempting to look for the theory of generic limits in the class $F_\sigma$, consisting of colimits of profinite structures. But this is beyond the scope of this paper.

Furthermore we are working in pure category theory, with no metric enrichment, therefore we are not considering almost commuting diagrams, even though they make sense, as our objects are metric spaces after all.

Our main results show that under suitable assumptions there exists a unique complete ultrametric structure that is ultra-homogeneous, both with respect to isometric isomorphisms and more general embeddings, preserving distances in a suitable natural way.

In particular, as a special case, we obtain the following result.

\begin{intthm}
	There exists a closed graph relation on the Baire space $\omega^\omega$ such that for every finite induced subgraph $F$, every finite extension $G \supseteq F$ is realized in $\omega^\omega$, even preserving the metric, in case $G$ is viewed as an ultrametric graph with distances $2^{-n}$, $n \in \mathbb{N}$.
	Furthermore, the graph relation is uniquely determined by this extension property. 
	
	Finally, every closed graph on a Polish ultrametric space embeds into the graph above.
\end{intthm}

This result is a special case of a much more general theorem involving a relational language, where a certain stronger variant of the amalgamation property is required. This is explained below in detail.

Let us add that the main idea of this note actually appeared (most likely for the first time) in~\cite{BKK17} in the context of Banach and Fr\'echet spaces. Namely, the Gurarii space  $\G$ is the \Fraisse limit of finite-dimensional Banach spaces, in the metric-enriched category-theoretic sense. It turns out that there exists a universal homogeneous non-expansive operator on $\G$ (constructed in~\cite{GarKub15}) and repeating it infinitely many times one obtains an inverse sequence whose limit is a universal (almost) homogeneous separable graded Fr\'echet space. 

This indicates that it should be possible to extend the results of our note to metric \Fraisse theory~\cite{Yaa15, K41}.

\vspace{2mm}

\begin{center}
	***
\end{center}

In 1953 Roland \Fraisse introduced the notion of an ``age'' as a tool to specify a class of (relational) structures \cite{Fra53a} . 

If $\struc{A}$ is a model-theoretic structure, then by $\Age(\struc{A})$ we denote the class of all finitely generated structures that embed into $\struc{A}$. Here a model-theoretic structure is specified through
\begin{enumerate}
	\item a signature $\Sigma=(\Phi,\Rho,\ar)$ where $\Phi$ and $\Rho$ are mutually disjoint sets of operational symbols and relational symbols, respectively, and where $\ar$ is a function that assigns to each element of $\Phi\cup\Rho$ its arity ($0$-ary operational symbols take the role of constant symbols),
	\item a tuple $\struc{A}=(A,(f^\struc{A})_{f\in\Phi}, (\varrho^\struc{A})_{\varrho\in\Rho})$, where for all $f\in\Phi$, $\varrho\in\Rho$ we have that $f^\struc{A}\colon A^{\ar(f)} \to A$, and $\varrho^\struc{A}\subseteq A^{\ar(\varrho)}$. The operations $f^\struc{A}$ of $\struc{A}$ are also called the \emph{basic operations} of $\struc{A}$. Similarly, the relations $\varrho^\struc{A}$ of $\struc{A}$ are called the \emph{basic relations} of $\struc{A}$.
\end{enumerate}

A structure $\struc{B}$ is said to be \emph{younger} than $\struc{A}$ if $\Age(\struc{B})\subseteq \Age(\struc{A})$ (originally, \Fraisse used the notation $\gamma_{\struc{A}}$ for $\Age(\struc{A})$ and denoted the class of structures younger that $\struc{A}$ by $\Gamma_{\struc{A}}$). 

Initially \Fraisse introduced ages for his famous characterization of countable homogeneous structures (recall that a structure $\struc{A}$ is called \emph{homogeneous} if every isomorphism between finitely generated substructures of $\struc{A}$ extends to an automorphism of $\struc{A}$):
   \begin{theorem}[{\cite{Fra53}}]
   	  Let $\class{C}$ be a class of finitely generated structures of the same type. Then $\class{C}$ is the age of a countable homogeneous structure if and only if 
   	  \begin{enumerate}
   	  	\item up to isomorphism $\class{C}$ contains countably many structures,
   	  	\item $\class{C}$ has the hereditary property (\HP), i.e., 
   	  	\[
   	  	\forall\struc{A},\struc{B}\,:\,\struc{B}\in\class{C},\,\struc{A}\injto\struc{B} \implies \struc{A}\in\class{C},
   	  	\]
   	  	\item $\class{C}$ has the joint embedding property (\JEP), i.e.,
   	  	\[
   	  	\forall \struc{A},\struc{B}\in\class{C}\,\exists\struc{C}\in\class{C}\,:\,\struc{A}\injto\struc{C},\,\struc{B}\injto\struc{C},
   	  	\] 
   	  	\item $\class{C}$ has the amalgamation property (\AP), i.e., 
   	  	\[
   	  	\forall\struc{A},\struc{B}_1,\struc{B}_2\in\class{C},\,\iota_1\colon\struc{A}\injto\struc{B}_1,\,\iota_2\colon\struc{A}\injto\struc{B}_2\,\exists \struc{C}\in\class{C},\,\kappa_1\colon\struc{B}_1\injto\struc{C},\,\kappa_2\colon\struc{B}_2\injto\struc{C} :
   	  	\]
   	  	the following diagram commutes:
   	  	\[
   	  		\begin{tikzcd}
   	  		\struc{B}_1 \ar[r,"\kappa_1",hook,dashed]& \struc{C}\\
   	  		\struc{A} \ar[u,"\iota_1",hook']\ar[r,"\iota_2",hook]&\struc{B}_2\ar[u,"\kappa_2"',hook',dashed]
   	  	\end{tikzcd}
   	  	\]
   	  \end{enumerate}
   	  Moreover, any two countable homogeneous structures with the same age are isomorphic.
   \end{theorem}
   The proof of this theorem usually involves the construction of an $\omega$-tower
   \[
   \begin{tikzcd}
   \struc{A}_0 \ar[r,"\le",hook] & \struc{A}_1 \ar[r,"\le",hook] & \struc{A}_2 \ar[r,"\le",hook] & \struc{A}_3 \ar[r,"\le",hook] & \dots    	
   \end{tikzcd}
   \]
   of structures from $\class{C}$. Clearly, the union
   \[
   \struc{A}_\infty\coloneqq \bigcup_{i<\omega}\struc{A}_i 
   \]
   has age $\class{C}$ if and only if the family $(\struc{A}_i)_{i<\omega}$ is \emph{cofinal} in $\class{C}$, i.e., 
   \[
   \forall \struc{B}\in\class{C}\,\exists i<\omega\,:\,\struc{B}\injto\struc{A}_i.
   \]
   It was observed in \cite{Kub14} that if $\Age(\struc{A}_\infty)=\class{C}$, then $\struc{A}_\infty$ is homogeneous if and only if the tower $(\struc{A}_i)_{i<\omega}$ has the \emph{absorption property}:
   \[
   \forall\struc{B}\in\class{C},\, i<\omega,\,\iota\colon\struc{A}_i\injto\struc{B}\,\exists k<\omega,\,\hat\iota\colon\struc{B}\injto\struc{A}_k\,:\, \hat\iota\circ\iota=\iota_i^k,
   \] 
   where $\iota_i^k$ is the inclusion homomorphism of $\struc{A}_i$ into $\struc{A}_k$. 
   
   Starting with an age $\class{C}$ the existence of a cofinal $\omega$-tower $(\struc{A}_i)_{i<\omega}$ with the absorption property may be proved using an abstract version of Banach-Mazur games (see \cite{KraKub21}). 
   
   All this emphasizes the role of $\omega$-towers when studying countable structures. A class $\class{C}$ of finitely generated structures is called a \emph{hereditary class} if it has the \HP.  Clearly, the class of all countably generated structures whose age is contained in a given hereditary class $\class{C}$ coincides with the class of all those structures that can be expressed as the union of an $\omega$-tower of elements from $\class{C}$. This motivates the following definition:
   \begin{definition}
   		Let $\class{C}$ be a class of structures of the same type. Then we define
   		\[
   		\sigma\class{C} \coloneqq \{\struc{A}\mid\struc{A}\text{ is the union of an $\omega$-tower of elements from $\class{C}$}\}.
   		\]
   \end{definition}
   Classes of the shape $\sigma\class{C}$ have been the subject of numerous investigations in model theory and in combinatorics. Questions of interest are, e.g., about the existence of universal structures, generic structures, Ramsey-structures, or homogeneous structures in $\sigma\class{C}$.

\section{Chains and cochains}

In this section we discuss connections between inverse sequences (called $\omega$-cochains) and ultrametrics. We also define the concept of an ultrametric structure, whose underlying space is the inverse limit of a sequence od countable discrete spaces.

\subsection{From colimits of \texorpdfstring{$\omega$}{omega}-chains to limits of \texorpdfstring{$\omega$}{omega}-cochains}
The construction of structures as unions of towers of structures is a special case of a very general, category theoretical construction---colimits of functors (cf. \cite{Mac98}).
In the present case we start with the category $\cat{S}_\Sigma$ that has as object-class the class  $\class{S}_\Sigma$ of all $\Sigma$-structures and as morphisms the homomorphisms between $\Sigma$-structures. 
  In $\cat{S}_\Sigma$ we fix the subcategory $\cat{C}$ induced by $\class{C}$. In this setting an $\omega$-tower is simply a special case of a functor from $\cat{\omega}$ to $\cat{C}$ (here  $\cat{\omega}$ denotes the category that has finite ordinals as objects and morphisms $i\to j$ for all $i\le j<\omega$). The  $\omega$-tower $(\struc{A}_i)_{i<\omega}$ corresponds to the functor $F\colon\cat{\omega}\to\cat{C}$, where $F(i)=\struc{A}_i$ and where $F(i\to j)$ is the identical embedding of $\struc{A}_i$ into $\struc{A}_j$. The union $\struc{A}_\infty=\bigcup_{i<\omega}\struc{A}_i$ is a colimit of $F$ in $\cat{S}_\Sigma$ with limiting cocone $(\alpha_i^\infty)_{i<\omega}$, where $\alpha_i^\infty\colon\struc{A}_i\injto\struc{A}_\infty$ is the identical embedding (in order for this to be precise we need to consider $F$ as a functor from $\cat{\omega}$ to $\cat{S}_\Sigma$, which is not a problem since $\cat{C}$ is a subcategory of $\cat{S}_\Sigma$) .
  
  It is not hard to see that for every hereditary class $\class{C}$ we have that $\sigma\class{C}$ consists of all colimits of functors $F\colon\cat{\omega}\to\cat{C}$ in $\cat{S}_\Sigma$ for which $F$ maps morphisms to embeddings. This motivates the following definition:
  \begin{definition}
  	Let $\class{C}$ be a class of structures of the same type. An \emph{$\omega$-chain} in $\cat{C}$ is a functor $\chain{A}\colon\cat{\omega} \to\cat{C}$ that maps morphisms of $\cat{\omega}$ to embeddings in $\cat{C}$.  If $\chain{A}$ acts like $i\mapsto \struc{A}_i$ on objects and as $(i\to j)\mapsto\alpha_i^j$ on morphisms then $\chain{A}$ will be denoted also by $((\struc{A}_i)_{i<\omega}, (\alpha_i^j)_{i\le j<\omega})$. A simplified representation of $\chain{A}$ as a diagram is given by:
  	\[
  	\begin{tikzcd}
  		\struc{A}_0\ar[r,"\alpha^0_1",hook] & \struc{A}_1\ar[r,"\alpha^1_2",hook] & \struc{A}_2\ar[r,"\alpha^2_3",hook] & \struc{A}_3\ar[r,"\alpha^3_4",hook] & \dots
  	\end{tikzcd}.
  	\]
  \end{definition} 
  By $\cat{\sigma\class{C}}$ we denote the full subcategory of $\cat{S}_\Sigma$ that is induced by $\sigma\class{C}$.  Thus $\cat{\sigma\class{C}}$ is the category of colimits of $\omega$-chains in $\cat{C}$. Each time when a notion is defined in terms of categories, we get for free other notions obtained by dualization. In the present case the most reasonable choice for a dual notion of colimits of $\omega$-chains is given by the notion of limits of $\omega$-cochains.  
  \begin{definition}
  	Let $\class{C}$ be a class of structures of the same type. An \emph{$\omega$-cochain} in $\cat{C}$ is a functor $\cochain{A}\colon\cat{\omega}^\op\to\cat{C}$. If $\cochain{A}$ acts on objects like $i\mapsto\struc{A}_i$ and on morphisms as $(i\leftarrow j)\mapsto \alpha^j_i$, then $\cochain{A}$ will also be denoted by $((\struc{A}_i)_{i<\omega}, (\alpha^j_i)_{i\le j<\omega})$. A simplified representation of $\cochain{A}$ as a diagram is given by:
  	\[
  	\begin{tikzcd}
  		\struc{A}_0 & \struc{A}_1\ar[l,"\alpha^1_0"']  & \struc{A}_2\ar[l,"\alpha^2_1"'] & \struc{A}_3\ar[l,"\alpha^3_2"'] & \dots\ar[l,"\alpha^4_3"'] 
  	\end{tikzcd}.
  	\]
  \end{definition}
  A limit of the $\omega$-cochain $\cochain{A}=((\struc{A}_i)_{i<\omega}, (\alpha^j_i)_{i\le j <\omega})$ is given by the substructure $\struc{A}_\infty$ of $\prod_{i<\omega}\struc{A}_i$ that is induced by all those tuples $\tup{c}=(c_i)_{i<\omega}$ satisfying
  \[
  \forall i<j<\omega\,:\, \alpha^j_i(c_j)=c_i.
  \]
  The limiting cone witnessing this fact is given by the family $(\alpha^\infty_i)_{i<\omega}$ of projection homomorphisms, where
  \[
  \alpha^\infty_i\colon \struc{A}_\infty\to\struc{A}_i\,:\, \tup{c}\mapsto c_i.
  \] 
  This cone will be called the \emph{canonical cone} of $\cochain{A}$. The structure $\struc{A}_\infty$ will be called the \emph{canonical limit} of $\cochain{A}$. 
  
  For a better intuition it is helpful to envision $\cochain{A}$ as a tree $T_{\cochain{A}}$ and $A_\infty$ as the set of its branches. The node set of $T_{\cochain{A}}$ is given by the set $\bigcup_{i<\omega} \{i\}\times A_i$. In $T_{\cochain{A}}$ we define
  \[
  		(i,a_i)\sqsubseteq (j,a_j) :\iff i\le j \text{ and } \alpha^j_i(a_j) = a_i.
  	\]
    
  \begin{observations}
  	For each $\omega$-cochain $\cochain{A}$  we have
  	\begin{enumerate}
  		\item $(T_{\cochain{A}},\sqsubseteq)$ is a tree, i.e., $\forall (j,a_j)\in T_{\cochain{A}}\,:\, (j,a_j)\down = \{(i,a_i)\in T_{\cochain{A}}\mid (i,a_i)\sqsubseteq(j,a_j)\}$ is a chain,
  		\item $(T_{\cochain{A}},\sqsubseteq)$ is well-founded, i.e., $\forall (i,a_i)\in T_{\cochain{A}}\,:\, (i,a_i)\down$ is finite,
  		\item $(T_{\cochain{A}},\sqsubseteq)$ is pruned if and only if for all $i\le j<\omega$ we have that $\alpha_i^j$ is surjective.
  	\end{enumerate}
  \end{observations}
As usually, the maximal chains in $(T_{\cochain{A}},\sqsubseteq)$ are called \emph{branches}. Coming back to the consideration of $A_\infty$, an important observation is that there is a one-to-one correspondence between the elements of $A_\infty$ and the branches of $T_{\cochain{A}}$. It is obtained by assigning to each element $\tup{a}=(a_i)_{i<\omega}\in A_\infty$ the set $\{(i,a_i)\mid i<\omega\}\subseteq T_{\cochain{A}}$. With a convenient definition of families (as functions from an index set to a set) one could argue that $(a_i)_{i<\omega}$, is actually equal to a branch of $T_{\cochain{A}}$

The canonical cone $(\alpha_i^\infty)_{i<\omega}$ of an $\omega$-cochain $\cochain{A}=((\struc{A}_i)_{i<\omega}, (\alpha^j_i)_{i\le j <\omega})$ naturally induces an ultrametric $\delta_{A_\infty}$ on $A_\infty$ that may be defined through 
\[
\delta_{A_\infty}(\tup{a},\tup{b}) \coloneqq \begin{cases}
2^{-\Delta(\tup{a},\tup{b})} & \tup{a}\neq\tup{b}\\
 	0 & \text{else,}
 \end{cases}
\] 
where
\[
	\Delta(\tup{a},\tup{b}) \coloneqq\min\{i<\omega\mid \alpha^\infty_i(\tup{a})\neq\alpha^\infty_i(\tup{b})\}= \min\{i<\omega\mid a_i\neq b_i\}. 
\]
\begin{observations} With the notions from above:
\begin{enumerate}
	\item \label{um1} $(A_\infty,\delta_{A_\infty})$ is a complete metric space,
	\item \label{um2} all basic operations of $\struc{A}_\infty$ are $1$-Lipschitz with respect to $\delta_{A_\infty}$,
	\item \label{um3} all basic relations of $\struc{A}_\infty$ are closed (in the appropriate product topology). 
\end{enumerate}
Moreover, if for each $i<\omega$ we have that the image of $\alpha^\infty_i$ is countable, then $(A_\infty,\delta_{A_\infty})$ is separable.
\end{observations}
\begin{proof}
	\textbf{about \ref{um1}:} Let $(\tup{a}_i)_{i<\omega}$ be a Cauchy-sequence in $A_\infty$. Then for every $M<\omega$ there exists $n_0=n_0(M)$, such that for all $n,m>n_0$ we have that $\alpha^\infty_M(\tup{a}_n) = \alpha^\infty_M(\tup{a}_m)=:b_M$. Clearly, the family $(b_i)_{i<\omega}$ forms a branch $\tup{b}$ in $T_{\cochain{A}}$.   Hence it is an element of  $A_\infty$. Moreover, $\lim_{i\to\infty} \tup{a}_i = \tup{b}$. 
	
	\textbf{about \ref{um2}:} Let $f^{\struc{A}_\infty}$ be an $n$-ary basic operation of $\struc{A}_\infty$. Let $\ttup{a}=(\tup{a}^{(0)},\dots,\tup{a}^{(n-1)})$, and $\ttup{b}=(\tup{b}^{(0)},\dots,\tup{b}^{(n-1)})$ be from $(A_\infty)^n$. Then
	\[
	\delta_{A_\infty}(\ttup{a}, \ttup{b}) = \max_{0\le i<n} \delta_{A_\infty}(\tup{a}^{(i)},\tup{b}^{(i)}) = 2^{-M},
	\]
	for some $M<\omega$. In particular, for all $m<M$ and for all $0\le i<n$ we have 
	\[ \tag{*}\label{star}
	\alpha^\infty_m(\tup{a}^{(i)}) =\alpha^\infty_m(\tup{b}^{(i)}).
	\] 
	We go on computing 
	\begin{align*}
		f^{\struc{A}_\infty}(\tup{a}^{(0)},\dots,\tup{a}^{(n-1)})&=\tup{c}, & f^{\struc{A}_\infty}(\tup{b}^{(0)},\dots, \tup{b}^{(n-1)}) &=\tup{d},\text{ where}\\
		f^{\struc{A}_i}(a_i^{(0)},\dots, a_i^{(n-1)}) &=c_i, & f^{\struc{A}_i}(b_i^{(0)},\dots,b_i^{(n-1)}) &= d_i, \text{ for all } i<\omega.
	\end{align*}
	By \eqref{star} we have $c_i=d_i$, for all $0\le i<M$.   Consequently $\delta_{A_\infty}(\tup{c},\tup{d})\le 2^{-M}$.
	
	\textbf{about \ref{um3}:} Let $\varrho^{\struc{A}_\infty}\subseteq (A_\infty)^n$ be a basic relation of $\struc{A}_\infty$. Let $\sigma\coloneqq (A_\infty)^n\setminus\varrho^{\struc{A}_\infty}$. Let us show that $\sigma$ is open: Let $\ttup{a}=(\tup{a}^{(0)},\dots,\tup{a}^{(n-1)})\in\sigma$. Let $i_0<\omega$ be minimal with the property that $(a^{(0)}_{i_0},\dots,a^{(n-1)}_{i_0})\notin\varrho^{\struc{A}_{i_0}}$. Let $\ttup{b}=(\tup{b}^{(0)},\dots,\tup{b}^{(n-1)})\in (A_\infty)^n$ be such, that $\delta_{A_\infty}(\ttup{a},\ttup{b})< 2^{-i_0}$. Then $(b^{(0)}_{i_0},\dots,b^{(n-1)}_{i_0}) = (a^{(0)}_{i_0},\dots,a^{(n-1)}_{i_0})\notin\varrho^{\struc{A}_{i_0}}$. Thus $\ttup{b}\in\sigma$. Consequently, $\sigma$ is open. 
\end{proof}

\subsection{From \texorpdfstring{$\omega$}{omega}-cochains to metric structures}
In the following let $\cochain{A}=((\struc{A}_i)_{i<\omega},(\alpha_i^j)_{i\le j<\omega})$ be an $\omega$-cochain in $\cat{S}_\Sigma$. Let $\struc{A}_\infty$ be its canonical limit and let $(\alpha_i^\infty)_{i<\omega}$ be its canonical cone. 
With each basic relation $\varrho^{\struc{A}_\infty}\subseteq (A_\infty)^n$ we may associate a function 
\begin{equation*}
	\ul\varrho^{\struc{A}_\infty}\colon(A_\infty)^n\to[0,1],\quad (\tup{x}^{(0)},\dots,\tup{x}^{(n-1)})\mapsto
	\begin{cases}
		0 & (\tup{x}^{(0)},\dots,\tup{x}^{(n-1)})\in\varrho^{\struc{A}_\infty}\\
		2^{-M(\tup{x}^{(0)},\dots,\tup{x}^{(n-1)})} & \text{else,}
	\end{cases}
\end{equation*}
where
\[
M(\tup{x}^{(0)},\dots,\tup{x}^{(n-1)}) = \min\{i<\omega\mid (x^{(0)}_i,\dots,x^{(n-1)}_i)\notin\varrho^{\struc{A}_i}\}.
\]
\begin{observation}
	For each basic relation $\varrho^{\struc{A}_\infty}$ of $\struc{A}_\infty$, we have that $\ul\varrho^{\struc{A}_\infty}$ is $1$-Lipschitz. Moreover, we have that 
	\[
	\varrho^{\struc{A}_\infty} = \{\tup{x}\mid \ul\varrho^{\struc{A}_\infty}(\tup{x})=0\}.
	\]
\end{observation}
\begin{proof}
	The second claim follows directly from the definition of $\ul\varrho^{\struc{A}_\infty}$. For the proof of the first claim take  $\ttup{x}=(\tup{x}^{(0)},\dots,\tup{x}^{(n-1)})$ and $\ttup{y}=(\tup{y}^{(0)},\dots,\tup{y}^{(n-1)})$ from $(A_\infty)^n$, such that $\ttup{x}\neq\bar{\tup{y}}$. Then, for some $M<\omega$ we have that 
	\[
	\delta_{A_\infty}(\ttup{x},\ttup{y})= \max_{0\le j<n}\delta_{A_\infty}(\tup{x}^{(j)},\tup{y}^{(j)}) = 2^{-M}.
	\]
	In particular for all $j\in\{0,\dots,n-1\}$ and for all $0\le i<M$ we  have that $x^{(j)}_i = y^{(j)}_i$.
	Next we distinguish cases depending on whether $\tup{x}$ and $\tup{y}$ are in $\varrho^{\struc{A}_\infty}$ or not. 
	\begin{enumerate}[label=\textbf{case \arabic*:}, ref=\arabic*,nosep,align=left,leftmargin=0em,labelindent=0em,itemindent=1em,labelsep=0.5em,labelwidth=!]
		\item Suppose that $\ttup{x},\ttup{y}\in\varrho^{\struc{A}_\infty}$. Then $|\ul\varrho^{\struc{A}_\infty}(\ttup{x})-\ul\varrho^{\struc{A}_\infty}(\ttup{y})|=0<\delta_{A_\infty}(\ttup{x},\ttup{y})$. 
		\item Suppose that $\ttup{x}\in\varrho^{\struc{A}_\infty}$,  $\ttup{y}\notin\varrho^{\struc{A}_\infty}$. Then $|\ul\varrho^{\struc{A}_\infty}(\ttup{x})-\ul\varrho^{\struc{A}_\infty}(\ttup{y})|=\ul\varrho^{\struc{A}_\infty}(\ttup{y})=2^{-\widehat{M}}$ for some $\widehat{M}<\omega$. In particular, 
		$(y^{(0)}_{\widehat{M}},\dots,y^{(n-1)}_{\widehat{M}})\notin\varrho^{\struc{A}_{\widehat{M}}}$. Since on the other hand we have that  
		$(x^{(0)}_{\widehat{M}},\dots,x^{(n-1)}_{\widehat{M}})\in\varrho^{\struc{A}_{\widehat{M}}}$, 
		it follows that $(x^{(0)}_{\widehat{M}},\dots,x^{(n-1)}_{\widehat{M}}) \neq (y^{(0)}_{\widehat{M}},\dots,y^{(n-1)}_{\widehat{M}})$. Consequently, $\widehat{M}\ge M$. 
		\item Suppose that $\ttup{x},\ttup{y}\notin\varrho^{\struc{A}_\infty}$. Then $\ul\varrho^{\struc{A}_\infty}(\ttup{x}) = 2^{-M_1}$ and $\ul\varrho^{\struc{A}_\infty}(\ttup{y})=2^{-M_2}$, for certain $M_1,M_2<\omega$. let $\widehat{M}\coloneqq\min(M_1,M_2)$. 
		\begin{enumerate}[label=\textbf{case \theenumi.\arabic*:}, ref=\theenumi.\arabic*,align=left,leftmargin=0em,labelindent=0em,itemindent=1em,labelsep=0.5em,labelwidth=!]
		\item Suppose that $\widehat{M}<M$. We claim that then $M_1=M_2$, for suppose on the contrary that $M_1\neq M_2$, say, $M_1< M_2$. Then $(y^{(0)}_{M_1},\dots, y^{(n-1)}_{M_1})\in\varrho^{\struc{A}_{M_1}}$. Since $M_1=\widehat{M}<M$, it follows that $(x^{(0)}_{M_1},\dots,x^{(n-1)}_{M_1}) = (y^{(0)}_{M_1},\dots,y^{(n-1)}_{M_1})$.  In particular, $(x^{(0)}_{M_1},\dots,x^{(n-1)}_{M_1})\in \varrho^{\struc{A}_{M_1}}$, a contradiction.
		\item Suppose that $M\le\widehat{M}$. Suppose without loss of generality that $M_1\le M_2$. Then
		\[
		|\ul\varrho^{\struc{A}_\infty}(\ttup{x}) - \ul\varrho^{\struc{A}_\infty}(\ttup{y})| = 2^{-M_1} - 2^{-M_2}\le 2^{-M_1} = 2^{-\widehat{M}} \le 2^{-M}.\qedhere
		\]
		\end{enumerate}
	\end{enumerate}
\end{proof}

There is a second natural way to measure how much a given tuple $\ttup{x}$ is \emph{not} contained in $\varrho^{\struc{A}_\infty}$. In particular we might define
\begin{equation*}
	\ol\varrho^{\struc{A}_\infty}(\ttup{x})\coloneqq\delta_{A_\infty}(\ttup{x},\varrho^{\struc{A}_\infty}) = 
	\begin{cases}
		\inf \{\delta_{A_\infty}(\ttup{x},\ttup{y})\mid \ttup{y}\in\varrho^{\struc{A}_\infty}\} & \varrho^{\struc{A}_\infty}\neq\emptyset,\\
		1 & \text{else.}
	\end{cases}
\end{equation*}
Similarly as for $\ul\varrho^{\struc{A}_\infty}$ it can be shown that $\ol\varrho^{\struc{A}_\infty}$ is $1$-Lipschitz, and that $\ol\varrho^{\struc{A}_\infty}(\ttup{x})=0$ if and only if $\ttup{x}\in\varrho^{\struc{A}_\infty}$. However, in general we have $\ol\varrho^{\struc{A}_\infty}(\ttup{x})\ge\ul\varrho^{\struc{A}_\infty}(\ttup{x})$, as the following example demonstrates:
\begin{example}
	Consider the complete graph $K_2$ on two vertices $a$ and $b$, and its complement graph, the empty graph $\ol{K}_2$ as relational structures of the signature that consists of exactly one binary relational symbol $\varrho$. Define $\struc{A}_0\coloneqq K_2$, and $\struc{A}_i\coloneqq\ol{K}_2$, for all $0<i<\omega$. For all $0\le i\le j<\omega$ let $\alpha^j_i\colon\{a,b\}\to\{a,b\}$ be the identity. Then $\cochain{A}=((\struc{A}_i)_{i<\omega},(\alpha^j_i)_{0\le i\le j<\omega})$ is an $\omega$-cochain. Its canonical limit $\struc{A}_\infty$ is the empty graph with vertex set $\{\tup{a},\tup{b}\}$, where $\tup{a}=(a,a,\dots)$ and where $\tup{b}=(b,b,\dots)$.
	
	Note that $\ol\varrho^{\struc{A}_\infty}(\tup{a},\tup{b}) = 1$, while $\ul\varrho^{\struc{A}_\infty}(\tup{a},\tup{b}) = \tfrac{1}{2}$.
\end{example}

\begin{observation}
	Given an $\omega$-cochain $\cochain{A}=((\struc{A}_i)_{i<\omega}, (\alpha_i^j)_{0\le i\le j<\omega})$ of $\Sigma$-structures. Let  $(\alpha_i^\infty)_{i<\omega}$ be its canonical cone, where $\alpha_i^\infty\colon\struc{A}_\infty\to\struc{A}_i$, for each $i<\omega$. Then the following are equivalent:
	\begin{enumerate}
		\item\label{st1} $\forall\varrho\in\Rho\,:\, \ol\varrho^{\struc{A}_\infty}=\ul\varrho^{\struc{A}_\infty}$,
		\item\label{st2} for all $i<\omega$ the homomorphism $\alpha_i^\infty$ is strong, i.e., for all $\varrho\in \Rho$ (say, $\ar(\varrho)=n$), whenever $a^{(0)},\dots,a^{(n-1)}$ are in the image of $\alpha_i^\infty$ and $(a^{(0)},\dots,a^{(n-1)})\in\varrho^{\struc{A}_i}$, then there exist $\tup{x}^{(0)},\dots,\tup{x}^{(n-1)}\in A_\infty$, such that $(a^{(0)},\dots,a^{(n-1)}) = (x^{(0)}_i,\dots,x^{(n-1)}_i)$, and such that  $(\tup{x}^{(0)},\dots,\tup{x}^{(n-1)})\in\varrho^{\struc{A}_\infty}$.
	\end{enumerate}
\end{observation}
\begin{proof}
	``\ref{st1}$\implies$\ref{st2}'': Let $\tup{x}^{(0)},\dots,\tup{x}^{(n-1)}\in A_\infty$, such that $(x^{(0)}_i,\dots,x^{(n-1)}_i)\in\varrho^{\struc{A}_i}$.  Then
	\[
	\ol\varrho^{\struc{A}_\infty}(\tup{x}^{(0)},\dots,\tup{x}^{(n-1)}) = \ul\varrho^{\struc{A}_\infty}(\tup{x}^{(0)},\dots,\tup{x}^{(n-1)})< 2^{-i}.
	\]
	Hence there exists $(\ttup{y}^{(0)},\dots,\ttup{y}^{(n-1)})\in\varrho^{\struc{A}_\infty}$, such that $\delta_{A_\infty}(\tup{x}^{(j)},\tup{y}^{(j)})< 2^{-i}$, for all $0\le j < n$. In particular, $x^{(j)}_i=y^{(j)}_i$, for all $0\le j < n$. This shows that $\alpha^\infty_i$ is strong.
	
	``\ref{st2}$\implies$\ref{st1}'': Let $\tup{x}^{(0)},\dots,\tup{x}^{(n-1)}\in A_\infty$. Suppose that $\ul\varrho^{\struc{A}_\infty}(\tup{x}^{(0)},\dots,\tup{x}^{(n-1)})< 2^{-i}$. Then $(x^{(0)}_i,\dots,x^{(n-1)}_i)\in\varrho^{\struc{A}_i}$. Since $\alpha_i^\infty$ is strong, there exists $(\tup{y}^{(0)},\dots,\tup{y}^{(n-1)})\in\varrho^{\struc{A}_\infty}$, such that $(x^{(0)}_i,\dots,x^{(n-1)}_i)=(y^{(0)}_i,\dots,y^{(n-1)}_i)$. 	It follows that $\delta_{A_\infty}(\ttup{x},\varrho^{\struc{A}_\infty})< 2^{-i}$. Since $i$ may be chosen arbitrarily large, it follows that $\ol\varrho^{\struc{A}_\infty}(\ttup{x})\le\ul\varrho^{\struc{A}_\infty}(\ttup{x})$. Hence equality holds.
\end{proof}
The previous observations show that 
\[
\ul{\mstr{A}}_\infty \coloneqq (A_\infty,\delta_{A_\infty}, (f^{\struc{A}_\infty})_{f\in\Phi}, (\ul\varrho^{\struc{A}_\infty})_{\varrho\in\Rho}) \quad\text{and} \quad \ol{\mstr{A}}_\infty\coloneqq(A_\infty,\delta_{A_\infty}, (f^{\struc{A}_\infty})_{f\in\Phi}, (\ol\varrho^{\struc{A}_\infty})_{\varrho\in\Rho})
\]
are \emph{metric structures} in the sense of \cite{YaaBerHenUsv09}. Let us recall the basic definitions concerning metric structures:
\begin{definition}
	Let $\Sigma=(\Phi,\Rho,\ar)$ be a signature. A \emph{metric $\Sigma$-structure} $\mstr{A}$ is  given by a quadruple $(A,\delta,(f^{\mstr{A}})_{f\in\Phi}, (\varrho^{\mstr{A}})_{\varrho\in\Rho})$, such that
	\begin{enumerate}
		\item $(A,\delta)$ is a bounded complete metric space,
		\item $\forall f\in\Phi\,:\, f^\mstr{A}\colon A^{\ar(f)}\to A$ is uniformly continuous,
		\item $\forall\varrho\in\Rho\,:\, \varrho^\mstr{A}\colon A^{\ar(\varrho)}\to\bR$ is bounded and uniformly continuous.
	\end{enumerate} 
	In case that $\Sigma$ is clear from the context, or if it is of no importance, we are going to skip the ``$\Sigma$-'' from ``metric $\Sigma$-structures''.
\end{definition}
Metric substructures of metric structures are defined in the obvious way:
\begin{definition}
	Let $\mstr{A}$ and $\mstr{B}$ be metric $\Sigma$-structures. Then we say that $\mstr{A}$ is a \emph{metric substructure} of $\mstr{B}$ (designated by $\mstr{A}\le\mstr{B}$) if 
	\begin{enumerate}
		\item $A\subseteq B$,
		\item $\forall f\in\Phi\,\forall \bar{x}\in A^{\ar(f)}\,:\, f^{\mstr{A}}(\bar{x}) = f^{\mstr{B}}(\bar{x})$,
		\item $\forall\varrho\in\Rho\,\forall\bar{x}\in A^{\ar(\varrho)}\,:\,\varrho^{\mstr{A}}(\bar{x}) = \varrho^{\mstr{B}}(\bar{x})$.
	\end{enumerate} 
\end{definition}
The definition of metric embeddings matches the definition of metric substructures:
\begin{definition}
	Let $\mstr{A}$ and $\mstr{B}$ be metric $\Sigma$-structures. A function $h\colon A\to B$ is called a \emph{metric embedding} of $\mstr{A}$ into $\mstr{B}$ if 
	\begin{enumerate}
		\item $h$ is injective,
		\item for all $f\in\Phi$ (say, of arity $n$) the following diagram commutes:
		\[
		\begin{tikzcd}
			A^n \ar[r,"f^{\mstr{A}}"]\ar[d,hook',"h^n"']& A\ar[d,hook',"h"]\\
			B^n \ar[r,"f^{\mstr{B}}"]& B ,
		\end{tikzcd}
		\]
		\item
		For all $\varrho\in\Rho$ (say, of arity $n$) the following diagram commutes:
		\[
		\begin{tikzcd}
			A^n \ar[r,"\varrho^{\mstr{A}}"]\ar[d,hook',"h^n"']& \bR\ar[d,equal] \\
			B^n \ar[r,"\varrho^{\mstr{B}}"]& \bR.
		\end{tikzcd}
		\]
	\end{enumerate}
	Bijective metric embeddings will be called \emph{metric isomorphisms}.
\end{definition}
Figure \ref{constructions}  recapitulates all our constructions so far.
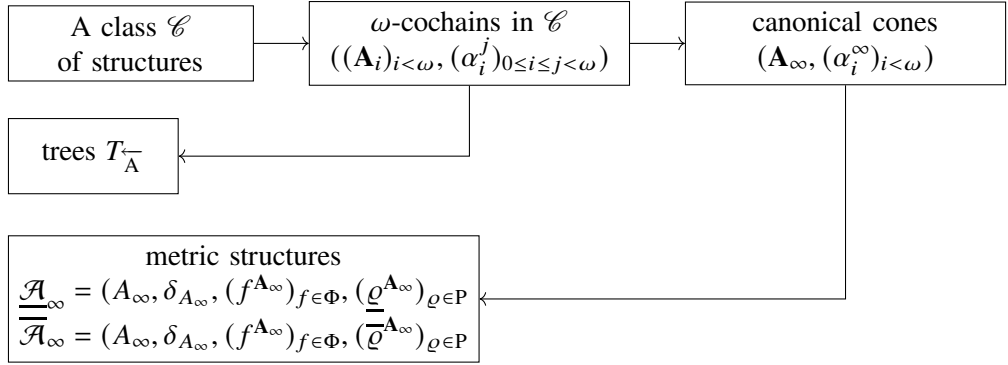
\begin{figure}[htb]
\begin{tikzpicture}
	\coordinate (orig)   at (0,0);
	\coordinate (age)   at (0,0);
	\coordinate (cochains) at (4,0);
	\coordinate (cone) at (9,0);
	\coordinate (met) at (0,-3.4);
	\coordinate (tree) at (0,-1.5);
	
	\node[draw, minimum width=2.53cm, minimum height=1cm, anchor=west, text width=3cm, align=center] (A) at (age) {A class $\class{C}$ of structures};	
	\node[draw, minimum width=4cm, minimum height=1cm, anchor=west, text width=4cm, align=center] (B) at (cochains) {$\omega$-cochains in $\class{C}$\\$((\struc{A}_i)_{i<\omega},(\alpha^j_i)_{0\le i\le j<\omega})$};	
	\node[draw, minimum width=4cm, minimum height=1cm, anchor=west, text width=4cm, align=center] (C) at (cone) {canonical cones\\$(\struc{A}_\infty, (\alpha^\infty_i)_{i<\omega})$};	
	\node[draw, minimum width=6cm, minimum height=1cm, anchor=west, text width=6cm, align=center] (D) at (met) {metric structures\\$\ul{\mstr{A}}_\infty = ( A_\infty,\delta_{A_\infty},(f^{\struc{A}_\infty})_{f\in\Phi},(\ul\varrho^{\struc{A}_\infty})_{\varrho\in\Rho}$\\
	$\ol{\mstr{A}}_\infty = ( A_\infty,\delta_{A_\infty},(f^{\struc{A}_\infty})_{f\in\Phi},(\ol\varrho^{\struc{A}_\infty})_{\varrho\in\Rho}$};
	\node[draw, minimum width=2cm, minimum height=1cm, anchor=west, text width=2cm, align=center] (E) at (tree) {trees $T_{\cochain{A}}$};	
	
	\path[draw,<-] (D.0) -|(C.270);
	\path[draw,<-] (E.0) -|(B.270);
	\path[draw,->] (A.0) -- (B.180);
	\path[draw,->] (B.0) -- (C.180);
\end{tikzpicture}
\caption{From structures to metric structures}\label{constructions}
\end{figure}
For each $\omega$-cochain $((\struc{A}_i)_{i<\omega},(\alpha^j_i)_{i\le j<\omega})$ the metric structure $\ul{\mstr{A}}_\infty$ will be called the \emph{canonical metric structure} of the $\omega$-cochain. Note that the metric $\delta_{A_\infty}$ of $\ul{\mstr{A}}$ and $\ol{\mstr{A}}$ is completely determined by the tree $T_{\cochain{A}}$. 
\begin{definition}
	Let $\class{C}$ be a class of $\Sigma$-structures. Then by $\pi\class{C}$ we denote the class of all those metric structures that are metrically isomorphic to the canonical metric structure of some $\omega$-cochain over $\class{C}$. 
\end{definition}

\subsection{An adjunction between \texorpdfstring{$\omega$}{omega}-cochains and ultrametric structures}
The nature of the construction of the canonical metric structure out of an $\omega$-cochain suggests that it should be functorial in some way.  Our next goal is to make this feeling concrete  by turning $\pi\class{C}$ into a category.  As of writing this paper the literature appears not to contain the definition of a concept of homomorphisms between metric structures. The definition that we are going to give does not pretend to fill this gap. However, for the very special case of metric structures relevant for us, we argue that our definition is the most natural one. Before actually defining homomorphisms, let us narrow down the class of metric structures under consideration:
\begin{definition}
	Let $\mstr{A}=(A,\delta_{\mstr{A}},(f^{\mstr{A}})_{f\in\Phi},(\varrho^{\mstr{A}})_{\varrho\in\Rho})$ be a metric $\Sigma$-structure. Then $\mstr{A}$ is called an \emph{ultrametric $\Sigma$-structure} if
	\begin{enumerate}
		\item $(A,\delta_{\mstr{A}})$ is an ultrametric space of diameter at most $1$,
		\item $f^{\mstr{A}}$ is $1$-Lipschitz, for each $f\in\Phi$,
		\item $\varrho^{\mstr{A}}\colon A^{\ar(\varrho)}\to [0,1]$ is $1$-Lipschitz, for each $\varrho\in\Rho$.
	\end{enumerate}
	If $\mstr{B}=(B,\delta_{\mstr{B}},(f^{\mstr{B}})_{f\in\Phi},(\varrho^{\mstr{B}})_{\varrho\in\Rho})$ is another ultrametric $\Sigma$-structure then a function $h\colon A\to B$ is called a \emph{metric homomorphism} from $\mstr{A}$ to $\mstr{B}$ (formally: $h\colon\mstr{A}\to\mstr{B}$) if
	\begin{enumerate}
		\item $h\colon(A,\delta_{\mstr{A}})\to(B,\delta_{\mstr{B}})$ is $1$-Lipschitz,
		\item for all $f\in\Phi$ (say, of arity $n$) the following diagram commutes:
		\[
		\begin{tikzcd}
			A^n \ar[r,"f^{\mstr{A}}"]\ar[d,hook',"h^n"']& A\ar[d,hook',"h"]\\
			B^n \ar[r,"f^{\mstr{B}}"]& B ,
		\end{tikzcd}
		\]
		\item \label{hom2}for all $\varrho\in\Rho$ (say, of arity $n$) and for all $x_0,\dots,x_{n-1}\in A$ we have:
		\[
		\varrho^{\mstr{A}}(x_0,\dots,x_{n-1})\ge \varrho^{\mstr{B}}(h(x_0),\dots,h(x_{n-1})).
		\]
	\end{enumerate}
	$h$ is called a \emph{metric embedding} if it is isometric and if \eqref{hom2} holds with equality. If $h$ is in addition bijective, then it is called a \emph{metric isomorphism}. The category of all ultrametric $\Sigma$-structures with metric homomorphisms will be denoted by $\cat{U}_\Sigma$. 
\end{definition}
It is important to note that canonical metric structures of $\omega$-cochains over $\class{C}$ are always ultrametric structures. To make the correspondence between $\omega$-cochains and ultrametric structures functorial, we still need to define its action on morphisms. Let $((\struc{A}_i)_{i<\omega}, (\alpha_i^j)_{i\le j<\omega})$ and $((\struc{B}_i)_{i<\omega},(\beta_i^j)_{i\le j<\omega})$ be $\omega$-cochains over $\class{C}$. Let $\struc{A}_\infty$ and $\struc{B}_\infty$ be their respective canonical limits and let $(\alpha^\infty_i)_{i<\omega}$ and $(\beta^\infty_i)_{i<\omega}$ be the corresponding canonical cones. Let 
\[
	(h_i)_{i<\omega} \colon ((\struc{A}_i)_{i<\omega}, (\alpha_i^j)_{i\le j<\omega}) \Longrightarrow((\struc{B}_i)_{i<\omega},(\beta_i^j)_{i\le j<\omega})
	\]
	be a natural transformation. In other words, for all $0\le i\le j<\omega$ we would like the following diagram to be  commutative:
\[
	\begin{tikzcd}
		\struc{A}_i \ar[d,"h_i"']& \struc{A}_j\ar[l,"\alpha^j_i"']\ar[d,"h_j"] \\
		\struc{B}_i & \struc{B}_j\ar[l,"\beta^j_i"'].
	\end{tikzcd}
\]
For each $i<\omega$ let $\tilde{\alpha}^\infty_i\colon\struc{A}_\infty\to\struc{B}_i$ be defined by $\tilde{\alpha}^\infty_i\coloneqq\alpha^\infty_i\circ h_i$.  Then $(\tilde\alpha^\infty_i)_{i<\omega}$ is a compatible cone for $((\struc{B}_i)_{i<\omega},(\beta_i^j)_{i\le j<\omega})$. I.e., for all $0\le i\le j<\omega$ we have that the following diagram is commutative:
\[
    \begin{tikzcd}
        \struc{B}_i & \struc{B}_j\ar[l,"\beta_i^j"] & \struc{A}_\infty.\ar[ll,bend right,"\tilde\alpha^\infty_i"']\ar[l,bend right,"\tilde\alpha^\infty_j"] 
    \end{tikzcd}
\]
Since the canonical cone $(\beta^\infty_i)_{i<\omega}$ is a limiting cone, it follows from the universal property of limits that there is a unique homomorphism $h$ from $\struc{A}_\infty$ to $\struc{B}_\infty$ such that for all $i<\omega$ we have  $\tilde\alpha^\infty_i = \beta^\infty_i\circ h$. 
\begin{observation}
	The above defined homomorphism $h$ acts like
	\[
	h\colon (x_i)_{i<\omega} \mapsto (h_i(x_i))_{i<\omega}.
	\]
\end{observation}
\begin{proof}
	This is a direct consequence of the identity $\alpha_i^\infty\circ h_i = \beta_i^\infty\circ h$.
\end{proof}

It is well-known that the mapping 
\[
	\plim \colon[\cat{\omega}^\op,\cat{S}_\Sigma] \to \cat{S}_\Sigma\quad:\quad ((\struc{A}_i)_{i<\omega},(\alpha^j_i)_{i\le j<\omega})\mapsto\struc{A}_\infty\quad,\qquad (h_i)_{i<\omega}\mapsto h
\]
defines a functor (see \cite[Theorem 8.6]{Kan58}). We are going to use the action of this functor on morphisms in order to define a functor from $[\cat{\omega}^\op,\cat{S}_\Sigma]$ to $\cat{U}_\Sigma$. 
\begin{observations}
	The above defined homomorphism $h\colon\struc{A}_\infty\to\struc{B}_\infty$ is also a metric homomorphism from $\mstr{A}_\infty$ to $\mstr{B}_\infty$. If $(h_i)_{i<\omega}$ is a natural embedding, then $h$ is even a metric embedding. 
\end{observations}
\begin{proof}
	First we show that  $h\colon (\struc{A}_\infty,\delta_{\struc{A}_\infty})\to (\struc{B}_\infty,\delta_{\struc{B}_\infty})$ is $1$-Lipschitz: Let $\tup{x},\tup{y}\in A_\infty$, and let $i<\omega$. Then 
		\begin{equation}\label{hinj}
		x_i = y_i\implies h_i(x_i)=h_i(y_i) \iff \tilde\alpha^\infty_i(\tup{x})=\tilde\alpha^\infty_i(\tup{y}) \iff \beta^\infty_i(h(\tup{x})) = \beta_i^\infty(h(\tup{y})). 
		\end{equation}
		This shows that $\delta_{\struc{A}_\infty}(\tup{x},\tup{y})\ge\delta_{\struc{B}_\infty}(h(\tup{x}),h(\tup{y}))$ and $h$ is indeed $1$-Lipschitz.

	Let $\varrho\in\Rho$ (say, $\ar(\varrho)=n$). Let $(\tup{a}^{(0)},\dots,\tup{a}^{(n-1)})\in (A_\infty)^n$, and let $i<\omega$. Then 
	\begin{align}\label{hmet}
		(a^{(0)}_i,\dots,a^{(n-1)}_i)\in\varrho^{\struc{A}_i} &\implies (h_i(a^{(0)}_i),\dots,h_i(a^{(n-1)}_i))\in\varrho^{\struc{B}_i}\\
		\notag&\iff (\tilde\alpha^\infty_i(\tup{a}^{(0)}),\dots,\tilde\alpha^\infty_i(\tup{a}^{(n-1)}))\in\varrho^{\struc{B}_i}\\
		\notag&\iff (\beta^\infty_i(h(\tup{a}^{(0)})),\dots,\beta^\infty_i(h(\tup{a}^{(n-1)})))\in\varrho^{\struc{B}_i}. 
	\end{align}
	Consequently, $\ul\varrho^{\struc{A}_\infty}(\tup{a}^{(0)},\dots,\tup{a}^{(n-1)})\ge \ul\varrho^{\struc{B}_\infty}(h(\tup{a}^{(0)}),\dots,h(\tup{a}^{(n-1)}))$. It follows that $h$ is a metric homomorphism.
	
	Clearly, if $(h_i)_{i<\omega}$ is a natural embedding, then the implications in \eqref{hinj} and \eqref{hmet} are equivalences. Consequently, in this case $h$ is a metric embedding.
\end{proof}
The previous observation implies that the assignment 
\begin{equation*}
	\mlim\colon ((\struc{A}_i)_{i<\omega},(\alpha_i^j)_{i\le j<\omega})\mapsto\mstr{A}_\infty,\qquad (h_i)_{i<\omega}\mapsto h
\end{equation*}
is a functor from $[\cat{\omega}^\op,\cat{S}_\Sigma]$ to $\cat{U}_\Sigma$. 
\begin{observation}
	There is a natural forgetful functor $\func{U} \colon\cat{U}_\Sigma\to \cat{S}_\Sigma$  that maps each ultrametric $\Sigma$-structure $\mstr{A}$ to its underlying $\Sigma$-structure $\struc{A}$ where $\struc{A}$ shares with $\mstr{A}$ the carrier set and the basic operations, and where the basic relations of $\struc{A}$ are defined by 
	\[
	\varrho^{\struc{A}} = \{\bar{a}\in A^{\ar(\varrho)}\mid \varrho^{\mstr{A}}(\bar{a})=0\}.
	\]
	The functor $\func{U}$ has a left-adjoint $\func{Met}$ that  maps every $\Sigma$-structure $\struc{A}$ to an ultrametric structure $\mstr{A}$ that shares with $\struc{A}$ the carrier set and the basic operations. The ultrametric $\delta_{\mstr{A}}$ of $\mstr{A}$ is the discrete metric given by 
	\[
	\delta_{\mstr{A}}(x,y)= \begin{cases}
		0 & x=y\\
		1 & x\neq y.
	\end{cases}
	\]
	Moreover, for each relational symbol $\varrho\in\Rho$, say, of arity $n$ we have that
	\[
	\varrho^{\mstr{A}}(x_0,\dots,x_{n-1}) = \begin{cases}
		0 & (x_0,\dots,x_{n-1})\in\varrho^{\struc{A}}\\
		1 & (x_0,\dots,x_{n-1})\notin\varrho^{\struc{A}}.
	\end{cases}
	\] 
	It is not hard to see that $\func{Met}$ fully embeds $\cat{S}_\Sigma$ into $\cat{U}_\Sigma$. 
\end{observation}

By now we showed that natural transformations between $\omega$-cochains give rise to metric homomorphisms between their canonical ultrametric structures. A relevant question is whether all metric homomorphisms between the canonical ultrametric structures are of this shape. In order to answer this question (and more) our goal is to define a functor $\Seq\colon\cat{U}_\Sigma\to[\cat\omega^\op,\cat{S}_\Sigma]$ that is a left-adjoint of $\mlim$. 

Given an ultrametric structure $\mstr{A}=(A,\delta_{\mstr{A}},(f^{\mstr{A}})_{f\in\Phi},(\varrho^{\mstr{A}})_{\varrho\in\Rho})$. For each $i<\omega$ define $\approx_i\subset A\times A$ according to 
\[
 x\approx_i y \defeq \delta_{\mstr{A}}(x,y)<2^{-i}.
 \]
 Since $\delta_{\mstr{A}}$ is an ultrametric, we have that all $\approx_i$ are equivalence relations on $A$. Next, for each $i<\omega$ let us define $A_i\coloneqq A/\mathord{\approx}_i$. Our goal is to make $A_i$ the carrier of some $\Sigma$-structure $\struc{A}_i$. To this end, for each $\varrho\in\Rho$, say, of arity $n$  we define $\varrho^{\struc{A}_i}\subseteq A_i^n$ according to 
 \[
 ([a_0]_{\approx_i},\dots,[a_{n-1}]_{\approx_i})\in\varrho^{\struc{A}_i} \defeq \exists b_0\in[a_0]_{\approx_i}\dots \exists b_{n-1}\in[a_{n-1}]_{\approx_i}\,:\, \varrho^{\mstr{A}}(b_0,\dots,b_{n-1})< 2^{-i}.
 \]
Next, for each operational symbol $f\in\Phi$ (say, $\ar(f)=n$) and for all $a_0,\dots,a_{n-1}\in A$ let us define
\[
f^{\struc{A}_i}([a_0]_{\approx_i},\dots,[a_{n-1}]_{\approx_i})\coloneqq[f^{\mstr{A}}(a_0,\dots,a_{n-1})]_{\approx_i}.
\]In order to see that $f^{\struc{A}_i}$ is well-defined let $b_0,\dots,b_{n-1}\in A$, such that $b_0\in[a_0]_{\approx_i},\dots,b_{n-1}\in[a_{n-1}]_{\approx_i}$. Let $\bar{a}\coloneqq(a_0,\dots,a_{n-1})$ and $\bar{b}\coloneqq (b_0,\dots,b_{n-1})$. Then $\delta_{\mstr{A}}(\bar{a},\bar{b})<2^{-i}$. Since $f^{\mstr{A}}$ is $1$-Lipschitz, we also have $\delta_{\mstr{A}}(f^\mstr{A}(\bar{a}),f^\mstr{A}(\bar{b}))< 2^{-i}$. However, this directly implies that $[f^\mstr{A}(\bar{a})]_{\approx_i}=[f^\mstr{A}(\bar{b})]_{\approx_i}$.

Now we define $\struc{A}_i\coloneqq(A/{\approx_i}, (f^{\struc{A}_i})_{f\in\Phi}, (\varrho^{\struc{A}_i})_{\varrho\in\Rho})$. Having done this, for every $0\le i\le j<\omega$ we define $\alpha_i^j\colon A_j\to A_i$ according to $\alpha_i^j\colon [x]_{\approx_j}\mapsto [x]_{\approx_i}$. Since $(\approx_i)\subseteq(\approx_j)$, it follows that $\alpha_i^j$ is well-defined, and we may define:
\[
 	\Seq(\mstr{A})\coloneqq((\struc{A}_i)_{i<\omega}, (\alpha_i^j)_{i\le j <\omega}).
 \]   
\begin{observation}\label{seqepi}
	For all $0\le i\le j <\omega$ we have that $\alpha_i^j\colon\struc{A}_j\to\struc{A}_i$ is a surjective homomorphism. 
\end{observation}
By now we only defined the action of the functor $\Seq$ on objects. Next we define  its action on morphisms.  Let $\mstr{A}$ and   $\mstr{B}$ be metric $\Sigma$-structures and let $h\colon\mstr{A}\to\mstr{B}$ be a metric homomorphism. It is natural to define $h_i\colon A_i\to B_i$ according to $h_i\colon [a]_{\approx_i}\mapsto[h(a)]_{\approx_i}$. 
\begin{observations}\label{seqpresemb}
	 $(h_i)_{i<\omega}\colon\Seq(\mstr{A})\Rightarrow\Seq(\mstr{B})$  is a natural transformation. Moreover, if $h$ is a metric  embedding then $(h_i)_{i<\omega}$ is a natural embedding. 
\end{observations}
\begin{proof}
	First observe that the fact that $h$ is $1$-Lipschitz entails that the $h_i$ are well-defined as functions.
	
	Let $f\in\Phi$, say, of arity $n$. Let $[a_0]_{\approx_i},\dots,[a_{n-1}]_{\approx_i}\in A_i$. Then
	\begin{align*}
	f^{\struc{B}_i}(h_i([a_0]_{\approx_i}),\dots,h_i([a_{n-1}]_{\approx_i}))
		&= f^{\struc{B}_i}([h(a_0)]_{\approx_i},\dots,[h(a_{n-1})]_{\approx_i})\\
		& = [f^{\mstr{B}}(h(a_0),\dots,h(a_{n-1}))]_{\approx_i}\\
		&= [h(f^{\mstr{A}}(a_0,\dots,a_{n-1}))]_{\approx_i}\\
		&= h_i([f^\mstr{A}(a_0,\dots,a_{n-1})]_{\approx_i})\\
		&= h_i(f^{\struc{A}_i}([a_0]_{\approx_i},\dots,[a_{n-1}]_{\approx_i})).
	\end{align*}
	Let now $\varrho\in\Rho$, say, of arity $n$, and let $[a_0]_{\approx_i},\dots,[a_{n-1}]_{\approx_i}\in A_i$. Then
	\[
		([a_0]_{\approx_i},\dots,[a_{n-1}]_{\approx_i})\in \varrho^{\struc{A}_i} \iff \forall j\in\{0,\dots,n-1\}\,\exists b_j\in [a_j]_{\approx_i}\varrho^{\mstr{A}}(b_0,\dots,b_{n-1})< 2^{-i}.
	\]
	Fix such a tuple $(b_0,\dots,b_{n-1})$. Then, by the definition of metric homomorphisms, we have that $\varrho^{\mstr{B}}(h(b_0),\dots,h(b_{n-1}))<2^{-i}$. Consequently, $([h(b_0)]_{\approx_i},\dots,[h(b_{n-1})]_{\approx_i})\in\varrho^{\struc{B}_i}$. However,
	\[
	([h(b_0)]_{\approx_i},\dots,[h(b_{n-1})]_{\approx_i}) = ([h(a_0)]_{\approx_i},\dots,[h(a_{n-1})]_{\approx_i}) = h_i([a_0]_{\approx_i},\dots,[a_{n-1}]_{\approx_i}).
	\]
	Thus $h_i$ is a homomorphism. 
	
	If $h$ is a metric embedding, then in the paragraph above we have additionally that 
	\[\varrho^{\mstr{B}}(h(b_0),\dots,h(b_{n-1}))< 2^{-i} \Rightarrow \varrho^{\mstr{A}}(b_0,\dots,b_{n-1})<2^{-i}.
	\]
	 Together with $h$ being an isometry, this has the consequence that $h_i$ is an embedding.
	
	It remains to check the naturality of $(h_i)_{i<\omega}$. Let $0\le i\le j<\omega$, and let $[a]_{\approx_j}\in A_j$. Then
	\[
	h_i(\alpha^j_i([a]_{\approx_j})) = h_i([a]_{\approx{i}}) = [h(a)]_{\approx_i} = \beta^j_i([h(a)]_{\approx_j}) = \beta^j_i(h_j([a]_{\approx_j})).\qedhere
	\] 
\end{proof}
At this point the functor $\Seq$ is completely defined. The compatibility of $\Seq$ with the composition of morphisms follows directly from the definition of the $h_i$.

Our next goal is to show that the functor $\Seq$ is left-adjoint to $\mlim$. To this end let us define for each $\mstr{A}\in\cat{U}_\Sigma$ define a function $\eta_\mstr{A}\colon \mstr{A}\to \mlim(\Seq(\mstr{A}))$ according to 
\[ 
\eta_{\mstr{A}}\colon x\mapsto ([x]_{\approx_i})_{i<\omega}.
\]
\begin{observation}\label{etabij}
	For each $\mstr{A}\in\cat{U}_\Sigma$ we have that $\eta_{\mstr{A}}$ is a bijective metric homomorphism.
\end{observation}
\begin{proof}
	First we show that $\eta_{\mstr{A}}$ is $1$-Lipschitz: Let $x,y\in A$. Then, by definition
	\[
		\delta(\eta_{\mstr{A}}(x),\eta_{\mstr{A}}(y)) = \delta(([x]_{\approx_i})_{i<\omega}, ([y]_{\approx_i})_{i<\omega}) = 2^{-M},
	\]
	where 
	\[ 
	M= \min\{j<\omega\mid [x]_{\approx_j}\neq[y]_{\approx_j}\}.
	\]
	Note that we have 
	\[ 
	[x]_{\approx_j}\neq [y]_{\approx_j}\iff y\notin[x]_{\approx_j} \iff \delta(x,y)\ge 2^{-j}.
	\]
	As a consequence we obtain that 
	\[
	M=\min\{j<\omega\mid d(x,y)\ge 2^{-j}\}.
	\]
	However, this entails that 
	\[
	\delta(x,y)\ge \delta(([x]_{\approx_i})_{i<\omega}, ([y]_{\approx_i})_{i<\omega}) = \delta(\eta_{\mstr{A}}(x),\eta_{\mstr{A}}(y)).
	\]
	
	Next we show that $\eta_{\mstr{A}}$ is bijective: First note that $\eta_{\mstr{A}}$ is injective because for each $x\in A$ we have
	\[
	\{x\} = \bigcap_{i<\omega} [x]_{\approx_i}.
	\]
	Towards the proof of surjectivity, let $(M_i)_{i<\omega}$ be an element of $\mlim(\Seq(\mstr{A}))$. For each $i<\omega$, choose some $x_i\in M_i$. Then $(x_i)_{i<\omega}$ is a Cauchy-sequence in $(A,\delta_{\mstr{A}})$.  As $(A,\delta_{\mstr{A}})$ is complete, this sequence has a limit $x$. Clearly, $M_i=[x]_{\approx_i}$, for each $i<\omega$. 
	
	It remains to show that $\eta_{\mstr{A}}$ satisfies the compatibility conditions for basic operations and for basic relations. For reasons of brevity, denote $\mlim(\Seq(\mstr{A}))$ by $\mstr{B}$. Let $f\in\Phi$, say of arity $n$. Let $x_0,\dots,x_{n-1}\in A$. Then 
	\begin{align*}
	 \eta_{\mstr{A}}(f^{\mstr{A}}(x_0,\dots,x_{n-1})) &= ([f^\mstr{A}(x_0,\dots,x_{n-1})]_{\approx_i})_{i<\omega} = (f^{\struc{A}_i}([x_0]_{\approx_i},\dots,[x_{n-1}]_{\approx_i}))_{i<\omega} \\
	 &= f^{\mstr{B}}(([x_0]_{\approx_i})_{i<\omega},\dots,([x_{n-1}]_{\approx_i})_{i<\omega}) = f^\mstr{B}(\eta_\mstr{A}(x_0),\dots, \eta_\mstr{A}(x_{n-1})).
	\end{align*}
	
	Let $\varrho\in\Rho$, say, of arity $n$. Let $x_0,\dots,x_{n-1}\in A$.  Then
	\[
	\varrho^{\mstr{B}}(\eta_{\mstr{A}}(x_0),\dots,\eta_\mstr{A}(x_{n-1})) = \varrho^{\mstr{B}}(([x_0]_{\approx_i})_{i<\omega},\dots, ([x_{n-1}]_{\approx_i})_{i<\omega}) = 2^{-M},
	\]
	where 
	\[ 
	M= \min\{j<\omega\mid ([x_0]_{\approx_j}, \dots,[x_{n-1}]_{\approx_j})\notin \varrho^{\struc{A}_j}\}.
	\]
	Note that we have 
	\[
	([x_0]_{\approx_j}, \dots,[x_{n-1}]_{\approx_j})\notin \varrho^{\struc{A}_j}\iff \forall a_0\in[x_0]_{\approx_j}\dots\forall a_{n-1}\in[x_{n-1}]_{\approx_j}\,:\, \varrho^{\mstr{A}}(a_0,\dots,a_{n-1})\ge 2^{-j}.
	\]
	In particular
	\[
	\varrho^{\mstr{A}}(x_0,\dots,x_{n-1})\ge 2^{-M}= \varrho^{\mstr{B}}(\eta_{\mstr{A}}(x_0),\dots,\eta_\mstr{A}(x_{n-1})).
	\]
\end{proof}

Next we consider the family $\eta=(\eta_\mstr{A})_{\mstr{A}\in\cat{U}_\Sigma}$. 
\begin{observation}
	$\eta\colon 1_{\cat{U}_\Sigma}\Rightarrow \mlim\circ\Seq$ is a natural transformation.
\end{observation} 
\begin{proof}
	Let $\mstr{A},\mstr{B}\in\cat{U}_\Sigma$, and let $h\colon \mstr{A}\to\mstr{B}$. We need to show that the following diagram commutes:
	\[
	\begin{tikzcd}
		\mstr{A}\ar[d,"h"']\ar[r,"\eta_{\mstr{A}}"] & \mlim(\Seq(\mstr{A}))\ar[d," \mlim(\Seq(h))"]\\
		\mstr{B}\ar[r,"\eta_\mstr{B}"] & \mlim(\Seq(\mstr{B})).
	\end{tikzcd}
	\]
	For this we need to show that 
	\[
	\mlim(\Seq(h))(([x]_{\approx_i})_{i<\omega})  = ([h(x)]_{\approx_i})_{i<\omega}.
	\]
	Recall that $\Seq(h)=(h_i)_{i<\omega}\colon\Seq(\mstr{A})\to\Seq(\mstr{B})$ is defined through 
	\[
	h_i\colon \struc{A_i}\to \struc{B_i}, \text{where }\quad  [x]_{\approx_i}\mapsto [h(x)]_{\approx_i}.
	\]
	Thus 
	\[\mlim(\Seq(h))(([x]_{\approx_i})_{i<\omega}) = (h_i([x]_{\approx_i}))_{i<\omega}=([h(x)]_{\approx_i})_{i<\omega},
	\]
	as desired.
\end{proof}
Let now $\cochain{A}= ((\struc{A_i})_{i<\omega},(\alpha_i^j)_{i\le j<\omega})$ be an $\omega$-cochain. Let $\mstr{A}$ be its canonical ultrametric structure. Let $\struc{A}_\infty\coloneqq U(\mstr{A})$ be the canonical limit of $\cochain{A}$ and let $(\alpha_i^\infty)_{i<\omega}$ be the canonical limiting cone.  Let $\cochain {B}\coloneqq\Seq(\mstr{A})= ((\struc{B}_i)_{i<\omega}, (\beta_i^j)_{i\le j <\omega})$ (recall that $\struc{B}_i=\struc{A}_\infty/\approx_i$).  We define $\varepsilon_{\cochain{A}}\colon \Seq(\mlim(\cochain{A}))\Rightarrow \cochain{A}$ according to  
\[
\varepsilon_{\cochain{A}} = (\varepsilon_{\cochain{A},i})_{i<\omega}\text{ where } \varepsilon_{\cochain{A},i}\colon \struc{B}_i\to\struc{A}_i\quad [\tup{a}]_{\approx_i}\mapsto\alpha^\infty_i(\tup{a}) = a_i.
\]
\begin{observation}
	Let $\varepsilon\coloneqq(\varepsilon_{\cochain{A}})_{\cochain{A}\in[\cat{\omega}^\op,\cat{S}_\Sigma]}$. Then $\varepsilon\colon \Seq\circ\mlim\Rightarrow \func{1}_{[\cat{\omega}^\op,\cat{S}_\Sigma]}$ is a natural transformation. 
\end{observation}
\begin{proof}
	Let $\cochain{K}=((\struc{K}_i)_{i<\omega},(\kappa_i^j)_{i\le j<\omega})\in[\cat{\omega}^\op,\cat{S}_\Sigma]$ with canonical structure $\mstr{K}$, and let $\cochain{L}=((\struc{L}_i,(\lambda_i^j)_{i\le j<\omega}))\coloneqq\Seq(\mstr{K})$. Let $(\zeta_i)_{i<\omega}\colon \cochain{A}\Rightarrow \cochain{K}$. Let $\zeta\coloneqq\mlim(\zeta_i)_{i<\omega}$ ($\zeta\colon\mlim\cochain{A}\to\mlim\cochain{K}$). We need to show that the following diagram commutes:
	\begin{equation}\label{epsnat}
	\begin{tikzcd}
		\cochain{B} \ar[r,Rightarrow,"\varepsilon_{\cochain{A}}"]\ar[d,Rightarrow,"\Seq(\zeta)"']& \cochain{A}\ar[d,Rightarrow,"(\zeta_i)_{i<\omega}"]\\
		\cochain{L} \ar[r,Rightarrow,"\varepsilon_{\cochain{K}}"]& \cochain{K}
	\end{tikzcd}
	\end{equation}	
	On the one hand we compute that for each $i<\omega$ we have 
	\[
	\zeta_i(\varepsilon_{\cochain{A},i}([\tup{x}]_{\approx_i})) = \zeta_i(x_i).
	\]
	On the other hand  we compute 
	\[
	\Seq(\zeta)_i([\tup{x}]_{\approx_i}) = [\zeta(\tup{x})]_{\approx_i} = [(\zeta_j(x_j))_{j<\omega}]_{\approx_i}.
	\]
	 Thus 
	\[
		\varepsilon_{\cochain{K},i}(\Seq(\zeta)_i([\tup{x}]_{\approx_i}) = \varepsilon_{\cochain{K},i}([(\zeta_j(x_j))_{j<\omega}]_{\approx_i}) = \zeta_i(x_i).
	\]
	Thus diagram \eqref{epsnat} commutes.
\end{proof}

\begin{proposition}
	The functor $\Seq$ is  left adjoint to the functor $\mlim$ with unit $\eta$ and counit $\varepsilon$.
\end{proposition}
\begin{proof}
	We are going to use the characterization of adjunctions from \cite[Theorem IV.1.2(v)]{Mac98}: 
	That is we show that $\eta$ and $\varepsilon$ satisfy the triangle identities:
	\begin{equation}\label{triagid} 
	\begin{tikzcd}[sep=large]
		\mlim \ar[dr,equal]\ar[r,Rightarrow,"\eta\mathop{\ast}\mlim"]& \mlim \circ \Seq\circ\mlim,\ar[d,Rightarrow,"\mlim\mathop{\ast}\varepsilon"]\\
		& \mlim   
	\end{tikzcd}\qquad\qquad
	\begin{tikzcd}[sep=large]
		\Seq \ar[r,Rightarrow,"\Seq\mathop{\ast}\eta"]\ar[dr,equal]& \Seq\circ\mlim\circ\Seq.\ar[d,Rightarrow,"\varepsilon\mathop{\ast}\Seq"]\\
		& \Seq
	\end{tikzcd}
	\end{equation}
	Here
	\begin{align*} 
	\eta\mathop{\ast}\mlim&\colon \mlim\Rightarrow \mlim\circ\Seq\circ\mlim & (\eta\mathop{\ast}\mlim)_{\cochain{A}} &= \eta_{\mlim(\cochain{A})},\\
	\mlim\mathop{\ast}\varepsilon&\colon \mlim\circ\Seq\circ\mlim\Rightarrow\mlim & (\mlim\mathop\ast\varepsilon)_{\cochain{A}} &= \mlim (\varepsilon_{\cochain{A}}),\\
	\Seq\mathop\ast\eta&\colon\Seq\Rightarrow \Seq\circ\mlim\circ\Seq & (\Seq\mathop\ast\eta)_{\mstr{A}} &= \Seq(\eta_{\mstr{A}}),\\
	\varepsilon\mathop\ast\Seq&\colon \Seq\circ\mlim\circ\Seq\Rightarrow\Seq & (\varepsilon\mathop\ast\Seq)_{\mstr{A}} &= \varepsilon_{\Seq(\mstr{A})},
\end{align*}
where	
\begin{align*}
	\eta_{\mlim(\cochain{A})}&\colon \tup{x}\mapsto ([\tup{x}]_{\approx_i})_{i<\omega},\\
	\mlim(\varepsilon_{\cochain{A}})&\colon([\tup{x}]_{\approx_i})_{i<\omega}\mapsto (\varepsilon_{\cochain{A},i}([\tup{x}]_{\approx_i}))_{i<\omega} = \tup{x},\\
	\Seq(\eta_{\mstr{A}})_i&\colon [x]_{\approx{i}}\mapsto [([x]_{\approx_j})_{j<\omega}]_{\approx_i} = \{([y]_{\approx_j})_{j<\omega}\mid [x]_{\approx_i} = [y]_{\approx_i}\},\\
	\varepsilon_{\Seq(\mstr{A}),i} &\colon [([x]_{\approx_j})_{j<\omega}]_{\approx_i}\mapsto [x]_{\approx_i}.
\end{align*}
In order to check that the left hand triangle in \eqref{triagid} commutes, let $\cochain{A}\in[\cat{\omega}^\op,\cat{S}_{\Sigma}]$, and let $\tup{x}\in\mlim(\cochain{A})$. Now we may chase $\tup{x}$ through the mentioned diagram:
\[
\begin{tikzcd}[sep = large]
	\tup{x} \ar[r,mapsto,"\eta_{\mlim(\cochain{A})}"]\ar[dr,equal]& ([\tup{x}]_{\approx_i})_{i<\omega}\ar[d,mapsto,"\mlim(\varepsilon_{\cochain{A}})"] \\
	& \tup{x}.
\end{tikzcd} 
\]
Finally, to observe the commutativity of the right hand triangle of \eqref{triagid}, let $\mstr{A}\in\cat{U}_\Sigma$ and let $x\in A$ and let $i<\omega$ be arbitrary. Again, we may chase $x$ through the right hand triangle of \eqref{triagid}:
\[
\begin{tikzcd}[sep=large]
	[x]_{\approx_i} \ar[r,mapsto,"\Seq(\eta_{\mstr{A}})_i"] \ar[dr,equal]& 
	 {[([x]_{\approx_j})_{j<\omega}]_{\approx_i}}\ar[d,mapsto,"\varepsilon_{\Seq(\mstr{A}),i}"]\\
	 & {[x]_{\approx_i}}.
\end{tikzcd}
\]
This finishes the proof.
\end{proof}

The adjunction $(\Seq,\mlim,\varepsilon,\eta)$ induces an adjoint equivalence between certain full subcategories of $[\cat{\omega}^\op,\cat{S}_\Sigma]$ and $\cat{U}_\Sigma$, respectively. On the side of $\omega$-cochains this subcategory is spanned by all those $\omega$-cochains $\cochain{A}$, for which $\varepsilon_{\cochain{A}}$ is  a natural isomorphism. On the side of ultrametric structures this subcategory is induced by all those $\mstr{A}$ for which $\eta_{\mstr{A}}$ is a metric isomorphism. The following observations make this more precise:
\begin{observations}\label{isoiso}
\begin{enumerate}[nosep]
	\item\label{epsiso} Let $\cochain{A}=((\struc{A}_i)_{i<\omega},(\alpha_i^j)_{i\le j<\omega})\in[\cat{\omega}^\op,\cat{S}_\Sigma]$. Then the natural transformation $\varepsilon_{\cochain{A}}\colon \Seq(\mlim(\cochain{A}))\Rightarrow\cochain{A}$ is a natural embedding. It is a natural isomorphism if and only if for all $i\le j<\omega$ we have that $\alpha_i^j$ is surjective. 
	
	\item\label{etaiso} Let $\mstr{A}\in\cat{U}_\Sigma$. We already saw in \ref{etabij} that $\eta_{\mstr{A}}\colon\mstr{A}\to \mlim(\Seq(\mstr{A}))$ is a bijective metric homomorphism. Now we add to this the observation that $\eta_{\mstr{A}}$ is a metric isomorphism if and only if $\mstr{A}$ is metrically isomorphic to $\mlim(\cochain{A})$ for some $\cochain{A}\in[\cat{\omega}^\op,\cat{S}_\Sigma]$. In other words, $\eta_{\mstr{A}}$ is a metric isomorphism if and only if $\mstr{A}\in\pi\class{S}_\Sigma$.
	\end{enumerate}
\end{observations} 
\begin{proof}
\textbf{about \ref{epsiso}:} Let $\mstr{A}_\infty=\mlim(\cochain{A})$, and let $\Seq(\mstr{A})=((\tilde{\struc{A}}_i)_{i<\omega},(\tilde\alpha_i^j)_{i\le j<\omega})$. Then $\varepsilon_{\cochain{A},i}\colon\tilde{\struc{A}}_i\to\struc{A}_i$, where $\varepsilon_{\cochain{A},i}\colon [\tup{x}]_{\approx_i}\mapsto x_i$. Let $\tup{y}\in A_\infty$, such that $x_i=y_i$. Then $\delta_{A_\infty}(\tup{x},\tup{y})<2^{-i}$. Hence $\tup{x}\approx_i\tup{y}$. This shows that $\varepsilon_{\cochain{A},i}$ is injective. 

Let $\varrho\in \Rho$, say, of arity $n$. Let $\tup{x}^{(0)},\dots,\tup{x}^{(n-1)}\in A_\infty$, such that  $(x_i^{(0)},\dots,x_i^{(n-1)})\in\varrho^{\struc{A}_i}$. Then 
\[
\varrho^{\mstr{A}}(\tup{x}^{(0)},\dots,\tup{x}^{(n-1)}) = \ul\varrho^{\struc{A}_\infty}(\tup{x}^{(0)},\dots,\tup{x}^{(n-1)}) < 2^{-i}.
\]
Hence, by the definition,
\[
([\tup{x}^{(0)}]_{\approx_i},\dots,[\tup{x}^{(n-1)}]_{\approx_i})\in\varrho^{\tilde{\struc{A}}_i}.
\]
Consequently, $\varepsilon_{\cochain{A},i}$ is an embedding.

Note that by \ref{seqepi} we have that $\tilde\alpha_i^j$ is surjective, for each $i\le j<\omega$. Thus, if $\varepsilon_{\cochain{A}}$ is a natural isomorphism, then it follows that $\alpha_i^j$ is surjective, for all $i\le j<\omega$. 

Suppose now that for all $i\le j<\omega$ we have that $\alpha_i^j$ is surjective. We need to show that $\varepsilon_{\cochain{A},i}$ is surjective. So let $x_i\in A_i$. For all $k<i$ we define $x_k\coloneqq\alpha^i_k(x_i)$ and for all $k\ge i$ the $x_k$ are defined inductively: If $x_k$ is defined for some $k\ge i$ then we choose  $x_{k+1}\in A_{k+1}$ in such a way that $\alpha^{k+1}_k(x_{k+1})=x_k$. Define $\tup{x}\coloneqq(x_k)_{k<\omega}$. Then, by construction we have that $\tup{x}\in A_\infty$. Moreover we have that $\varepsilon_{\cochain{A},i}([\tup{x}]_{\approx_i}) = x_i$, as desired. Thus $\varepsilon_{\cochain{A}}$ is indeed a natural isomorphism. 

\textbf{about \ref{etaiso}:} Note that $\mlim(\Seq(\mstr{A}))\in\pi\class{S}_\Sigma$. If $\eta_{\mstr{A}}$ is a metric  isomorphism, then also $\mstr{A}\in\pi\class{S}_\Sigma$. 

So suppose in the following that $\mstr{A}\in\pi\class{S}_\Sigma$. Without loss of generality we may assume that $\mstr{A}=\mstr{A}_\infty=\mlim(\cochain{A})$, for some $\cochain{A}\in[\cat{\omega}^\op,\cat{S}_\Sigma]$. 
	Then $\eta_{\mstr{A}}\colon\tup{x}\mapsto ([\tup{x}]_{\approx_i})_{i<\omega}$. 
	
	Denote 	$\Seq(\mstr{A})=((\tilde{\struc{A}}_i)_{i<\omega},(\tilde\alpha_i^j)_{i\le j<\omega})$.  Let $\widetilde{\mstr{A}}_\infty\coloneqq\mlim(\Seq(\mstr{A}))$. Let $(\tilde\alpha_i^\infty)_{i<\omega}$ be the canonical cone for $\tilde{\struc{A}}_\infty= U(\widetilde{\mstr{A}}_\infty)$.  
	
	First we show that $\eta_{\mstr{A}}\colon (A_\infty,\delta_{\mstr{A}_\infty})\to(\tilde{A}_\infty,\delta_{\widetilde{\mstr{A}}})$ is an isometry: Let $\tup{a},\tup{b}\in A_\infty$ ($\tup{a}\neq\tup{b}$).  Let $i<\omega$. Then 
	\[
		\alpha_i^\infty(\tup{a}) = \alpha_i^\infty(\tup{b}) \iff a_i = b_i \iff [\tup{a}]_{\approx_i} = [\tup{b}]_{\approx_i}
	\]
	and
	\[\tilde\alpha_i^\infty(\eta_{\mstr{A}}(\tup{a})) = \tilde\alpha_i^\infty(\eta_{\mstr{A}}(\tup{b})) \iff  \tilde\alpha^\infty_i(([\tup{a}]_{\approx_j})_{j<\omega}) = \tilde\alpha^\infty_i(([\tup{b}]_{\approx_j})_{j<\omega})\iff [\tup{a}]_{\approx_i} = [\tup{b}]_{\approx_i}.
	\]
	Thus, $\delta_{\mstr{A}_\infty}(\tup{a},\tup{b}) = \delta_{\widetilde{\mstr{A}}_\infty}(\eta_{\mstr{A}}(\tup{a}),\eta_{\mstr{A}}(\tup{b}))$. 

	Similarly, for $\varrho\in\Rho$ of arity $n$, for $\tup{x}^{(0)},\dots,\tup{x}^{(n-1)}\in A_\infty$, and for $i<\omega$
	we argue
	\begin{multline*}
	(\alpha^\infty_i(\tup{x}^{(0)}),\dots,\alpha^\infty_i(\tup{x}^{(n-1)}))\in\varrho^{\struc{A}_i}\iff (x_i^{(0)},\dots,x_i^{(n-1)})\in\varrho^{\struc{A}_i}\\\iff \ul\varrho^{\struc{A}_\infty}(\tup{x}^{(0)},\dots,\tup{x}^{(n-1)})< 2^{-i}.
	\end{multline*}
	On the other hand we have 
	\begin{multline*}
	\varrho^{\widetilde{\mstr{A}}_\infty}(\eta_{\mstr{A}}(\tup{x}^{(0)}),\dots,\eta_{\mstr{A}}(\tup{x}^{(n-1)}))<2^{-i}   \iff \varrho^{\widetilde{\mstr{A}}_\infty} (([\tup{x}^{(0)}]_{\approx_j})_{j<\omega},\dots, ([\tup{x}^{(n-1)}]_{\approx_j})_{j<\omega}) < 2^{-i} \\\iff ([\tup{x}^{(0)}]_{\approx_i},\dots,[\tup{x}^{(n-1)}]_{\approx_i})\in\varrho^{\tilde{\struc{A}}_i}\\
	\iff \exists\tup{y}^{(0)}\in[\tup{x}^{(0)}]_{\approx_i},\dots,\tup{y}^{(n-1)}\in[\tup{x}^{(n-1)}]_{\approx_i}\,:\,\varrho^{\mstr{A}_\infty}(\tup{y}^{(0)},\dots,\tup{y}^{(n-1)})<2^{-i}\\
	\iff \exists\tup{y}^{(0)}\in[\tup{x}^{(0)}]_{\approx_i},\dots,\tup{y}^{(n-1)}\in[\tup{x}^{(n-1)}]_{\approx_i}\,:\,(y^{(0)}_i,\dots,y^{(n-1)}_i)\in\varrho^{\struc{A}_i}
	\end{multline*}
	Since for all $\tup{y}^{(0)}\in[\tup{x}^{(0)}]_{\approx_i},\dots,\tup{y}^{(n-1)}\in[\tup{x}^{(n-1)}]_{\approx_i}$ we have that $(y^{(0)}_i,\dots,y^{(n-1)}_i)$ is equal to $(x^{(0)}_i,\dots,x^{(n-1)}_i)$, the two chains of equivalences result in  
	\[
	\varrho^{\widetilde{\mstr{A}}_\infty}(\eta_{\mstr{A}}(\tup{x}^{(0)}),\dots,\eta_{\mstr{A}}(\tup{x}^{(n-1)})) = \ul\varrho^{\struc{A}_\infty}(\tup{x}^{(0)},\dots,\tup{x}^{(n-1)})\quad(=\varrho^{\mstr{A}_\infty}(\tup{x}^{(0)},\dots,\tup{x}^{(n-1)})).\qedhere
	\]
\end{proof}
In the following, for every full subcategory $\cat{C}$ of $\cat{S}_\Sigma$, by $[[\cat{\omega}^\op,\cat{C}]]$ we denote the full subcategory of $[\cat{\omega}^\op,\cat{C}]$ that is induced by all such $\omega$-cochains $\cochain{A}=((\struc{A}_i)_{i<\omega},(\alpha_i^j)_{i\le j<\omega})$ for which all $\struc{A}_i$ are from $\cat{C}$ and all $\alpha_i^j$ are surjective (where $i\le j<\omega$). Moreover, by $\cat{\pi C}$ we will denote the full subcategory of $\cat{U}_\Sigma$ that is induced by $\pi\class{C}$. Observation~\ref{isoiso} implies that $[[\cat{\omega}^\op,\cat{S}_\Sigma]]$ and $\cat{\pi S}_\Sigma$ are  equivalent categories (the adjoint equivalence is given by the restrictions of $\mlim$ and $\Seq$ to $[[\cat{\omega}^\op,\cat{S}_\Sigma]]$ and $\cat{\pi S}_\Sigma$). 

\subsection{The case of hereditary classes}
If we restrict our attention to such full subcategories $\cat{C}$ of $\cat{S}_\Sigma$ for which $\class{C}$ is a hereditary class (i.e., it consists of finitely generated $\Sigma$-structures and it has the \HP), then the mathematics of the operators $\sigma$ and $\pi$ becomes smoother. We already noted that for a hereditary class $\class{C}\subseteq\class{S}_\Sigma$ we have that $\sigma\class{C}$ consists of all countably generated $\Sigma$-structures whose age is contained in $\class{C}$. Concerning the operator $\pi$, in this case we can get more out of the  adjoint equivalence constructed in the previous section:
\begin{proposition}\label{CEquiv}
	Let $\class{C}\subseteq\class{S}_\Sigma$, such that $\class{C}$ has the hereditary property in $\cat{S}_\Sigma$. Then $[[\cat{\omega}^\op,\cat{C}]]$ and $\cat{\pi C}$ are equivalent categories. An adjoint equivalence is given by the appropriate restrictions of $\mlim$ and $\Seq$. 
\end{proposition}
\begin{proof}
	We only need to show that the image of $\mlim\!\restr_{\cat{C}}$ lies in $\cat{\pi C}$, and that the image of $\Seq\!\restr_{\cat{\pi C}}$ lies within $[[\cat{\omega}^\op,\cat{C}]]$. The former holds trivially. So let $\mstr{A}\in\pi\class{C}$. Without loss of generality, $\mstr{A}=\mlim(\cochain{A})$, for some $\cochain{A}\in[\cat{\omega}^\op,\cat{C}]$. By Observation~\ref{isoiso}(\ref{epsiso}), we have that $\varepsilon_{\cochain{A}}\colon\Seq(\mlim(\cochain{A}))\Rightarrow\cochain{A}$ is a natural embedding. Since $\class{C}$ has the \HP in $\cat{S}_\Sigma$, it follows with Observation~\ref{seqepi} that $\Seq(\mlim(\cochain{A}))=\Seq(\mstr{A})$ is in $[[\cat{\omega}^\op,\cat{C}]]$. 
\end{proof}
So far we considered the hereditary property only for classes of model-theoretic structures. It is natural to extend this concept to metric structures and to examine under which conditions a class $\pi\class{C}$ has the generalized hereditary property. 
\begin{definition}
	Let $\class{D}\subseteq \class{M}$ be a classes of metric $\Sigma$-structures. We say that $\class{D}$ has the hereditary property (\HP) in $\class{M}$ if whenever $\mstr{B}\in\class{D}$ and $\iota\colon\mstr{A}\injto\mstr{B}$ is a metric embedding of $\mstr{A}\in\class{M}$ into $\mstr{B}$, then we have that $\mstr{A}$ is also in $\class{D}$.   
\end{definition}

\begin{observation}\label{hereditarypi}
	Let $\class{C}\subseteq\class{S}_\Sigma$ have the  hereditary property in $\class{S}_\Sigma$. Then  $\pi\class{C}$ has the hereditary property in $\class{U}_\Sigma$. 
\end{observation}
\begin{proof}
	  Let $\mstr{B}\in\pi\class{C}$, let $\mstr{A}\in\class{U}_\Sigma$, and let $\iota\colon\mstr{A}\injto\mstr{B}$ be a metric embedding. Since $\class{C}$ has the \HP in $\class{S}_\Sigma$, it follows from Proposition~\ref{CEquiv} that $\Seq(\mstr{B})\in[[\cat{\omega}^\op,\cat{C}]]$. By Observation~\ref{seqpresemb}, $\Seq(\iota)\colon\Seq(\mstr{A})\to\Seq(\mstr{B})$ is a natural embedding. Again, since $\class{C}$ has the \HP in $\class{S}_\Sigma$, it follows that $\Seq(\mstr{A})\in[[\cat{\omega}^\op,\cat{C}]]$. Consequently, $\mlim(\Seq(\mstr{A}))\in\pi\class{C}$. 
	
	Since $\eta$ is a natural transformation, the following diagram commutes:
	\[
	\begin{tikzcd}[sep=large]
		\mstr{A} \ar[r,hook,"\iota"]\ar[d,"\eta_{\mstr{A}}"']& \mstr{B}\ar[d,"\eta_{\mstr{B}}"] \\
		\mlim(\Seq(\mstr{A})) \ar[r,"\mlim(\Seq(\iota))"]& \mlim(\Seq(\mstr{B}))
	\end{tikzcd}  
	\]
	By Observation~\ref{isoiso}(\ref{etaiso}), $\eta_{\mstr{B}}$ is a metric isomorphism. It follows that $\mlim(\Seq(\iota))\circ\eta_{\mstr{A}}$ is a metric embedding, too. Since both, $\eta_{\mstr{A}}$ and $\mlim(\Seq(\iota))$ are $1$-Lipschitz and since $\eta_{\mstr{A}}$ is bijective (see Observation~\ref{etabij}), it follows that both are actually isometries. Analogously, since $\eta_{\mstr{B}}\circ\iota$ preserves all basic relations, so do $\eta_{\mstr{A}}$ and $\mlim(\Seq(\iota))$. It follows that $\eta_A$ and $\mlim(\Seq(\iota))$ are both metric embeddings. Hence $\eta_{\mstr{A}}$ is a metric isomorphism.  This implies that $\mstr{A}\in\pi\class{C}$. Thus, $\pi\class{C}$ has the \HP in $\class{U}_\Sigma$. 
\end{proof}

\section{Ultrametric homogeneity and universality}

This section contains our main results, starting from a discussion of smallness in the new setting. The key ingredient for obtaining ultrametric \Fraisse limits is the existence of a universal skew-homogeneous homomorphism.

\subsection{Notions of smallness for ultrametric structures}
In model theory, a structure is considered to be \emph{small} if the cardinality of its carrier set is bounded from above by some (usually unspecified) cardinal. In \Fraisse theory it is customary to call a relational structure small if it is finite and, more generally, to call a structure small if it is finitely generated. In the case of metric structures, given that they are equipped with a metric, we see  at least four possible notions of smallness. A metric structure  $\mstr{A}$ might be called \emph{small} if one of the following conditions holds:
\begin{enumerate}
	\item $\mstr{A}$ is finite,
	\item $\mstr{A}$ is finitely generated,
	\item $\mstr{A}$ is compact,
	\item $\mstr{A}$ is compactly generated.
\end{enumerate}
Here   compactness is to be understood with respect to the topology induced on $A$ by the metric of $\mstr{A}$.
\begin{remark}
 Recall that a subset $A$ of a metric space $(M,\delta)$  is called \emph{totally bounded} if for every $\varepsilon> 0$ we have that $A$ can be covered by finitely many open $\varepsilon$-balls. $A$ is called \emph{Cauchy-precompact} if every sequence in $A$ admits a Cauchy subsequence. It is well-known that $A$ is totally bounded if and only if it is Cauchy-precompact. Moreover, $A$ is compact if and only if it is totally bounded and Cauchy-complete. 
\end{remark}
\begin{observation}\label{precomp}
	Let $\mstr{A}\in\class{U}_\Sigma$ be an ultrametric structure. A subset $M\subseteq A$ is Cauchy-precompact if and only if for each $i<\omega$ we have that $M/\approx_i$ is finite. 
\end{observation}
\begin{proof}
	``$\Rightarrow$'': Suppose that for some $i<\omega$ the set $M/\approx_i$ is infinite. Let $([a_j]_{\approx_i})_{j<\omega}$ be a sequence of distinct elements of $M/\approx_i$. Then for all $j_1< j_2<\omega$ we have that $\delta_{\mstr{A}}(a_{j_1},a_{j_2})\ge 2^{-i}$. Thus, the sequence $(a_j)_{j<\omega}$ in $M$ has no Cauchy-subsequence. In other words, $M$ is not Cauchy-precompact. 
	
	``$\Leftarrow$'': Let $(a_j)_{j<\omega}$ be a sequence in $M$. For each $i<\omega$ let $L_i\coloneqq\{[a_j]_{\approx_i}\mid j<\omega\}$, and $N_i\coloneqq\{[a]_{\approx_i}\in L_i\mid\{j<\omega\mid [a_j]_{\approx_i}=[a]_{\approx_i}\}\text{ is infinite}\}$. Since $M/\approx_i$ is finite, $N_i$ is non-empty. Let $N\coloneqq\bigcup_{i<\omega} N_i$.  Then $(N,\supseteq)$ is a tree. Since all $N_i$ are finite and non-empty, by \Konig's tree-lemma, $N$ has a branch $([b_i]_{\approx_i})_{i<\omega}$. Next, by induction,  we define a sequence $j_0<j_1<\dots<\omega$, such that for all $i<\omega$ we have that $[a_{j_i}]_{\approx_i}=[b_i]_{\approx_i}$. We start by defining $j_0\coloneqq 0$. Note that $[b_0]_{\approx_0}= [a_0]_{\approx_0}=M$.  Suppose that $j_0,\dots,j_n$ are already fixed. Take $j_{n+1}\in J_{n+1}\coloneqq\{k<\omega\mid [a_k]_{\approx_{n+1}}= [b_{n+1}]_{\approx_{n+1}}\}$, such that $j_{n+1}> j_n$ (this is possible, since $J_{n+1}$ is infinite). We claim that the sequence $(a_{j_i})_{i<\omega}$ is  Cauchy. Indeed, if $i< k<\omega$, then $[a_{j_i}]_{\approx_i} = [b_i]_{\approx_i}$, and $[a_{j_k}]_{\approx_k} = [b_k]_{\approx_k}$. Since $[b_k]_{\approx_k}\subseteq [b_i]_{\approx_i}$, we have $[a_{j_k}]_{\approx_i} = [a_{j_i}]_{\approx_i}$ in other words $\delta_{A_\infty}(a_{j_i},a_{j_k})<2^{-i}$. This proves the claim. Consequently, $M$ is Cauchy-precompact.
\end{proof}

An immediate consequence is:
\begin{corollary}
	Let $\mstr{A}\in\class{U}_\Sigma$ be an ultrametric structure with $\Seq(\mstr{A})=((\struc{A}_i)_{i<\omega},(\alpha_i^j)_{i\le j<\omega})$. Then $\mstr{A}$ is compact if and only if for all $i<\omega$ the structure $\struc{A}_i$ is finite.	
\end{corollary}
\begin{proof}
	Clear.
\end{proof}

For compactly generated ultrametric structures a slightly weaker observation may be formulated:
\begin{corollary}
	Let $\mstr{A}\in\class{U}_{\Sigma}$ be compactly generated with $\Seq(\mstr{A})=((\struc{A}_i)_{i<\omega},(\alpha_i^j)_{i\le j<\omega})$. Then  for all $i<\omega$ the structure $\struc{A}_i$ is finitely generated.
\end{corollary}
\begin{proof}
	Let $M\subseteq A$ be a compact generating system of $\mstr{A}$.  By Observation~\ref{precomp} we have that $M/{\approx_i}$ is finite for each $i<\omega$. Let $(\alpha_i^\infty)_{i<\omega}$ be the canonical cone for $\Seq(\mstr{A})$. Since $\alpha_i^\infty\circ U(\eta_{\mstr{A}})\colon U(\mstr{A})\to \struc{A}_i$ is a surjective homomorphism (cf.~\ref{etabij}), it follows that $M/\approx_i$ is a generating set of $\struc{A}_i$, for each $i<\omega$.
\end{proof}

Note that in general it is not true that $\omega$-cochains of finitely generated structures have a compactly generated canonical ultrametric structure, as the following example shows:
\begin{example}
	Consider the signature of monoids. It is given by $\Sigma=(\Phi,\Rho,\ar)$, where $\Rho=\emptyset$ and where $\Phi=\{\cdot,1\}$ with $\ar(\cdot)=2$ and $\ar(1)=0$. Our goal is to define an $\omega$-cochain $\cochain{A}=((\struc{A}_i)_{i<\omega}, (\alpha_i^j)_{i\le j<\omega})$ of finitely generated monoids in such a way that $\mlim\cochain{A}$ is not compactly generated.
	
	Let $X=\{x_i\mid i<\omega\}$ be a set of distinct letters. For each $i<\omega$, let $X_i\coloneqq\{x_0,\dots,x_{i-1}\}$. Further let $\struc{A}_i\coloneqq X_i^\ast$  be the free monoid freely generated by $X_i$ (in particuler, $X_i^\ast$ consists of all finite words over $X_i$, the concatenation of words is the monoid-multiplication, and the empty word $\varepsilon$ is the neutral element). Note that $\struc{A}_0=\{\varepsilon\}$. Define $\alpha^1_0$ to be the unique homomorphism that maps $x_0$  to $\varepsilon$. Similarly, for $i>0$ let  $\alpha^{i+1}_i$ be the unique monoid homomorphism that maps $x_j$ to itself for all $j < i$ and that maps $x_i$ to $x_0^i$ (the existence and uniqueness of all these homomorphisms is guaranteed by the fact that $\struc{A}_{i+1}$ is freely generated by $X_{i+1}$). Thus we have defined the $\omega$-cochain $\cochain{A}$:
	\[
	\begin{tikzcd}
		\struc{A}_0 & \struc{A}_1\ar[l,two heads,"\alpha^1_0"'] & \struc{A}_2\ar[l,two heads,"\alpha^2_1"'] & \struc{A}_3\ar[l,two heads,"\alpha^3_2"'] & \dots\ar[l,two heads,"\alpha^4_3"']
	\end{tikzcd}
	\] 
	Let $\mstr{A}_\infty\coloneqq\mlim\cochain{A}$. Let us show that $\mstr{A}_\infty$ is not compactly generated. A first important observation is that for each $i<\omega$ we have that every generating set of $\struc{A}_i$ has to contain $X_i$ as a subset. Let $M\subseteq A_\infty$ be any generating set of $\mstr{A}_\infty$. Then for every $i<\omega$ there exists $\tup{a}^{(i)}\in A_\infty$, such that $a^{(i)}_{i+1} = x_i$. But then $a^{(i)}_1 = x_1^i$. A consequence is that $\alpha^\infty_1(M)\supseteq\{x_1^i \mid 0< i<\omega\}$. Consequently, by Observation~\ref{precomp}, $M$ is not Cauchy-precompact.	
\end{example} 

The previous example suggests a  weaker variant of smallness:
\begin{definition}
	 An ultrametric structure $\mstr{A}$ is called \emph{pro-finitely generated} if in the image of $\Seq(\mstr{A})$ each structure is finitely generated.
\end{definition}
Clearly, if an ultrametric structure is compactly generated, then it is also pro-finitely generated. As the example above shows, the opposite is not true. However, in case of pure relational signatures, there is no difference between compactness, compact generatedness and pro-finite generatedness. Our notion of choice for smallness among ultrametric structures is the notion of pro-finite generatedness.

\subsection{Universal homogeneous metric structures through projective limits}
In this  section we start to explore the combined power of the operators $\sigma$ and $\pi$. Motivated by descriptive set theory we work towards the creation of something that could be called descriptive model theory or, less ambitiously, descriptive \Fraisse theory. The first steps in this direction are exploratory. We start with a \Fraisse class $\class{C}$ and consider the class $\pi\sigma\class{C}$. In particular we ask: Does $\pi\sigma\class{C}$ contain interesting structures? Here ``interesting'' means, e.g., to be highly symmetric and/or to be universal for $\pi\sigma\class{C}$. 
\begin{definition}
	Let $\class{M}\subseteq\class{U}_\Sigma$ be a hereditary class of ultrametric structures. A structure $\mstr{U}\in\class{M}$ is called \emph{$\class{M}$-universal} if every ultrametric structure from $\class{M}$ metrically embeds into $\mstr{U}$. 
	
	 A metric structure $\mstr{H}\in\class{U}_\Sigma$ is called \emph{homogeneous} if every metric isomorphism between pro-finitely generated  substructures of $\mstr{H}$ extends to a metric automorphism of $\mstr{H}$. 
\end{definition}
\begin{remark}
	Universal homogeneous metric structures have been studied by Ben~Yaacov in \cite{Yaa15}, where small structures  are finitely generated metric structures. \Fraisse limits are approximately homogeneous in the sense that finite partial metric isomorphisms extend to metric automorphism up to an arbitrarily small error.
	
	Our approach differs from the one of \cite{Yaa15} in that we are interested in ``sharply'' homogeneous structures (extensions are guaranteed without a margin for errors). Also, we use a different concept of smallness. While in \cite{Yaa15} small metric structures are finite, our small structures are much larger (pro-finitely generated).
	\end{remark}
	
In order to be able to formulate our result concerning homogeneous ultrametric structures we still need to introduce some terms.
\begin{definition}
	Let $\class{C}$ be a class of structures of the same type. We say that $\class{C}$ has the \emph{amalgamated extension property} (\AEP) if for all $\struc{A},\struc{B}_1,\struc{B}_2,\struc{T}\in\class{C}$, $f_1\colon \struc{A}\injto\struc{B}_1$, $f_2\colon \struc{A}\injto\struc{B}_2$, $h_1\colon \struc{B}_1\to\struc{T}$, $h_2\colon \struc{B}_2\to\struc{T}$, if $h_1\circ f_1=h_2\circ f_2$, then there exist $\struc{C},\struc{T}'\in\class{C}$, $g_1\colon \struc{B}_1\injto\struc{C}$, $g_2\colon \struc{B}_2\injto\struc{C}$, $h\colon \struc{C}\to\struc{T}'$, $k\colon \struc{T}\injto\struc{T}'$ such that the following diagram commutes:
	\[
	\begin{tikzcd}
		& & \struc{T}\arrow[r,dashed,hook,"k"] & \struc{T}'.\\
		\struc{B}_1\arrow[r,dashed,hook,"g_1"]\arrow[urr,bend left,"h_1"] & \struc{C}\arrow[rru,dashed,bend right,near end,"h"']\\
		\struc{A}\arrow[u,hook,"f_1"]\arrow[r,hook,"f_2"] & \struc{B}_2\arrow[u,dashed,hook,"g_2"]\arrow[uur,bend right,"h_2"']
	\end{tikzcd}
	\]
\end{definition}

Now we are ready to formulate our main results:
\begin{theorem}\label{mainthm}
	Let $\class{C}\subseteq\class{S}_\Sigma$ be an age.  Then $\class{C}$ has the \AP and the \AEP if and only if $\pi\sigma\class{C}$ contains a universal and homogeneous ultrametric structure.
	 Moreover, any two universal and  homogeneous structures in $\pi\sigma\class{C}$ are metrically isomorphic.
	 \end{theorem}

In a couple of interesting cases we are able to show that the universal homogeneous ultrametric structures are even more symmetric. This comes with the price of stronger assumptions on the signature and on the age $\class{C}$:
\begin{definition}
	Let $\class{C}$ be a class of structures of the same type. We say that $\class{C}$ has the \HAP if for all $\struc{A}$, $\struc{B}_1$, $\struc{B}_2$ from $\class{C}$, for all  $f_1\colon \struc{A}\injto\struc{B}_1$, and for all  $f_2\colon \struc{A}\to\struc{B}_2$ there exist $\struc{C}\in\class{C}$,  $g_1\colon \struc{B}_1\to\struc{C}$, and $g_2\colon \struc{B}_2\injto\struc{C}$ such that the following diagram commutes:
	\[
	\begin{tikzcd}
		\struc{B}_1 \arrow[r,dashed,"g_1"]& \struc{C} \\
		\struc{A}\arrow[u,hook',"f_1"]\arrow[r,"f_2"] & \struc{B}_2.\arrow[u,hook',dashed,"g_2"]
	\end{tikzcd}
	\]
\end{definition} 
\begin{remark}
	The acronym \HAP is usually translated as \emph{homo amalgamation property} stressing, that two morphisms in the commuting square are merely homomorphisms instead of embeddings. Given that in the definition of the \HAP diverse sorts of morphisms are used, another legitimate translation would be \emph{hetero amalgamation property}.   We should mention that the \HAP has been around in mathematical literature for quite a long time. We could trace it back to the paper \cite{Ban70} by Banaschewski, where it is called the \emph{transferability property}. It appeared as the $\operatorname{1PHEP}$ in  \cite{Dol11}, and as the \emph{mixed amalgamation property} in \cite{Kub15}.
	Be it as it may, in this paper we will stick with the by now standard  acronym \HAP and leave it to the reader to decide what the letter H is standing for.
\end{remark}

	\begin{theorem}\label{secmainthm}
		Let $\Sigma$ be a finite and purely relational signature, and let $\class{C}\subseteq\class{S}_\Sigma$  be an age with the \AP, the \AEP, and the \HAP. Let $\mstr{U}$ be the unique universal homogeneous metric structure in $\pi\sigma\class{C}$ postulated by Theorem~\ref{mainthm}. Then every isomorphism between finite substructures of its underlying structure $\struc{U}$ extends to an automorphism of\/ $\struc{U}$. Moreover, this extension may be chosen to be an isometry with respect to the metric of $\mstr{U}$ on every ball of radius $< 2^{-l}$ for some $l>0$.
	\end{theorem}

At this point a formally correct and complete proof of these claims would be overwhelmingly technical and  cumbersome. 
We chose to proceed in a couple of smaller,  manageable steps. Let us start by giving a sketch of the general idea.  We start from a hereditary class $\class{C}$. Clearly, then also $\sigma\class{C}$ is hereditary.   By Proposition~\ref{CEquiv} the appropriate restrictions of the functors $\Seq$ and $\mlim$ induce a categorical equivalence between the category $[[\cat{\omega}^\op,\cat{\sigma C}]]$ of surjective $\omega$-cochains over $\sigma\class{C}$ and $\cat{\pi \sigma C}$. So instead of constructing a universal  homogeneous ultrametric structure $\mstr{U}$ in $\pi\sigma\class{C}$, we may construct an appropriate surjective $\omega$-cochain $\cochain{U}$ over $\sigma\class{C}$, such that $\mlim\cochain{U}\cong\mstr{U}$.  

In order to obtain a universal ultrametric structure $\mstr{U}\in\pi\sigma\class{C}$ we need to construct an $\omega$-cochain $\cochain{U}$ that is universal for $[[\cat{\omega}^\op,\cat{\sigma C}]]$.  To this end, we use the notion of universal homomorphisms:
\begin{definition}
Let $\class{C}$ be an age, and let $\struc{U},\struc{V}\in\sigma\class{C}$. A homomorphism $\Omega\colon\struc{U}\to\struc{V}$ is called \emph{universal in $\cat{\sigma{C}}$}  if for all $\struc{A}\in\sigma\class{C}$, and for all $h\colon\struc{A}\to\struc{V}$ there exists an embedding $\iota\colon\struc{A}\injto\struc{U}$, such that the following diagram commutes:
\[
\begin{tikzcd}[row sep = small]
	\struc{U} \arrow[rd,"\Omega"] \\
	& \struc{V}.\\
	\struc{A}\arrow[uu,hook',dashed,"\iota"]\arrow[ru,"h"] 
\end{tikzcd}
\]
\end{definition}
A first simple but very useful observation about universal homomorphisms is:
\begin{observation}
Every universal homomorphism in $\cat{\sigma\class{C}}$ is a retraction.	
\end{observation}
\begin{proof}
Let $\Omega\colon\struc{U}\to\struc{V}$ be universal in $\cat{\sigma\class{C}}$.
 Then there exists an embedding $\iota$, such that the following diagram commutes: 
\[
\begin{tikzcd}[row sep = small]
	\struc{U} \arrow[rd,"\Omega"]\\
	&\struc{V}\\
	\struc{V} \arrow[ru,"1_{\struc{V}}"']\arrow[uu,hook',dashed,"\iota"] 
\end{tikzcd}
\]
In other words, $1_{\struc{V}}=\Omega\circ\iota$. That means that $\iota$ is a section and that $\Omega$ is a retraction.	
\end{proof}

Let us postpone the question about the existence of universal homomorphisms for later. For now let us see, how they may be used for the construction of universal ultrametric structures:
\begin{observation}\label{sequniv}
	Let $\cochain{U}=((\struc{U}_i)_{i<\omega},(\Omega_i^j)_{i\le j<\omega})$ be an $\omega$-cochain over $\sigma\class{C}$, such that $\struc{U}_0$ is universal for $\sigma\class{C}$ and such that for all $i<\omega$ we have that $\Omega_i^{i+1}$ is universal in $\cat{\sigma C}$. Let $\mstr{U}\coloneqq\mlim\cochain{U}$. Then $\mstr{U}$ is universal in $\pi\sigma\class{C}$.  
\end{observation}
\begin{proof}
	Let $\mstr{X}\in\pi\sigma\class{C}$. 	Let $\cochain{X}\coloneqq\Seq(\mstr{X})= ((\struc{X}_i)_{i<\omega}, (\chi_i^j)_{i\le j<\omega})\in[[\cat{\omega}^\op,\cat{\sigma C}]]$. Let $\iota_0\colon\struc{X}_0\injto\struc{U}_0$ be an arbitrary embedding (this exists because $\struc{U}_0$ is universal in $\sigma\class{C}$). Suppose that $\iota_i$ has already been defined. Consider $\iota_i\circ\chi_i^{i+1}\colon\struc{X}_{i+1}\to\struc{U}_i$.  Since $\Omega_i^{i+1}$ is universal, there exists $\iota_{i+1}\colon \struc{X}_{i+1}\injto\struc{U}_{i+1}$ such that the following diagram commutes:
	\[
	\begin{tikzcd}
		\struc{U}_i & \struc{U}_{i+1}\ar[l,twoheadrightarrow,"\Omega_i^{i+1}"']\\
		\struc{X}_i \ar[u,hook',"\iota_i"]& \struc{X}_{i+1}\ar[l,twoheadrightarrow,"\chi_i^{i+1}"]\ar[u,hook',dashed,"\iota_{i+1}"'].
	\end{tikzcd}
	\]
	Thus, $(\iota_i)_{i<\omega}\colon\cochain{X}\Rightarrow\cochain{U}$ is a natural embedding. Hence $\mlim((\iota_i)_{i<\omega})\colon\mlim\cochain{X}\injto\mlim\cochain{U}$ is a metric embedding. By Proposition~\ref{CEquiv} we have that $\mlim(\cochain{X})\cong\mstr{X}$. Thus the claim follows. 
\end{proof}

For homogeneity a similar induction should do the trick. The proper mode of homomorphisms in $\cat{\sigma C}$ is defined below:
\begin{definition}
	Let $\class{C}$ be an age, and let $\struc{U},\struc{V}\in\sigma\class{C}$. A homomorphism $\Omega\colon\struc{U}\to\struc{V}$ is called \emph{skew-homogeneous} in $\cat{\sigma C}$ if for each $\struc{A}\in\class{C}$,  for all $h\colon \struc{A}\to\struc{V}$, for all $\iota,\kappa\colon\struc{A}\injto\struc{U}$, and for all $\psi\in\Aut(\struc{V})$ such that 
	\[
	\begin{tikzcd}
		\struc{U}\ar[r,"\Omega"] & \struc{V}\\
		\struc{A} \ar[d,hook',"\kappa"']\ar[u,hook',"\iota"]\ar[r,"h"]& \struc{V}\ar[u,equals]\ar[d,"\psi","\cong"']\\
		\struc{U} \ar[r,"\Omega"] & \struc{V}
	\end{tikzcd},
	\]
	commute, there exists $\varphi\in\Aut(\struc{U})$, such that the following diagram commutes:
	\[
	\begin{tikzcd}
		\struc{U}\ar[r,"\Omega"]\ar[dd,bend right=40,dashed,"\varphi"',"\cong"] & \struc{V}\\
		\struc{A} \ar[d,hook',"\kappa"']\ar[u,hook',"\iota"]\ar[r,"h"]& \struc{V}\ar[u,equals]\ar[d,"\psi","\cong"']\\
		\struc{U} \ar[r,"\Omega"] & \struc{V}
	\end{tikzcd},
	\]
\end{definition}
Again, let us not worry for now about the existence of skew-homogeneous homomorphisms, but let us instead have a look onto their usefulness:
\begin{observation}\label{seqhom}
	Let $\class{C}$ be a \Fraisse class. Let $\cochain{U}=((\struc{U}_i)_{i<\omega}, (\Omega_i^j)_{i\le j<\omega})$ be an $\omega$-cochain over $\sigma\class{C}$, such that $\struc{U}_0$ is homogeneous and such that for all $i<\omega$ we have that $\Omega_i^{i+1}$ is skew-homogeneous in $\cat{\sigma C}$. Let $\mstr{U}\coloneqq\mlim\cochain{U}$. Then $\mstr{U}$ is homogeneous. 
\end{observation}
\begin{proof}
	Let $\mstr{X}$ be a pro-finitely generated ultrametric structure and let $\iota,\kappa$ be metric embeddings of $\mstr{X}$ into $\mstr{U}$.  By Observation~\ref{hereditarypi}, $\pi\sigma\class{C}$ is hereditary.  Since $\mstr{U}\in\pi\sigma\class{C}$ it follows that $\mstr{X}\in\pi\sigma\class{C}$. Without loss of generality, $\mstr{X}=\mlim(\cochain{X})$, where $\cochain{X} = ((\struc{X}_i)_{i<\omega},(x_i^j)_{i\le j<\omega})$.  
	
	Denote $\Seq(\mstr{U}) = ((\struc{V}_i)_{i<\omega},(v_i^j)_{i\le j<\omega})$, $\Seq(\mstr{X})=((\struc{Y}_i)_{i<\omega},(y_i^j)_{i\le j<\omega})$, $\Seq(\iota)=(\iota_i)_{i<\omega}$, and $\Seq(\kappa)=(\kappa_i)_{i<\omega}$. By Observation~\ref{seqpresemb}, $\Seq(\iota)$ and $\Seq(\kappa)$ are natural embeddings.  By Observation~\ref{isoiso}(\ref{epsiso}), $\varepsilon_{\cochain{U}}\colon\Seq(\mstr{U})\Rightarrow\cochain{U}$ is a natural embedding. Thus, $\varepsilon_{\cochain{U}}\circ\Seq(\iota)$ and  $\varepsilon_{\cochain{U}}\circ\Seq(\kappa)$ are natural embeddings, too. By induction we are going to construct a natural automorphism $\varphi=(\varphi_i)_{i<\omega}$ of $\cochain{U}$, such that $\varphi\circ\varepsilon_{\cochain{U}}\circ \Seq(\iota) = \varepsilon_{\cochain{U}}\circ\Seq(\kappa)$. 
	
	Since $\mstr{X}$ is pro-finitely generated, $\struc{Y}_0$ is finitely generated. Since $\struc{U}_0$ is homogeneous, there exists $\varphi_0\in\Aut(\struc{U}_0)$, such that the following diagram commutes:
	\[
	\begin{tikzcd}
		\struc{U}_0 \ar[r,"\varphi_0","\cong"',dashed]& \struc{U}_0\\
		\struc{Y}_0 \ar[r,equals]\ar[u,hook',"\varepsilon_{\cochain{U},0}\circ\iota_0"]& \struc{Y}_0
		\ar[u,hook',"\varepsilon_{\cochain{U},0}\circ\kappa_0"']
	\end{tikzcd}
	\]
	Suppose that $\varphi_i\in\Aut(\struc{U}_i)$ is already constructed. In particular, the following diagram commutes:
	\[
	\begin{tikzcd}
		  \struc{U}_i\ar[dddd,bend right,"\varphi_i"',"\cong"] & \struc{U}_{i+1}\ar[l,"\Omega^{i+1}_i"']\\
		 \struc{V}_i \ar[u,hook',"\varepsilon_{\cochain{U},i}"']& \struc{V}_{i+1}\ar[u,hook',"\varepsilon_{\cochain{U},i+1}"]\ar[l,two heads,"v^{i+1}_i"']\\
		\struc{Y}_i \ar[u,hook',"\iota_i"']\ar[d,hook',"\kappa_i"]& \struc{Y}_{i+1}\ar[u,hook',"\iota_{i+1}"]\ar[d,hook',"\kappa_{i+1}"']\ar[l,two heads,"y^{i+1}_i"']\\
		\struc{V}_i\ar[d,hook',"\varepsilon_{\cochain{U},i}"] & \struc{V}_{i+1}\ar[d,hook',"\varepsilon_{\cochain{U},i+1}"']\ar[l,two heads,"v^{i+1}_i"']\\
		\struc{U}_i & \struc{U}_{i+1}\ar[l,"\Omega^{i+1}_i"']\\
	\end{tikzcd}
	\]
	In particular, the following diagram commutes:
	\[
	\begin{tikzcd}[column sep=12ex]
		\struc{U}_i\ar[d,equals] & \struc{U}_{i+1}\ar[l,"\Omega^{i+1}_i"']\\
		\struc{U}_i \ar[d,"\varphi_i"',"\cong"]& \struc{Y}_{i+1}\ar[l,"\varepsilon_{\cochain{U},i}\circ\iota_i\circ y_i^{i+1}"]\ar[u,hook',"\varepsilon_{\cochain{U},i+1}\circ\iota_{i+1}"']\ar[d,hook',"\varepsilon_{\cochain{U},i+1}\circ\kappa_{i+1}"]\\
		\struc{U}_i & \struc{U}_{i+1}.\ar[l,"\Omega^{i+1}_i"']
	\end{tikzcd}
	\]
	Since $\mstr{X}$ is pro-finitely generated, $\struc{Y}_{i+1}$ is finitely generated. Since $\Omega_i^{i+1}$ is skew-homogeneous, there exists $\varphi_{i+1}\in\Aut(\struc{U}_{i+1})$, such that the following diagram commutes:
	\[
	\begin{tikzcd}[column sep=12ex]
		\struc{U}_i\ar[d,equals] & \struc{U}_{i+1}\ar[l,"\Omega^{i+1}_i"']\ar[dd,bend left=30,"\varphi_{i+1}",dashed]\\
		\struc{U}_i \ar[d,"\varphi_i"',"\cong"]& \struc{Y}_{i+1}\ar[l,"\varepsilon_{\cochain{U},i}\circ\iota_i\circ y_i^{i+1}"]\ar[u,hook',"\varepsilon_{\cochain{U},i+1}\circ\iota_{i+1}"]\ar[d,hook',"\varepsilon_{\cochain{U},i+1}\circ\kappa_{i+1}"']\\
		\struc{U}_i & \struc{U}_{i+1}.\ar[l,"\Omega^{i+1}_i"']
	\end{tikzcd}
	\]
	In particular the following diagram commutes:
	\[
	\begin{tikzcd}
		  \struc{U}_i\ar[dddd,bend right,"\varphi_i"',"\cong"] & \struc{U}_{i+1}\ar[l,"\Omega^{i+1}_i"']\ar[dddd,bend left,dashed,"\varphi_{i+1}","\cong"']\\
		 \struc{V}_i \ar[u,hook',"\varepsilon_{\cochain{U},i}"']& \struc{V}_{i+1}\ar[u,hook',"\varepsilon_{\cochain{U},i+1}"]\ar[l,two heads,"v^{i+1}_i"']\\
		\struc{Y}_i \ar[u,hook',"\iota_i"']\ar[d,hook',"\kappa_i"]& \struc{Y}_{i+1}\ar[u,hook',"\iota_{i+1}"]\ar[d,hook',"\kappa_{i+1}"']\ar[l,two heads,"y^{i+1}_i"']\\
		\struc{V}_i\ar[d,hook',"\varepsilon_{\cochain{U},i}"] & \struc{V}_{i+1}\ar[d,hook',"\varepsilon_{\cochain{U},i+1}"']\ar[l,two heads,"v^{i+1}_i"']\\
		\struc{U}_i & \struc{U}_{i+1}\ar[l,"\Omega^{i+1}_i"']\\
	\end{tikzcd}
	\]
	Thus $\varphi$ is completely specified. Applying the functor $\mlim$ to this situation, we obtain 
	\begin{equation}\label{starr}
	\mlim(\varphi)\circ\mlim(\varepsilon_{\cochain{U}})\circ \mlim(\Seq(\iota)) = \mlim(\varepsilon_{\cochain{U}})\circ\mlim(\Seq(\kappa)).
	\end{equation}
	Note that the following diagram commutes:
	\[
	\begin{tikzcd}[column sep=12ex]
		\mlim(\Seq(\mlim(\cochain{X})))\ar[r,hook,"\mlim(\Seq(\iota))"] & \mlim(\Seq(\mlim(\cochain{U})))\ar[r,"(\mlim\ast\varepsilon)_{\cochain{U}}"] & \mlim(\cochain{U})\\
		\mlim(\cochain{X})\ar[u,"(\eta\ast\mlim)_{\cochain{X}}"]\ar[r,"\iota",hook] & \mlim(\cochain{U}).\ar[u,"(\eta\ast\mlim)_{\cochain{U}}"]  \ar[ur,equals]
	\end{tikzcd}
	\]
	Here the left quadrangle commutes because $(\eta\ast\mlim)$ is a natural transformation and the right hand triangle commutes because of one of the triangle identities. Analogously we have that the following diagram commutes:
	\[
	\begin{tikzcd}[column sep=12ex]
		\mlim(\Seq(\mlim(\cochain{X})))\ar[r,hook,"\mlim(\Seq(\kappa))"] & \mlim(\Seq(\mlim(\cochain{U})))\ar[r,"(\mlim\ast\varepsilon)_{\cochain{U}}"] & \mlim(\cochain{U})\\
		\mlim(\cochain{X})\ar[u,"(\eta\ast\mlim)_{\cochain{X}}"]\ar[r,"\kappa",hook] & \mlim(\cochain{U}).\ar[u,"(\eta\ast\mlim)_{\cochain{U}}"]  \ar[ur,equals]
	\end{tikzcd}
	\]
	Thus, keeping in mind that $(\mlim\ast\varepsilon)_{\cochain{U}} = \mlim(\varepsilon_{\cochain{U}})$, we may compute
	\begin{align*}
		\mlim(\varphi)\circ\iota &= \mlim(\varphi)\circ (\mlim\ast\varepsilon)_{\cochain{U}}\circ\mlim(\Seq(\iota))\circ(\eta\ast\mlim)_{\cochain{X}}\\
		&\overset{\eqref{starr}}{=} (\mlim\ast\varepsilon)_{\cochain{U}}\circ\mlim(\Seq(\kappa))\circ(\eta\ast\mlim)_{\cochain{X}}= \mlim(\kappa)
	\end{align*}
	Thus, $\mstr{U}$ is homogeneous.
\end{proof}

\subsection{Universal homogeneous homomorphisms}
Abstract \Fraisse theory arose in a number of steps done in papers by Droste, G\"obel, and \Kubis (see \cite{DroGoe90,DroGoe92,Kub14}). Let us recall some basic facts from abstract \Fraisse theory needed in this paper. 

	 An object $A$ of a category $\cat{X}$ is called \emph{$\omega$-small} if for every $\omega$-chain $((C_i)_{i<\omega},(c_i^j)_{i\le j<\omega})$  with limiting cocone $(C_\infty,(c_i^\infty)_{i<\omega})$ and for every morphism $h\colon A\to C_\infty$ there exists $i<\omega$ and $h'\colon A\to C_i$, such that $h=c^\infty_i\circ h'$. The full subcategory of $\cat{X}$ that is induced by all $\omega$-small objects is denoted by $\cats{X}$. 

	\begin{definition}
	A category $\cat{X}$ is called \emph{semi-algebroidal} if all $\omega$-chains in $\cats{X}$ have a colimit in $\cat{X}$ and if every object of $\cat{X}$ is the colimit of an $\omega$-chain in $\cats{X}$.
	\end{definition}

	\begin{definition}
	A category $\cat{X}$ is called a \emph{\Fraisse category} if 
	\begin{enumerate}
		\item all its morphisms are monomorphisms,
		\item $\cat{X}$ has a \emph{countable dominating subcategory}, i.e., it has a subcategory $\cat{D}$ with countably many objects and morphisms such that 
		\begin{enumerate}
			\item $\cat{D}$ is cofinal in $\cat{X}$, i.e., for all $A\in\cat{X}$ there exists $B\in\cat{D}$, such that $\cat{X}(A,B)\neq\emptyset$,
			\item for all $A\in\cat{D}$, $B\in\cat{X}$, $f\in\cat{X}(A,B)$ there exists $C\in\cat{D}$, $g\in\cat{X}(B,C)$, such that $g\circ f\in\cat{D}(A,C)$, 
		\end{enumerate}
		\item $\cat{X}$ is \emph{directed}, i.e., for all $A,B\in\cat{X}$ there exists $C\in\cat{X}$ such that $\cat{X}(A,C)\neq\emptyset$, and $\cat{X}(B,C)\neq\emptyset$. 
		\item $\cat{X}$ has the \emph{amalgamation property} (\AP), i.e., for all $A,B_1,C_1\in\cat{X}$ and for all $f_1\colon A\to B_1$, $f_2\colon A\to B_2$ there exist $C\in\cat{X}$, $g_1\colon B_1\to C$, and $g_2\colon B_2\to C$, such that the following diagram commutes:
	\[
	\begin{tikzcd}
		B_1 \ar[r,"g_1",dashed]& C\\
		A \ar[u,"f_1"]\ar[r,"f_2"]& B_2.\ar[u,"g_2"',dashed]
	\end{tikzcd}
	\]
	\end{enumerate}		
	\end{definition}
	
	One of the fundamental notions in abstract \Fraisse theory is that of \Fraisse sequences:
	\begin{definition}
		Let $\cat{X}$ be a category. An $\omega$-chain $\chain{A}=((A_i)_{i<\omega},(\alpha_i^j)_{i\le j<\omega})\in[\cat{\omega},\cat{X}]$ is called a \emph{\Fraisse sequence} if 
		\begin{enumerate}
			\item the image of $\chain {A}$ is cofinal in $\cat{X}$, i.e., for all $B\in\cat{X}$ there exists $i<\omega$, such that $\cat{X}(B,A_i)\neq\emptyset$,
			\item $\chain{A}$ has the \emph{absorption property}, i.e., for all $B\in\cat{X}$, $i<\omega$, $f\in\cat{X}(A_i, B)$ there exist $j\ge i$, $g\in\cat{X}(B,A_j)$ such that $g\circ f = \alpha_i^j$. 
		\end{enumerate}
	\end{definition}
	Note that if $\cat{X}$ has a \Fraisse sequence then the image of this $\omega$-chain forms a countable dominating subcategory of $\cat{X}$. Finally, let us formally introduce the notions of universality and homogeneity in the abstract setting of categories:
	\begin{definition}
		Let $\cat{X}$ be a category and let $\cat{Y}$ be a full subcategory of $\cat{X}$.    An object $U$ of $\cat{X}$ is called \emph{$\cat{Y}$-universal} if for all $Y\in\cat{Y}$ we have $\cat{X}(Y,U)\neq\emptyset$.  Moreover, $U$ is called \emph{$\cat{Y}$-homogeneous} if for all $Y\in\cat{Y}$ and for all morphisms $f,g\colon Y\to U$ there exists $\varphi\in\Aut(U)$, such that $\varphi\circ g = f$. 
	\end{definition}
	\begin{remark}
		In case that $\cat{Y}=\cat{X}$, then instead of ``$\cat{Y}$-universal'' we just say ``universal''. Moreover if $\cat{Y}=\cats{X}$, then instead of ``$\cat{Y}$-homogeneous'' we just say ``homogeneous''.
	\end{remark}

	Of particular interest in abstract \Fraisse theory are semi-algebroidal categories $\cat{X}$ with all morphisms monic, for which $\cats{X}$ is a \Fraisse category. 

In the following we collect the results from abstract \Fraisse theory that are needed in the present context:
\begin{theorem}[Droste,G\"obel \cite{DroGoe90}, Kubi\'s \cite{Kub14}]\label{FraisseCat}
	Let $\cat{X}$ be a semi-algebroidal category all of whose morphisms are monic. Suppose that $\cats{X}$ is a \Fraisse category. Let $\cat{D}$ be a countable dominating subcategory of $\cats{X}$. Then
	\begin{enumerate}
		\item\label{fseqexist} $\cats{X}$ affords a \Fraisse sequence $\chain{F}$ whose image lies completely in  $\cat{D}$,
		\item\label{univhomobj} if  $U$ is a colimit of $\chain{F}$ in $\cat{X}$, then it is universal and homogeneous in $\cat{X}$,
		\item\label{univhomunique} any two universal homogeneous objects of $\cat{X}$ are isomorphic. 
	\end{enumerate}
\end{theorem}
\begin{proof}
	(\ref{fseqexist}) is (the proof of) \cite[Corollary 3.8]{Kub14}.
	(\ref{univhomunique}) is \cite[Theorem 1.1]{DroGoe90}.
	
	For (\ref{univhomobj}) is implicit in \cite{Kub14}. For the convenience of the reader, let us give the technical details: Suppose that $\chain{F}=((A_i)_{i<\omega},(\alpha_i^j)_{i\le j<\omega})$. Let $(U,(\alpha_i^\infty)_{i<\omega})$ be a limiting cocone.

	First we show that $U$ is homogeneous. Let $B\in\cats{X}$, and let $\iota,\kappa\colon B\to U$. By induction we are going to construct the following commuting diagram:
	\[
	\begin{tikzcd}[column sep=small]
		& A_{n_0} \ar[dr,"\mu_1"] \ar[rr,"\alpha_{n_0}^{n_1}"]& & A_{n_1}\ar[dr,"\mu_2"] \ar[rr,"\alpha_{n_1}^{n_2}"] & & A_{n_2} \ar[r]& \dots\\
		B=B_0 \ar[ur,"\kappa_0"] \ar[dr,"\iota_0"'] \ar[rr,"\beta_0^1"]& & B_1\ar[ur,"\kappa_1"] \ar[dr,"\iota_1"'] \ar[rr,"\beta_1^2"] & & B_2\ar[ur,"\kappa_2"] \ar[dr,"\iota_2"'] \ar[r] &\dots\\
		& A_{n_0} \ar[ur,"\lambda_1"']\ar[rr,"\alpha_{n_0}^{n_1}"']& & A_{n_1}\ar[ur,"\lambda_2"']\ar[rr,"\alpha_{n_1}^{n_2}"'] & & A_{n_2} \ar[r]& \dots
	\end{tikzcd}
	\]
	Where $(n_i)_{i<\omega}$ is a strictly increasing sequence of non-negative integers.
	
	Since $\cat{X}$ is semi-algebroidal, there exists $n_0<\omega$, $\iota_0,\kappa_0\colon B\to A_{n_0}$, such that $\iota=\alpha_{n_0}^\infty\circ\iota_0$ and $\kappa=\alpha_{n_0}^\infty\circ\kappa_0$. Define $B_0\coloneqq B$. Suppose that $\iota_i,\kappa_i\colon B_i\to A_{n_i}$ are already constructed.   By the \AP there exists $B_{i+1}\in\cats{X}$, $\mu_{i+1}\colon A_{n_i}\to B_{i+1}$, $\lambda_{i+1}\colon A_{n_i}\to B_{i+1}$, such that $\mu_{i+1}\circ\kappa_i = \lambda_{i+1}\circ\iota_i$. Define $\beta^{i+1}_i\coloneqq\lambda_{i+1}\circ\iota_i$. Since $\chain{F}$ is a \Fraisse sequence, there exists $n_{i+1}> n_i$ and $k_{i+1},\iota_{i+1}\colon B_{i+1}\to A_{n_{i+1}}$, such that $\alpha^{n_{i+1}}_{n_i} = \iota_{i+1}\circ\lambda_{i+1} = \kappa_{i+1}\circ\mu_{i+1}$. Thus the diagram is constructed.
	
	By \cite[Proposition 3.3(a)]{Kub14}, the chain $((A_{n_i})_{i<\omega}, (\alpha_{n_i}^{n_j})_{i\le j<\omega})$ is a \Fraisse sequence, too. Without loss of generality we may assume that for each $i<\omega$ we have that $n_i=i$. Under this assumption the commuting diagram from above looks as follows  
	\[
	\begin{tikzcd}[column sep=small]
		& A_{0} \ar[dr,"\mu_1"] \ar[rr,"\alpha_{0}^{1}"]& & A_{1}\ar[dr,"\mu_2"] \ar[rr,"\alpha_{1}^{2}"] & & A_{2} \ar[r]& \dots\\
		B_0 \ar[ur,"\kappa_0"] \ar[dr,"\iota_0"'] \ar[rr,"\beta_0^1"]& & B_1\ar[ur,"\kappa_1"] \ar[dr,"\iota_1"'] \ar[rr,"\beta_1^2"] & & B_2\ar[ur,"\kappa_2"] \ar[dr,"\iota_2"'] \ar[r] &\dots\\
		& A_{0} \ar[ur,"\lambda_1"']\ar[rr,"\alpha_{0}^{1}"']& & A_{1}\ar[ur,"\lambda_2"']\ar[rr,"\alpha_{1}^{2}"'] & & A_{2} \ar[r]& \dots
	\end{tikzcd}
	\]
	Let $(B_\infty,(\beta^\infty_i)_{i<\omega})$ be a limiting cocone for $\chain{B}=((B_i)_{i<\omega},(\beta_i^j)_{i\le j<\omega})$. Note that $(U, (\alpha^\infty_i\circ\iota_i)_{i<\omega})$ and $(U, (\alpha_i^\infty\circ\kappa_i)_{i<\omega})$ are compatible cocones for $\chain{B}$. Let $\iota_\infty,\kappa_\infty\colon B_\infty\to U$ be the respective mediating morphisms. In particular we have for each $i<\omega$ that $\iota_\infty\circ\beta^\infty_i =  \alpha^\infty_i\circ\iota_i$. and $\kappa_\infty\circ\beta^\infty_i =  \alpha^\infty_i\circ\kappa_i$.

	Similarly note that $(B_\infty, (\beta^\infty_{i+1}\circ\mu_{i+1})_{i<\omega})$ and $(B_\infty, (\beta^\infty_{i+1}\circ\lambda_{i+1})_{i<\omega})$ are compatible cocones for $\chain{F}$.  Let $\mu_\infty,\lambda_\infty\colon U\to B_\infty$ be the respective mediating morphisms.  In particular we have for each $i<\omega$ that $\mu_\infty\circ \alpha^\infty_i = \beta_{i+1}^\infty\circ\mu_{i+1}$, and $\lambda_\infty\circ \alpha^\infty_i = \beta_{i+1}^\infty\circ\lambda_{i+1}$.
	
	Next we compute for every $i<\omega$ that 
	\[
	\lambda_\infty\circ\iota_\infty\circ\beta^\infty_i = \lambda_\infty\circ \alpha^\infty_i\circ\iota_i = \beta^\infty_{i+1}\circ\lambda_{i+1}\circ\iota_i = \beta^\infty_{i+1}\circ\beta_i^{i+1} = \beta^\infty_i.
	\]
	Thus, $\lambda_\infty\circ\iota_\infty= 1_{B_\infty}$. Similarly we may compute for each $i<\omega$ that 
	\[
	\iota_\infty\circ\lambda_\infty\circ\alpha_i^\infty = \iota_\infty\circ \beta_{i+1}^\infty\circ\lambda_{i+1} = \alpha_{i+1}^\infty\circ \iota_{i+1}\circ\lambda_{i+1} = \alpha^\infty_{i+1}\circ\alpha_i^{i+1}= \alpha_i^\infty. 
	\]
	Hence, $\iota_\infty\circ\lambda_\infty = 1_U$. It follows that both $\iota_\infty$ and $\lambda_\infty$ are mutually inverse isomorphisms. Similarly it can be shown that $\kappa_\infty$ and $\mu_\infty$ are mutually inverse isomorphisms. Define $\varphi\coloneqq\iota_\infty\circ\mu_\infty$. Then 
	\begin{align*}
	\varphi\circ\kappa &= \varphi\circ\alpha^\infty_0\circ\kappa_0 = \iota_\infty\circ\mu_\infty\circ\alpha^\infty_0\circ\kappa_0 = \iota_\infty\circ \beta^\infty_{1}\circ\mu_{1}\circ\kappa_0= \alpha_1^\infty\circ\iota_1\circ\mu_1\circ\kappa_0 = \alpha_1^\infty\circ\iota_1\circ\beta^0_1\\
	&= \alpha_1^\infty\circ \iota_1\circ\lambda_1\circ\iota_0= \alpha_1^\infty\circ \alpha_0^1\circ\iota_0 = \iota_0\circ\alpha_0^\infty = \iota.
	\end{align*}
	Thus $U$ is homogeneous. Universality of $U$ follows easily from the cofinality of the image of $\chain{F}$ in $\cats{X}$. 
\end{proof}
The previous, \Fraisse-type theorem  may be formulated even stronger:
\begin{observation}
	Let $\cat{C}$ be a category all of whose morphisms are monomorphisms, such that $\cat{C}$ affords a \Fraisse sequence. Then $\cat{C}$ is a \Fraisse category.
\end{observation}
\begin{proof}
	\cite[Proposition 3.1]{Kub14} entails that $\cat{C}$ is directed and has the \AP. The image of a \Fraisse sequence in $\cat{C}$ is a countable dominating subcategory of $\cat{C}$. 
\end{proof}

\subsection{Universal and skew-homogeneous homomorphisms through comma-categories}\label{commacats}
The question of existence of universal and/or homogeneous homomorphisms of various kinds was treated in \cite{PecPec18b}, where, more generally, universal homogeneous objects in comma categories are studied. 

Recall that a comma category $(\func{F}\comma \func{G})$ is specified through two functors $\func{F}\colon\cat{A}\to\cat{C}$ and $\func{G}\colon\cat{B}\to\cat{C}$. Its objects are triples $(A,f,B)$, where $A$ and $B$ are objects of $\cat{A}$ and $\cat{B}$, respectively, and where $f\colon \func{F}A\to\func{G}B$ is a morphism in $\cat{C}$. The morphisms from $(A,f,B)$ to $(A',f',B')$ are pairs $(a,b)$, where $a\colon A\to A'$ in $\cat{A}$, $b\colon B\to B'$ in $\cat{B}$, such that the following diagram is commutative:
\[
\begin{tikzcd}
	\func{F}A' \ar[r,"f'"]& \func{G}{B'} \\
	\func{F}A \ar[r,"f"]\ar[u,"\func{F}a"]& \func{G}{B}\ar[u,"\func{G}b"'].
\end{tikzcd}
\] 
The composition of two arrows is defined component wise, i.e., $(c,d)\circ (a,b)\coloneqq (c\circ a,d\circ b)$.

In terms of comma categories the previously defined notions of universal and skew-homogeneous homomorphisms  appear as follows: With $(\cat{\sigma\class{C}},\injto)$  we denote the subcategory of $\cat{\sigma\class{C}}$ that is spanned by all embeddings. Moreover, for any $\struc{V}\in\sigma\class{C}$ and for every subgroup $H$ of $\Aut(\struc{V})$ we denote by $(\struc{V},H)$ the category with exactly one object $\struc{V}$, such that the morphisms are the elements of $H$. By $\func{F}$ we will denote the identical embedding functor of $(\cat{\sigma C},\injto)$ into $\cat{\sigma C}$. The identical embedding functor of $(\struc{V},H)$ into $\cat{\sigma C}$ is denoted by $\func{G}_{(\struc{V},H)}$. In the special case that $H$ consists just of the identity automorphism of $\struc{V}$, instead of $\func{G}_{(\struc{V},\{1_{\struc{U}}\})}$ we just write $\func{G}_{\struc{V}}$.  
\begin{observation}\label{univhomcomma}
	Let $\struc{U},\struc{V}\in\sigma\class{C}$, and let  $\Omega\colon\struc{U}\to\struc{V}$. Then
	\begin{enumerate}
		\item $\Omega$ is universal if and only if $(\struc{U},\Omega,\struc{V})$ is a universal object in $(\func{F}\comma \func{G}_{\struc{V}})$,
		\item $\Omega$ is skew-homogeneous if and only if $(\struc{U},\Omega,\struc{V})$ is a  homogeneous object in $(\func{F}\comma \func{G}_{(\struc{V},\Aut(\struc{V}))})$. 
	\end{enumerate}
\end{observation} 
\begin{proof}
	Clear.
\end{proof}
	
\begin{observation}\label{algebroidal}
	The category $(\func{F}\comma\func{G}_{(\struc{V},\Aut(\struc{V}))})$ is semi-algebroidal. Moreover, an object $(\struc{A},h,\struc{V})$ of $(\func{F}\comma\func{G}_{(\struc{V},\Aut(\struc{V}))})$ is $\omega$-small if and only if $\struc{A}$ is finitely generated (i.e., it is $\omega$-small in $\cat{\sigma C}$).
\end{observation}
\begin{proof}
	This follows directly from \cite[Proposition 4.4]{PecPec18b}. The conditions $P1,\dots P7$ from \cite[Definition 4.1]{PecPec18b} are easily verified in the present case. 
\end{proof}

For the  category $(\func{F}\comma\func{G}_{(\struc{V},\Aut(\struc{V}))})_{<\omega}$ a criterion for the \AP is:
\begin{observation}\label{skewhom}
	The category $(\func{F}\comma\func{G}_{(\struc{V},\Aut(\struc{V}))})_{<\omega}$ has the \AP if and only if for all $\struc{A},\struc{B}_1,\struc{B}_2\in\class{C}$, $f_1\colon \struc{A}\injto\struc{B}_1$, $f_2\colon \struc{A}\injto\struc{B}_2$, $h_1\colon \struc{B}_1\to\struc{V}$, $h_2\colon \struc{B}_2\to\struc{V}$, if $h_1\circ f_1=h_2\circ f_2$, then there exist $\struc{C}\in\class{C}$, $g_1\colon \struc{B}_1\injto\struc{C}$, $g_2\colon \struc{B}_2\injto\struc{C}$, $h\colon \struc{C}\to\struc{V}$, $k\in\Aut(\struc{V})$ such that the following diagram commutes:
	\[
	\begin{tikzcd}
		& & \struc{V}\arrow[r,dashed,"k","\cong"'] & \struc{V}.\\
		\struc{B}_1\arrow[r,dashed,hook,"g_1"]\arrow[urr,bend left,"h_1"] & \struc{C}\arrow[rru,dashed,bend right,near end,"h"']\\
		\struc{A}\arrow[u,hook',"f_1"]\arrow[r,hook,"f_2"] & \struc{B}_2\arrow[u,dashed,hook',"g_2"]\arrow[uur,bend right,"h_2"']
	\end{tikzcd}
	\]
\end{observation}
\begin{proof}
	This follows directly from the observation that for any subgroup $H\le\Aut(\struc{V})$ the category $(\struc{V},H)$ has the \AP, in conjunction with \cite[Proposition 5.2(2)]{PecPec18b}.
\end{proof}
\begin{remark}\label{vvalapdef}
	Note how the formulation of the condition in Observation~\ref{skewhom} is redundant. Since $k$ is an automorphism of $\struc{V}$, it can be removed and $h$ may be replaced by $k^{-1}\circ h$. The reformulation of the condition is then the following:
	
	For all $\struc{A},\struc{B}_1,\struc{B}_2\in\class{C}$, $f_1\colon \struc{A}\injto\struc{B}_1$, $f_2\colon \struc{A}\injto\struc{B}_2$, $h_1\colon \struc{B}_1\to\struc{V}$, $h_2\colon \struc{B}_2\to\struc{V}$, if $h_1\circ f_1=h_2\circ f_2$, then there exist $\struc{C}\in\class{C}$, $g_1\colon \struc{B}_1\injto\struc{C}$, $g_2\colon \struc{B}_2\injto\struc{C}$, $h\colon \struc{C}\to\struc{V}$ such that the following diagram commutes:
	\[
	\begin{tikzcd}
		& & \struc{V}.\\
		\struc{B}_1\arrow[r,dashed,hook,"g_1"]\arrow[urr,bend left,"h_1"] & \struc{C}\arrow[ru,dashed,near end,"h"']\\
		\struc{A}\arrow[u,hook',"f_1"]\arrow[r,hook,"f_2"] & \struc{B}_2\arrow[u,dashed,hook',"g_2"]\arrow[uur,bend right,"h_2"']
	\end{tikzcd}
	\]
	If $\class{C}$ satisfies this condition for a given fixed $\struc{V}\in\sigma\class{C}$, then we say that $\class{C}$ has the \emph{$\struc{V}$-valued amalgamation property}.
\end{remark}
This leads us immediately to the next observation:
\begin{observation}\label{apvvap}
	The category $(\func{F}\comma \func{G}_{\struc{V}})$ has the \AP if and only if $\class{C}$ has the $\struc{V}$-valued amalgamation property. 
\end{observation}
\begin{proof}
	This follows directly from \cite[Proposition 5.2(2)]{PecPec18b}.
\end{proof}

An immediate consequence is that $(\func{F}\comma\func{G}_{(\struc{V},\Aut(\struc{V}))})_{<\omega}$ has the \AP if and only if $(\func{F}\comma\func{G}_{\struc{V}})_{<\omega}$ does.

Our next observation regards directedness.
\begin{observation}\label{apdir}
	Suppose that $(\func{F}\comma\func{G}_{(\struc{V},\Aut(\struc{V}))})_{<\omega}$ has the \AP. Then it is also directed. The same holds for $(\func{F}\comma\func{G}_{\struc{V}})_{<\omega}$.
\end{observation}
\begin{proof}
	Since $\class{C}$ is an age, it has in particular the \JEP. Consequently, for any two structures $\struc{A},\struc{B}\in\class{C}$ we have that $\langle \emptyset\rangle_{\struc{A}}\cong \langle \emptyset\rangle_{\struc{B}}$. Let us denote by $\langle\emptyset\rangle$ this unique (up to isomorphism) joint smallest substructure for all elements of $\class{C}$. Clearly, $\langle\emptyset\rangle$ is an initial object in $\cat{C}$. Let $h\colon \langle\emptyset\rangle\to\struc{V}$ be the unique homomorphism. Then $(\langle\emptyset\rangle,h,\struc{V})$ is weakly initial in $(\func{F}\comma\func{G}_{\struc{V},\Aut(\struc{V})})_{<\omega}$. It is easy to see that under these circumstances the \AP entails directedness.
	
	The proof for   $(\func{F}\comma\func{G}_{\struc{V}})_{<\omega}$ goes completely analogously.
\end{proof}

It remains to check for the existence of a countable dominating subcategory:
\begin{observation}\label{countdom}
	It should be noted that $(\func{F}\comma \func{G}_{\struc{V}})$ forms a subcategory of $(\func{F}\comma \func{G}_{(\struc{V},\Aut(\struc{V}))})$. Since $\class{C}$ is an age, It is not hard to see that any skeleton of  $(\func{F}\comma \func{G}_{\struc{V}})_{<\omega}$ is countable. However, every category is dominated by any of its skeletons.  For $(\func{F}\comma \func{G}_{(\struc{V},\Aut(\struc{V}))})$ we need to take care that the category of $\omega$-small objects may be uncountable, because $|\Aut(\struc{V})|$ might be uncountable. However,  $(\func{F}\comma \func{G}_{\struc{V}})_{<\omega}$ is a  dominating subcategory of $(\func{F}\comma \func{G}_{(\struc{V},\Aut(\struc{V}))})_{<\omega}$. Thus, clearly,  every countable dominating subcategory of $(\func{F}\comma \func{G}_{\struc{V}})_{<\omega}$ dominates $(\func{F}\comma \func{G}_{(\struc{V},\Aut(\struc{V}))})_{<\omega}$, too.
\end{observation}
\begin{proof}
We have to show that that $(\func{F}\comma \func{G}_{\struc{V}})_{<\omega}$ dominates $(\func{F}\comma \func{G}_{(\struc{V},\Aut(\struc{V}))})_{<\omega}$.

	Since $(\func{F}\comma \func{G}_{(\struc{V},\Aut(\struc{V}))})_{<\omega}$ and $(\func{F}\comma \func{G}_{\struc{V}})_{<\omega}$ have the same objects, cofinality is trivially given.

	Let $(f,\varphi) \colon (\struc{A}_1,h_1,\struc{V})\to (\struc{A}_2,h_2,\struc{V})$ be a morphism of  $(\func{F}\comma \func{G}_{(\struc{V},\Aut(\struc{V}))})_{<\omega}$. In particular the following diagram commutes:
	\[
	\begin{tikzcd}
		\struc{A}_2 \ar[r,"h_2"]& \struc{V}\\
		\struc{A}_1\ar[u,hook',"f"]\ar[r,"h_1"] & \struc{V}\ar[u,"\varphi"',"\cong"]
	\end{tikzcd}
	\]
	Note that also the following diagram commutes:
	\[
	\begin{tikzcd}
		\struc{A}_2 \ar[r,"\varphi^{-1}\circ h_2"]& \struc{V}\\
		\struc{A}_2 \ar[r,"h_2"]\ar[u,"1_{\struc{A}_2}","\cong"']& \struc{V}\ar[u,"\varphi^{-1}"',"\cong"]\\
		\struc{A}_1\ar[u,hook',"f"]\ar[r,"h_1"] & \struc{V}.\ar[u,"\varphi"',"\cong"]
	\end{tikzcd}
	\]
	In particular $(1_{\struc{A}_2},\varphi^{-1})\circ (f,\varphi)=(f,1_{\struc{V}})$ is a morphism of $(\func{F}\comma \func{G}_{\struc{V}})_{<\omega}$.
\end{proof}

\begin{corollary}\label{corvvap}
		Let $\class{C}$ be an age, and let $\struc{V}\in\sigma\class{C}$. Then the following are equivalent:
		\begin{enumerate}
			\item\label{aep} $\class{C}$ has the $\struc{V}$-valued \AP,
			\item\label{largecommafraisse} $(\func{F}\comma \func{G}_{(\struc{V},\Aut(\struc{V}))})$ is a \Fraisse category,
			\item\label{smallcommafraisse} $(\func{F}\comma \func{G}_{\struc{V}})$ is a \Fraisse category.
		\end{enumerate}
\end{corollary}
\begin{proof}
	``(\ref{aep})$\Rightarrow$(\ref{largecommafraisse})'': By Observation~\ref{skewhom} and the remark following it,  $(\func{F}\comma \func{G}_{(\struc{V},\Aut(\struc{V}))})_{<\omega}$ has the \AP.  By Observation~\ref{apdir}, $(\func{F}\comma \func{G}_{(\struc{V},\Aut(\struc{V}))})_{<\omega}$ is directed. By Observation~\ref{countdom}, $(\func{F}\comma \func{G}_{(\struc{V},\Aut(\struc{V}))})_{<\omega}$ has a countable dominating subcategory. Hence it is a \Fraisse category. 
	
	``(\ref{largecommafraisse})$\Rightarrow$(\ref{aep})'': This is a direct consequence of Observation~\ref{skewhom} and the remark following it.
	
	``(\ref{aep})$\Leftrightarrow$(\ref{smallcommafraisse})'': This is proved analogously, using Observation~\ref{apvvap}.
\end{proof}

Now we are ready to compare universal homogeneous objects in  $(\func{F}\comma\func{G}_{\struc{V}})$  and  $(\func{F}\comma \func{G}_{(\struc{V},\Aut(\struc{V}))})$:
\begin{observation}\label{equivunivhom}
	Let $\class{C}$ be an age, $\struc{V}\in\class{C}$, such that $\class{C}$ has the $\struc{V}$-valued \AP. Let $(\struc{U},\Omega,\struc{V})\in (\func{F}\comma \func{G}_{\struc{V}})$. Then  $(\struc{U},\Omega,\struc{V})$  universal homogeneous in $(\func{F}\comma \func{G}_{\struc{V}})$ if and only if it is universal homogeneous in $(\func{F}\comma \func{G}_{(\struc{V},\Aut(\struc{V}))})$.
\end{observation}
\begin{proof}
	``$\Leftarrow$'':
	Let $(\struc{A},h,\struc{V})\in (\func{F}\comma \func{G}_{\struc{V}})_{<\omega}$, and let $(\iota_1,1_{\struc{V}}),(\iota_2,1_{\struc{V}})\colon (\struc{A},h,\struc{V})\injto(\struc{U},\Omega,\struc{V})$. Since $(\struc{U},\Omega,\struc{V})$ is homogeneous in $(\func{F}\comma \func{G}_{(\struc{V},\Aut(\struc{V}))})$, there exists $(\varphi,\psi)\in\Aut(\struc{U},\Omega,\struc{V})$, such that $(\varphi,\psi)\circ(\iota_1,1_{\struc{V}}) = (\iota_2,1_{\struc{V}})$. In particular, $\psi=1_{\struc{V}}$. This shows that $(\struc{U},\Omega,\struc{V})$ is homogeneous in $(\func{F}\comma \func{G}_{\struc{V}})$. 
	
	Let now $(\widetilde{\struc{U}},\widetilde{\Omega},\struc{V})$ be universal homogeneous in $(\func{F}\comma \func{G}_{\struc{V}})$ (this exists by Corollary~\ref{corvvap} in conjunction with  Theorem~\ref{FraisseCat}). Since $(\struc{U},\Omega,\struc{V})$ is universal in $(\func{F}\comma \func{G}_{(\struc{V},\Aut(\struc{V}))})$, there exists some  $(\iota,\varphi)\colon(\widetilde{\struc{U}},\widetilde{\Omega},\struc{V})\to(\struc{U},\Omega,\struc{V})$. 
	Let $(\struc{A},h,\struc{V})\in (\func{F}\comma \func{G}_{\struc{V}})_{<\omega}$. 
	Since $(\widetilde{\struc{U}},\widetilde{\Omega},\struc{V})$ is universal in $(\func{F}\comma \func{G}_{\struc{V}})$, there exists $(\kappa,1_{\struc{V}})\colon (\struc{A},\varphi^{-1}\circ h,\struc{V})\to (\widetilde{\struc{U}},\widetilde{\Omega},\struc{V})$. Note that then $(\iota\circ\kappa, 1_{\struc{V}})\colon (\struc{A},h,\struc{V}) \to(\struc{U},\Omega,\struc{V})$.   

	``$\Rightarrow$'': Let $(\widetilde{\struc{U}},\widetilde{\Omega},\struc{V})$ be universal homogeneous in $(\func{F}\comma \func{G}_{(\struc{V},\Aut(\struc{V}))})$ (this exists by Corollary~\ref{corvvap} in conjunction with  Theorem~\ref{FraisseCat}). By the other part of the proof, $(\widetilde{\struc{U}},\widetilde{\Omega},\struc{V})$ is homogeneous in $(\func{F}\comma \func{G}_{\struc{V}})$. By Theorem~\ref{FraisseCat}, there exists an isomorphism $(\varphi,1_{\struc{V}})\colon(\struc{U},\Omega,\struc{V})\to(\widetilde{\struc{U}},\widetilde{\Omega},\struc{V})$. Since $(\varphi,1_{\struc{V}})$ is also an isomorphism in $(\func{F}\comma \func{G}_{(\struc{V},\Aut(\struc{V}))})$, it follows that $(\struc{U},\Omega,\struc{V})$ is homogeneous in $(\func{F}\comma \func{G}_{(\struc{V},\Aut(\struc{V}))})$. 
	
	Universality of $(\struc{U},\Omega,\struc{V})$ in $(\func{F}\comma \func{G}_{(\struc{V},\Aut(\struc{V}))})$ is immediate.
\end{proof}
Let us now connect our findings with the previous section: 
\begin{corollary}\label{critunivskewhom}
	Let $\class{C}$ be an age and let $\struc{V}\in\sigma\class{C}$. Then a universal skew-homogeneous homomorphism $\Omega\colon \struc{U}\to\struc{V}$ exists for some $\struc{U}\in\sigma\class{C}$ if and only if $\class{C}$ has the $\struc{V}$-valued \AP. 
\end{corollary}
\begin{proof}
	``$\Rightarrow$'': Let $\Omega\colon \struc{U}\to\struc{V}$ be a universal skew-homogeneous homomorphism in $\cat{\sigma C}$.  By Observation~\ref{univhomcomma}, $\Omega$ is universal in $(\func{F}\comma \func{G}_{\struc{V}})$ and homogeneous in $(\func{F}\comma \func{G}_{(\struc{V},\Aut(\struc{V}))})$. Clearly, universality in $(\func{F}\comma \func{G}_{\struc{V}})$ entails universality in $(\func{F}\comma \func{G}_{(\struc{V},\Aut(\struc{V}))})$. By \cite[Proposition 2.2(c)]{DroGoe92}, $(\func{F}\comma \func{G}_{(\struc{V},\Aut(\struc{V}))})_{<\omega}$ has the \AP. Finally, by the remark after Observation~\ref{skewhom}, $\class{C}$ has the $\struc{V}$-valued \AP.
	
	``$\Leftarrow$'': Suppose that $\class{C}$ has the $\struc{V}$-valued \AP. By Corollary~\ref{corvvap}, $(\func{F}\comma \func{G}_{(\struc{V},\Aut(\struc{V}))})$ is a \Fraisse category. By Theorem~\ref{FraisseCat}, $(\func{F}\comma \func{G}_{(\struc{V},\Aut(\struc{V}))})$ has a universal homogeneous object $(\struc{U},\Omega,\struc{V})$. By Observation~\ref{equivunivhom}, $(\struc{U},\Omega,\struc{V})$ is also universal homogeneous in $(\func{F}\comma \func{G}_{\struc{V}})$. Finally, by Observation~\ref{univhomcomma}, $\Omega$ is skew-homogeneous and universal. 
\end{proof}

\subsection{Concerning the \texorpdfstring{$\struc{V}$}{V}-valued amalgamation property}
Our previous findings suggest that in order to be able to construct $\omega$-cochains as in Observations~\ref{sequniv} and \ref{seqhom}, we need to check the validity of the $\struc{V}$-valued \AP for many, potentially different, structures $\struc{V}$. On the positive side, since we would like to construct an $\omega$-cochain that satisfies the requirements of both Observations, the first structure $\struc{U}_0$ in the sequence should be universal and homogeneous in $\sigma\class{C}$. That means that from the beginning we may assume that $\class{C}$ is not only an age but indeed a \Fraisse class. Before coming to the main result of this section, let us observe that the $\struc{V}$-valued \AP is in some sense hereditary:

\begin{observation}\label{vvalapcsub}
	Let $\class{C}$ be an age, let $\struc{U},\struc{V}\in\sigma\class{C}$, and let $\iota\colon\struc{U}\injto\struc{V}$. If $\class{C}$ has the $\struc{V}$-valued \AP, then it also has the $\struc{U}$-valued \AP.
\end{observation}
\begin{proof}
	Given the following commuting diagram (where $\struc{A},\struc{B}_1,\struc{B}_2\in\class{C}$):
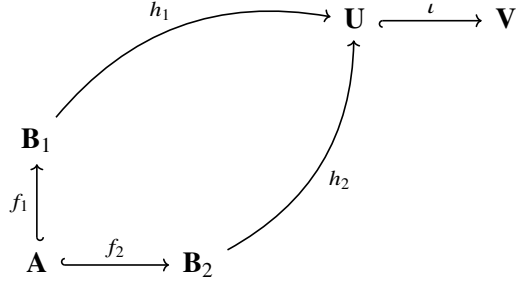
	\[
	\begin{tikzcd}
		& & \struc{U}\ar[r,hook',"\iota"] & \struc{V}\\
		\struc{B}_1\arrow[urr,bend left,"h_1"]\\
		\struc{A} \ar[u,hook',"f_1"]\ar[r,hook,"f_2"]& \struc{B}_2\arrow[uur,bend right,"h_2"']
	\end{tikzcd}
	\]
	Since $\class{C}$ has the $\struc{V}$-valued \AP, there exist $\struc{C}\in\class{C}$, and $g_1,g_2,h$, such that the following diagram commutes:
	
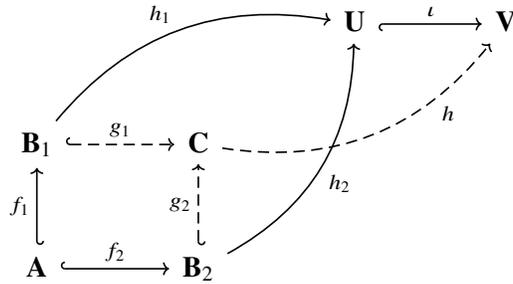
\[
	\begin{tikzcd}
		& & \struc{U}\ar[r,hook',"\iota"] & \struc{V}\\
		\struc{B}_1\arrow[urr,bend left,"h_1"]\ar[r,hook,"g_1",dashed] & \struc{C}\ar[rru,near end,"h"',bend right,dashed]\\
		\struc{A} \ar[u,hook',"f_1"]\ar[r,hook,"f_2"]& \struc{B}_2\arrow[uur,bend right,"h_2"']\ar[u,hook',"g_2",dashed]
	\end{tikzcd}
	\]
	Without loss of generality, $\struc{C} = \langle g_1(B_1)\cup g_2(B_2)\rangle_{\struc{C}}$. 
	Let $\struc{D}\coloneqq\langle h_1(B_1)\cup h_2(B_2)\rangle_{\struc{U}}$. Now we compute:
	\begin{align*}
		\iota(\struc{D}) &= \iota(\langle h_1(B_1)\cup h_2(B_2)\rangle_{\struc{U}}) = \langle\iota(h_1(B_1))\cup\iota(h_2(B_2))\rangle_{\struc{V}}\\
		&=\langle h(g_1(B_1))\cup h(g_2(B_2))\rangle_{\struc{V}} = h(\langle g_1(B_2)\cup g_2(B_2)\rangle_{\struc{U}}) = h(\struc{C}).
	\end{align*}
	Let $\widetilde{\struc{D}}\coloneqq\iota(\struc{D})$, and let $\tilde\iota$ be the image restriction of $\iota$ to $\widetilde{\struc{D}}$. Then, in particular, $\tilde\iota$ is an isomorphism. Let $\tilde{h}$ be the image restriction of $h$ to $\widetilde{\struc{D}}$. Then for each $i\in\{1,2\}$ the following diagram commutes:
	
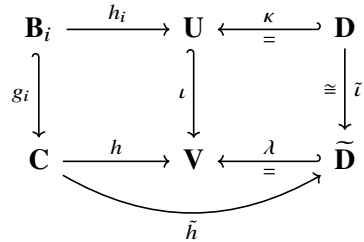
\[
	\begin{tikzcd}
		& \struc{B}_i \ar[r,"h_i"]\ar[d,"g_i"',hook'] & \struc{U}\ar[d,hook',"\iota"'] & \struc{D}\ar[l,"\kappa"',hook',"="]\ar[d,"\tilde\iota","\cong"']\\
		& \struc{C}\ar[r,"h"]\ar[rr,bend right=30,"\tilde{h}"'] & \struc{V} & \widetilde{\struc{D}}\ar[l,"=","\lambda"',hook']
	\end{tikzcd}
	\]
	Define $h'\coloneqq\kappa\circ{\tilde\iota}^{\,-1}\circ\tilde{h}$. Then we may compute:
	\[
	\iota\circ h'\circ g_i = \iota\circ \kappa\circ\tilde\iota^{\,-1}\circ\tilde{h}\circ g_i  = \lambda\circ \tilde\iota\circ\tilde\iota^{\,-1}\circ\tilde{h}\circ g_i = \lambda\circ\tilde{h}\circ g_i = h\circ g_i = \iota\circ h_i. 
	\]
	Since $\iota$ is injective, it follows that $h'\circ g_i = h_i$, for each $i\in\{1,2\}$. Consequently, $\class{C}$ has the $\struc{U}$-valued \AP.  
\end{proof}

Now we can sum up all our findings concerning the $\struc{V}$-valued amalgamation property:
\begin{proposition}\label{propertieseq}
	Let $\class{C}$ be an \Fraisse class with \Fraisse limit $\struc{V}$. Then the following are equivalent 
	\begin{enumerate}
		\item\label{aepc} $\class{C}$ has the \AEP,
		\item\label{vvalpc} $\class{C}$ has the $\struc{V}$-valued \AP. 
		\item\label{uvalapc} $\class{C}$ has the $\struc{U}$-valued \AP, for every $\struc{U}\in\sigma\class{C}$,
		\item\label{avalapc} $\class{C}$ has the $\struc{A}$-valued \AP, for every $\struc{A}\in\class{C}$.
	\end{enumerate}
\end{proposition}
\begin{proof}
	``(\ref{aepc})$\Rightarrow$(\ref{vvalpc})'': Given $\struc{A},\struc{B}_1,\struc{B}_2,f_1,f_2,h_1,h_2$ as required in the definition of the $\struc{V}$-valued \AP (see page~\pageref{vvalapdef}). Let $\struc{T}\coloneqq\langle h_1(B_1)\cup h_2(B_2)\rangle_{\struc{V}}$, and let $\tilde{h}_1$ and $\tilde{h}_2$ be the image restrictions of $h_1$ and $h_2$ to $\struc{T}$, respectively. By the \AEP there exist $\struc{C},\struc{T}', g_1,g_2,h,k$, such that the following diagram commutes: 
	
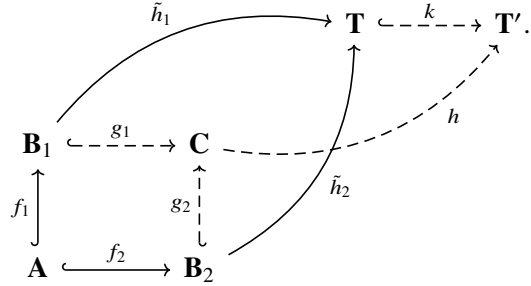
\[
	\begin{tikzcd}
		& & \struc{T}\arrow[r,dashed,hook,"k"] & \struc{T}'.\\
		\struc{B}_1\arrow[r,dashed,hook,"g_1"]\arrow[urr,bend left,"\tilde{h}_1"] & \struc{C}\arrow[rru,dashed,bend right,near end,"h"']\\
		\struc{A}\arrow[u,hook,"f_1"]\arrow[r,hook,"f_2"] & \struc{B}_2\arrow[u,dashed,hook',"g_2"]\arrow[uur,bend right,"\tilde{h}_2"']
	\end{tikzcd}
	\]
	Since $\struc{V}$ is the \Fraisse limit of $\class{C}$, there exists $\iota\colon\struc{T}'\injto\struc{V}$, such that 
	
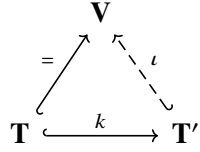
\[
	\begin{tikzcd}[column sep = small]
		& \struc{V} \\
		\struc{T} \arrow[ur,hook',"="] \arrow[rr,hook,"k"]& & \struc{T}'\arrow[ul,hook',dashed,"\iota"']
	\end{tikzcd}
	\] 
	commutes. Together this gives the following commuting diagram:
	
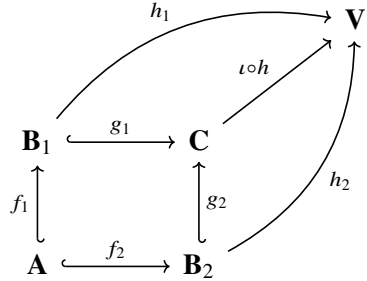
\[
	\begin{tikzcd}
		& & \struc{V} \\
		\struc{B}_1\arrow[urr,bend left,"h_1"]\arrow[r,hook,"g_1"] & \struc{C}\arrow[ur,"\iota\circ h"] \\
		\struc{A} \arrow[u,hook',"f_1"] \arrow[r,hook,"f_2"]& \struc{B}_2\arrow[u,hook',"g_2"']\arrow[ruu,bend right,"h_2"']
	\end{tikzcd}
	\]
	This shows that $\class{C}$ has the $\struc{V}$-valued \AP.

 ``(\ref{vvalpc})$\Rightarrow$(\ref{uvalapc})'':  This follows from the universality of $\struc{V}$ in conjunction with Observation~\ref{vvalapcsub}.
 
 ``(\ref{uvalapc})$\Rightarrow$(\ref{avalapc})$\Rightarrow$(\ref{aepc})'': Clear.
\end{proof}

\section{Proofs of the main results}

\subsection{Proof of Theorem~\ref{mainthm}}
For convenience, let us repeat the formulation of Theorem~\ref{mainthm}:
\begin{mainthm}
	Let $\class{C}\subseteq\class{S}_\Sigma$ be an age.  Then $\class{C}$ has the \AP and the \AEP if and only if $\pi\sigma\class{C}$ contains a universal and homogeneous ultrametric structure.
	 Moreover, any two universal and  homogeneous structures in $\pi\sigma\class{C}$ are metrically isomorphic.
\end{mainthm}
\begin{proof}[Proof of existence]\label{poe}
	Suppose that $\class{C}$ has the \AP and the \AEP. 
	
	Let $\struc{U}_0$ be a \Fraisse limit of $\class{C}$. It follows from Proposition~\ref{propertieseq} that $\class{C}$ has the $\struc{V}$-valued \AP, for every $\struc{V}\in\sigma\class{C}$. By Corollary~\ref{critunivskewhom}, there exists a universal skew-homogeneous homomorphism $\omega^1_0\colon\struc{U}_1\to\struc{U}_0$. By induction, this may be extended into an $\omega$-cochain $\cochain{U}= ((\struc{U}_i)_{i<\omega},(\Omega^j_i)_{i\le j<\omega})$, such that for all $i<\omega$ we have that $\Omega^{i+1}_i$ is universal and skew-homogeneous. The other homomorphisms $\Omega^{i+k}_i$ are defined naturally by suitably composing homomorphisms of the shape $\Omega^{i+1}_i$.  Let $\mstr{U}\coloneqq\mlim \cochain{U}$. By 	Observations~\ref{sequniv} and \ref{seqhom},  we have that $\mstr{U}$ is universal in $\pi\sigma\class{C}$ and  homogeneous.
\end{proof}

Before giving the proof of uniqueness of Theorem ~\ref{mainthm}, we need to move a standard observation from \Fraisse theory to the context of ultrametric structures:
\begin{observation}\label{injective}
	Let $\mstr{V}\in\pi\sigma\class{C}$ be universal and  homogeneous. Let $\mstr{A},\mstr{B}\in\pi\sigma\class{C}$ be pro-finitely generated, and suppose that $\mstr{A}$ is an isometric substructure of $\mstr{B}$. Let $\iota\colon\mstr{A}\injto\mstr{V}$ be a metric embedding. Then there is a metric embedding $\hat\iota\colon\mstr{B}\injto\mstr{V}$ that makes the following diagram commutative:
	\[
	\begin{tikzcd}
		\mstr{B} \arrow[r,hook,"\hat\iota"]& \mstr{V} \\
		\mstr{A}\ar[u,hook',"\le"]\ar[r,hook,"\iota"] & \mstr{V}\ar[u,equal].
	\end{tikzcd}
	\]
\end{observation}
\begin{proof}
	By universality, there exists an isometric embedding $\kappa$ of $\mstr{B}$ into $\mstr{V}$. Without loss of generality, we may assume that $\mstr{B}$ is a metric substructure of $\mstr{V}$. By isometric homogeneity, there exists a metric automorphism $f$ of $\mstr{V}$ such that the following diagram commutes:
	\[
	\begin{tikzcd}
		\mstr{B} \ar[r,hook,"\le"',"\kappa"] & \mstr{V}\ar[d,"f"] \\
		\mstr{A} \ar[u,hook',"\le"] \ar[r,hook,"\iota"]& \mstr{V}.
	\end{tikzcd}
	\] 
	With $\hat\iota\coloneqq f\circ\kappa$ the claim follows.
\end{proof}

\begin{proof}[Proof of uniqueness in Theorem~\ref{mainthm}]
	Let $\mstr{A},\mstr{B}\in\pi\sigma\class{C}$ be universal  homogeneous ultrametric structures.   Without loss of generality, suppose that $\mstr{A}=\mlim\cochain{A}$ where $\cochain{A}$ is given by
	\[
	\begin{tikzcd}
		\struc{A}_0 & \struc{A}_1\arrow[l,two heads,"\alpha^1_0"'] & \struc{A}_2\arrow[l,two heads,"\alpha^2_1"'] & \struc{A}_3\arrow[l,two heads,"\alpha^3_2"'] & \arrow[l,two heads,"\alpha_3^4"']\dots.
	\end{tikzcd}
	\]
	Let $(\struc{A}_\infty,(\alpha^\infty_i)_{i<\omega})$ be the corresponding canonical cone (in particular, the carrier of $\mstr{A}$ is $A_\infty$). 
	
	In the same way we may  assume that  $\mstr{B}=\mlim\cochain{B}$, where $\cochain{B}$ is given by  
	\[
	\begin{tikzcd}
		\struc{B}_0 & \struc{B}_1\arrow[l,two heads,"\beta^1_0"'] & \struc{B}_2\arrow[l,two heads,"\beta^2_1"'] & \struc{B}_3\arrow[l,two heads,"\beta^3_2"'] & \arrow[l,two heads,"\beta_3^4"']\dots.
	\end{tikzcd}
	\]
	Let $(\struc{B}_\infty,(\beta^\infty_i)_{i<\omega})$ be the corresponding canonical cone (in particular, the carrier of $\mstr{B}$ is $B_\infty$). 
	
	Using a back-and-forth argument, we construct a metric isomorphism between countable dense incomplete metric substructures of $\mstr{A}$ and $\mstr{B}$. 
	
	Let us formulate the construction as a game between two players. The game is identical to the classical Ehrenfeucht-\Fraisse game of length $\omega$ with a twist concerning the winning condition. A play $\bigl((\tup{a}_i)_{i<\omega},(\tup{b}_i)_{i<\omega}\bigr)$ is a win for player 2 if the assignment $\varphi\colon \tup{a}_i\mapsto\tup{b}_i$ induces a metric  isomorphism $\hat\varphi$ from $\langle\tup{a}_i\mid i<\omega\rangle_{\mstr{A}}$ to $\langle \tup{b}_i\mid i<\omega\rangle_{\mstr{B}}$.
	
	Let us describe a winning strategy for player 2. Suppose that after $n$ rounds the position of the game is $(\tup{a}_0,\dots,\tup{a}_{n-1}; \tup{b}_0,\dots,\tup{b}_{n-1})$, such that the assignment $\tup{a}_i\mapsto\tup{b}_i$ ($0\le i<n$) induces an isomorphism $\iota$ from $\mstr{A}_n\coloneqq\langle \tup{a}_0,\dots,\tup{a}_{n-1}\rangle_{\mstr{A}}$ to $\mstr{B}_n\coloneqq\langle \tup{b}_0,\dots,\tup{b}_{n-1}\rangle_{\mstr{B}}$. 
	
	Suppose further that player 1 chooses $\tup{a}_n\in A_\infty$. Let $\mstr{A}_{n+1}\coloneqq\langle \tup{a}_0,\dots,\tup{a}_n\rangle_{\mstr{A}}$.  Since $\mstr{B}$ is universal and homogeneous, by Observation~\ref{injective}, there exists an isometric embedding $\hat\iota\colon\mstr{A}_{n+1}\injto \mstr{B}$ that makes the following diagram commutative:
	\[
	\begin{tikzcd}
		\mstr{A}_{n+1} \ar[rr,hook,"\hat\iota"]& &\mstr{B} \\
		\mstr{A}_n \ar[r,hook,"\iota"']\ar[u,hook',"\le"]& \mstr{B}_n \ar[r,hook,"\le"']& \mstr{B}\ar[u,equals].
	\end{tikzcd}
	\]
	Now player 2 answers with $\tup{b}_n\coloneqq\hat\iota(\tup{a}_n)$. 
	
	On the other hand, if player 1 chooses $\tup{b}_n$ from $B_\infty$, then the choice of $\tup{a}_n$ from $A_\infty$ for player 2 goes analogously, using the universality and  homogeneity of $\mstr{A}$.
	
	It is  not hard to see that the described strategy is a winning strategy for player 2. 
	
	It remains to show that player 1 has a strategy to enforce that each game $\bigl((\tup{a}_i)_{i<\omega},(\tup{b}_i)_{i<\omega}\bigr)$ has the property that $\{\tup{a}_i\mid i<\omega\}$ is dense in $A_\infty$ and that $\{\tup{b}_i\mid i<\omega\}$ is dense in $B_\infty$. For this we fix bijections $\chi\colon\omega\to\bigcup_{i<\omega} \{i\}\times A_i$ and $\xi\colon\omega\to\bigcup_{i<\omega} \{i\}\times B_i$. Suppose again that  after $n$ rounds the position of the game is $(\tup{a}_0,\dots,\tup{a}_{n-1}; \tup{b}_0,\dots,\tup{b}_{n-1})$. If $n$ is even, then player 1 chooses the smallest $j<\omega$, such that with $\chi(j)=(k,x)$ we have that $x\notin\{\alpha^\infty_k(\tup{a}_i)\mid i<n\}$, and goes on to take any $\tup{a}_n\in A_\infty$ with $\alpha^\infty_k(\tup{a}_n)=x$.  
	
	If $n$ is odd, then player 1 chooses $\tup{b}_n\in B_\infty$ in an analogous way, using $\xi$. 
	
	If players 1  and 2 each play their respective strategy, then the resulting game defines an isometric isomorphism from a dense isometric substructure of $\mstr{A}$ to a dense isometric substructure of $\mstr{B}$.  Using the completion functor, this isomorphism extends to an isometric isomorphism from $\mstr{A}$ to $\mstr{B}$, as desired.
\end{proof}

\begin{proof}[Proof of the backwards-implication in Theorem~\ref{mainthm}]
	Let $\mstr{U}$ be a universal homogeneous ultrametric structure in $\pi\sigma\class{C}$. Without loss of generality, $\mstr{U}=\mlim \cochain{U}$, where $\cochain{U}= ((\struc{U}_i)_{i<\omega}, (u^j_i)_{i\le j<\omega})$ and where all $u_i^j$ are surjective . We are going to show that $\class{C}$ has the \AP and the \AEP. 
	
	Consider any $\omega$-cochain $\cochain{X}$ of the shape  
	\[
	\begin{tikzcd}
		\struc{X} & \struc{X}\ar[l,two heads,"1_{\struc{X}}"'] & \struc{X}\ar[l,two heads,"1_{\struc{X}}"']&\dots\ar[l,two heads,"1_{\struc{X}}"']\\
	\end{tikzcd}
	\] 
	This is a surjective $\omega$-cochain. By Observation~\ref{isoiso}(\ref{epsiso}) we have that both,  $\varepsilon_{\cochain{X}}\colon \Seq\mlim\cochain{X}\to\cochain{X}$ and $\varepsilon_{\cochain{U}}\colon \Seq\mlim\cochain{U}\to\cochain{U}$, are natural isomorphisms. By universality of $\mstr{U}$ there exists a metric embedding $\kappa\colon\mlim\cochain{X}\injto\mstr{U}$. Note that then $(\kappa_i)_{i<\omega}\coloneqq \varepsilon_{\cochain{U}}\circ\Seq\kappa\circ\varepsilon_{\cochain{X}}^{-1}$ is a natural embedding from $\cochain{X}$ into $((\struc{U}_i)_{i<\omega}, (u^j_i)_{i\le j<\omega})$. In other words, the following diagram commutes:
	\[
	\begin{tikzcd}
		\struc{U}_0 & \struc{U}_1\ar[l,two heads,"u^1_0"'] & \struc{U}_2\ar[l,two heads,"u^2_1"'] & \dots\ar[l,two heads,"u^3_2"']\\
		\struc{X} \ar[u,hook',"\kappa_0"]& \struc{X}\ar[u,hook',"\kappa_1"]\ar[l,two heads,"1_{\struc{X}}"'] & \struc{X}\ar[u,hook',"\kappa_2"]\ar[l,two heads,"1_{\struc{X}}"']& \dots \ar[l,two heads,"1_{\struc{X}}"']
	\end{tikzcd}
	\]
	Let us now show that $\class{C}$ has the \AP.
	Let $\struc{A},\struc{B},\struc{C}\in\class{C}$, and let $\iota\colon\struc{A}\injto\struc{B}$, $\lambda\colon\struc{A}\injto\struc{C}$. With the same reasoning as above, there exist natural embeddings $(\iota_i)_{i<\omega}$ and $(\lambda_i)_{i<\omega}$, such that the following diagram commutes:
	\[
	\begin{tikzcd}
		\struc{U}_0 & \struc{U}_1\ar[l,two heads,"u^1_0"'] & \struc{U}_2\ar[l,two heads,"u^2_1"'] & \dots\ar[l,two heads,"u^3_2"']\\
		\struc{B} \ar[u,hook',"\iota_0"]& \struc{B}\ar[u,hook',"\iota_1"] \ar[l,two heads,"1_{\struc{B}}"']& \struc{B}\ar[u,hook',"\iota_2"]\ar[l,two heads,"1_{\struc{B}}"'] &\dots\ar[l,two heads,"1_{\struc{B}}"']\\
		\struc{A} \ar[d,hook',"\lambda"']\ar[u,hook',"\iota"]& \struc{A}\ar[d,hook',"\lambda"']\ar[u,hook',"\iota"]\ar[l,two heads,"1_{\struc{A}}"'] & \struc{A}\ar[d,hook',"\lambda"']\ar[u,hook',"\iota"]\ar[l,two heads,"1_{\struc{A}}"'] &\dots\ar[l,two heads,"1_{\struc{A}}"']\\
		\struc{C} \ar[d,hook',"\lambda_0"']& \struc{C}\ar[d,hook',"\lambda_1"']\ar[l,two heads,"1_{\struc{C}}"'] & \struc{C}\ar[d,hook',"\lambda_2"']\ar[l,two heads,"1_{\struc{C}}"'] &\dots\ar[l,two heads,"1_{\struc{C}}"']\\
		\struc{U}_0 & \struc{U}_1\ar[l,two heads,"u^1_0"'] & \struc{U}_2\ar[l,two heads,"u^2_1"'] &\dots\ar[l,two heads,"u^3_2"'] 
	\end{tikzcd}
	\]
	By the homogeneity of $\mstr{U}$, there exists a natural automorphism $(\varphi_i)_{i<\omega}$ of $\cochain{U}$, such that for all $i<\omega$ we have $\varphi_i\circ\iota_i\circ\iota = \lambda_i\circ\lambda$.  Let $\struc{D}\coloneqq\langle\varphi_0(\iota_0(B))\cup \lambda_0(C)\rangle_{\struc{U}_0}$. Let $\tilde\lambda_0\colon\struc{C}\injto\struc{D}$ be the image restriction of $\lambda_0$ to $D$ and let $\tilde\iota_0\colon\struc{B}\injto\struc{D}$ be the image restriction of $\varphi_0\circ\iota_0$ to $D$. Then the following diagram commutes:
	\[
	\begin{tikzcd}
		\struc{U}_0\ar[rrd,"\varphi_0","\cong"',near start]\\
		 \struc{B} \ar[u,hook',"\iota_0"]\ar[r,hook,"\tilde\iota_0"]& \struc{D} \ar[r,hook,"="]& \struc{U}_0\\
		 \struc{A}\ar[u,hook',"\iota"] \ar[r,hook,"\lambda"]& \struc{C}\ar[u,hook',"\tilde\lambda_0"']\ar[ur,hook',"\lambda_0"']& .
	\end{tikzcd}
	\]
	In particular, $\class{C}$ has the \AP. 
	
	Our next step is to show that $\class{C}$ has the \AEP. To this end we consider the comma-category $(\func{F}\comma \func{F})$ (see the beginning of Section~\ref{commacats} for a definition of $\func{F}$). 
	The objects of $(\func{F}\comma\func{F})$ are triples $(\struc{A},h,\struc{B})$, where $\struc{A},\struc{B}\in\sigma\class{C}$ and where $h\colon\struc{A}\to\struc{B}$. The homomorphisms from $(\struc{A}_1,h_1,\struc{B}_1)$ to $(\struc{A}_2,h_2,\struc{B}_2)$ are of the shape $(\iota,\kappa)$, where $\iota\colon\struc{A}_1\injto\struc{A}_2$ and $\kappa\colon\struc{B}_1\injto\struc{B}_2$, such that the following diagram commutes:
	\[
	\begin{tikzcd}
		\struc{A}_2 \ar[r,"h_2"] & \struc{B}_2\\
		\struc{A}_1 \ar[r,"h_1"] \ar[u,hook',"\iota"] & \struc{B}_1.\ar[u,hook',"\kappa"]
	\end{tikzcd}
	\]
	With the same argument as in the proof Observation~\ref{algebroidal} (namely, observing that the premises of  \cite[Proposition 4.4]{PecPec18b} are trivially satisfied in this case), we observe that $(\func{F}\comma\func{F})$ is semi-algebroidal, and that the objects of $(\func{F}\comma\func{F})_{<\omega}$ are of the shape $(\struc{A},h,\struc{B})$, where $\struc{A},\struc{B}\in \class{C}$. 
	
	Next we are going to show that $(\func{F}\comma \func{F})_{<\omega}$ has the \AP. When this is done, then we may use \cite[Proposition 5.2(2)]{PecPec18b} in order to conclude that $\class{C}$ has the \AEP. 
	
	Let $(\struc{A}_1, a^1_0,\struc{A}_0),(\struc{B}_1,b^1_0,\struc{B}_0), (\struc{C}_1,c^1_0,\struc{C}_0)\in(\func{F}\comma\func{F})$.  In the first step let us assume that $a^1_0$, $b^1_0$, and $c_1^0$ are surjective. Like above, by the universality of $\mstr{U}$, there exist natural embeddings $(\nu_i)_{i<\omega}$ and $(\kappa_i)_{i<\omega}$, such that the following diagram commutes:
		\[
	\begin{tikzcd}
		\struc{U}_0 & \struc{U}_1\ar[l,two heads,"u^1_0"'] & \struc{U}_2\ar[l,two heads,"u^2_1"'] & \struc{U}_3 \ar[l,two heads,"u^3_2"'] & \dots \ar[l,two heads,"u^4_3"']\\
		\struc{B}_0 \ar[u,hook',"\nu_0"]& \struc{B}_1\ar[u,hook',"\nu_1"] \ar[l,two heads,"b^1_0"']& \struc{B}_1\ar[u,hook',"\nu_2"]\ar[l,two heads,"1_{\struc{B}_1}"'] & \struc{B}_1\ar[u,hook',"\nu_3"]\ar[l,two heads,"1_{\struc{B}_1}"'] & \dots \ar[l,two heads,"1_{\struc{B}_1}"']\\
		\struc{A}_0 \ar[d,hook',"\lambda_0"']\ar[u,hook',"\iota_0"]& \struc{A}_1\ar[d,hook',"\lambda_1"']\ar[u,hook',"\iota_1"]\ar[l,two heads,"a^1_0"'] & \struc{A}_1\ar[d,hook',"\lambda_1"']\ar[u,hook',"\iota_1"]\ar[l,two heads,"1_{\struc{A}_1}"'] & \struc{A}_1\ar[d,hook',"\lambda_1"']\ar[u,hook',"\iota_1"]\ar[l,two heads,"1_{\struc{A}_1}"'] & \dots \ar[l,two heads,"1_{\struc{A}_1}"']\\
		\struc{C}_0 \ar[d,hook',"\kappa_0"']& \struc{C}_1\ar[d,hook',"\kappa_1"']\ar[l,two heads,"c^1_0"'] & \struc{C}_1\ar[d,hook',"\kappa_2"']\ar[l,two heads,"1_{\struc{C}_1}"'] &\struc{C}_1\ar[l,two heads,"1_{\struc{C}_1}"']\ar[d,hook',"\kappa_3"'] & \dots\ar[l,two heads,"1_{\struc{C}_1}"']\\
		\struc{U}_0 & \struc{U}_1\ar[l,two heads,"u^1_0"'] & \struc{U}_2\ar[l,two heads,"u^2_1"'] &\struc{U}_3\ar[l,two heads,"u^3_2"'] & \dots \ar[l,two heads,"u^4_3"']
	\end{tikzcd}
	\]
	By the homogeneity of $\mstr{U}$, there exists a natural automorphism $(\varphi_i)_{i<\omega}$ of $\cochain{U}$, such that  $\varphi_0\circ\nu_0\circ\iota_0 = \kappa_0\circ\lambda_0$ and such that for all $1\le i<\omega$ we have that $\varphi_i\circ\nu_i\circ\iota_1 = \kappa_i\circ\lambda_1$. 

	Let $\struc{D}_1\coloneqq\langle\varphi_1(\nu_1(B_1))\cup \kappa_1(C_1)\rangle_{\struc{U}_1}$, and let $\struc{D}_0\coloneqq\langle\varphi_0(\nu_0(B_0))\cup \kappa_0(C_0)\rangle_{\struc{U}_0}$.  Then $u^1_0(D_1)= D_0$. Indeed:
	\begin{align*}
	u^1_0(D_1)\ &= u^1_0(\langle\varphi_1(\nu_1(B_1))\cup \kappa_1(C_1)\rangle_{\struc{U}_1}) = \langle u^1_0(\varphi_1(\nu_1(B_1))\cup \kappa_1(C_1))\rangle_{\struc{U}_0} \\
	&=\langle u^1_0(\varphi_1(\nu_1(B_1)))\cup u^1_0(\kappa_1(C_1))\rangle_{\struc{U}_0}= \langle \varphi_0(u^1_0(\nu_1(B_1))) \cup \kappa_0(c^1_0(C_1))\rangle_{\struc{U}_0}\\
	&= \langle \varphi_0(\nu_0(b^1_0(B_1)))\cup \kappa_0(c^1_0(C_1))\rangle_{\struc{U}_0} = \langle \varphi_0(\nu_0(B_0))\cup \kappa_0(C_0)\rangle_{\struc{U}_0} = D_0
	\end{align*}
	Let $d^1_0$ be the unique homomorphism that makes the following diagram commutative:
	\[
	\begin{tikzcd}
		\struc{U}_1 \ar[r,two heads,"u^1_0"]& \struc{U}_0\\
		\struc{D}_1 \ar[r,two heads,"d^1_0"]\ar[u,hook',"="]& \struc{D}_0\ar[u,hook',"="]
	\end{tikzcd}  
	\]
	For each $i\in\{1,2\}$ let $\tilde\kappa_i$ be the image restriction of $\kappa_i$ to $D_i$, and let $\tilde\nu_i$ be the image restriction of $\varphi_0\circ\nu_i$ to $D_i$. In particular, the following diagrams commute:
	\[
	\begin{tikzcd}
		& \struc{U}_0 \\
		\struc{B}_0\ar[r,hook,"\tilde\nu_0"'] \ar[ur,hook',"\varphi_0\circ\nu_0"]&\struc{D}_0\ar[u,hook',"="] & \struc{C}_0\ar[l,hook',"\tilde\kappa_0"]\ar[ul,hook',"\kappa_0"']
	\end{tikzcd}\quad\text{and}\quad
	\begin{tikzcd}
		& \struc{U}_1 \\
		\struc{B}_1\ar[r,hook,"\tilde\nu_1"'] \ar[ur,hook',"\varphi_1\circ\nu_1"]&\struc{D}_1\ar[u,hook',"="] & \struc{C}_1\ar[l,hook',"\tilde\kappa_1"]\ar[ul,hook',"\kappa_1"']
	\end{tikzcd}
	\]
	Then the following diagram commutes, too:
	
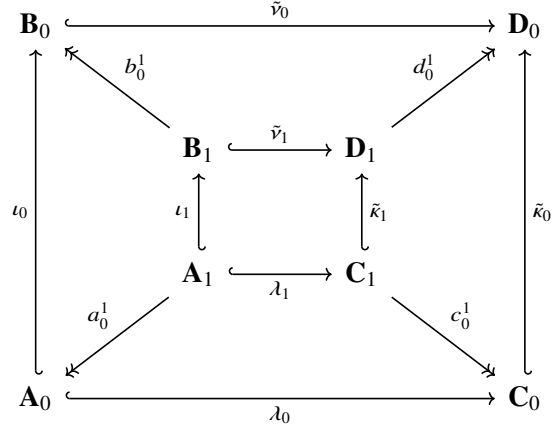
\[
	\begin{tikzcd}
		\struc{B}_0 \ar[rrr,hook,"\tilde\nu_0"]& & & \struc{D}_0\\
		&\struc{B}_1 \ar[r,hook,"\tilde\nu_1"]\ar[ul,two heads,"b^1_0"']& \struc{D}_1\ar[ur,two heads,"d^1_0"]\\
		&\struc{A}_1 \ar[r,hook,"\lambda_1"']\ar[u,hook',"\iota_1"]\ar[dl,two heads,"a^1_0"']& \struc{C}_1\ar[u,hook',"\tilde\kappa_1"']\ar[dr,two heads,"c^1_0"]\\
		\struc{A}_0 \ar[uuu,hook',"\iota_0"]\ar[rrr,hook,"\lambda_0"']& & & \struc{C}_0\ar[uuu,hook',"\tilde\kappa_0"']
	\end{tikzcd}
	\]
	This shows the amalgamation property for surjective objects of $(\func{F}\comma\func{F})_{<\omega}$. The general case bases on this one as follows: Let $(\struc{A}_1,a^1_0,\struc{A}_0),(\struc{B}_1,b^1_0,\struc{B}_0), (\struc{C}_1,c^1_0,\struc{C}_0)\in (\func{F}\comma\func{F})_{<\omega}$,   $(\iota_1,\iota_0)\colon (\struc{A}_1,a^1_0,\struc{A}_0)\injto (\struc{B}_1,b^1_0,\struc{B}_0)$, and $(\lambda_1,\lambda_0)\colon (\struc{A}_1,a^1_0,\struc{A}_0)\injto (\struc{C}_1,c^1_0,\struc{C}_0)$. Consider the epi-mono factorizations of $a^1_0$, $b^1_0$, and $c^1_0$:
	\[
	\begin{tikzcd}
		\struc{B}_1 & \struc{A}_1\ar[l,hook',"\iota_1"']\ar[r,hook,"\lambda_1"] & \struc{C}_1 \\
		\widetilde{\struc{B}}_1 \ar[u,hook',"="]& \widetilde{\struc{A}}_1\ar[l,hook',"\tilde\iota_1"',dashed]\ar[r,hook,"\tilde\lambda_1",dashed]\ar[u,hook',"="] & \widetilde{\struc{C}}_1\ar[u,hook',"="]\\
		\struc{B}_0 \ar[uu,bend left=40,"b^1_0",near end]\ar[u,two heads,"\tilde{b}^1_0"]& \struc{A}_0 \ar[uu,bend left=40,"a^1_0", near end]\ar[u,two heads,"\tilde{a}^1_0"]\ar[l,hook',"\iota_0"]\ar[r,hook,"\lambda_0"']& \struc{C}_0\ar[uu,bend left=40,"c^1_0",near end]\ar[u,two heads,"\tilde{c}^1_0"]
	\end{tikzcd}
	\]
	Here the dashed arrows exist and are unique because
	\begin{align*}
	\iota_1(\widetilde{A}_1) &= \iota_1(a^1_0(A_0)) = b^1_0(\iota_0(A_0))\subseteq b^1_0(B_0) = \widetilde{B}_1,\quad\text{and}\\ 
	\lambda_1(\widetilde{A}_1) &= \lambda_1(a^1_0(A_0)) = c^1_0(\lambda_0(A_0))\subseteq c^1_0(C_0) = \widetilde{C}_1.
	\end{align*}
	Using the amalgamation property for surjective objects and thrice the amalgamation property of $\class{C}$ we obtain  the following commuting diagram:
		
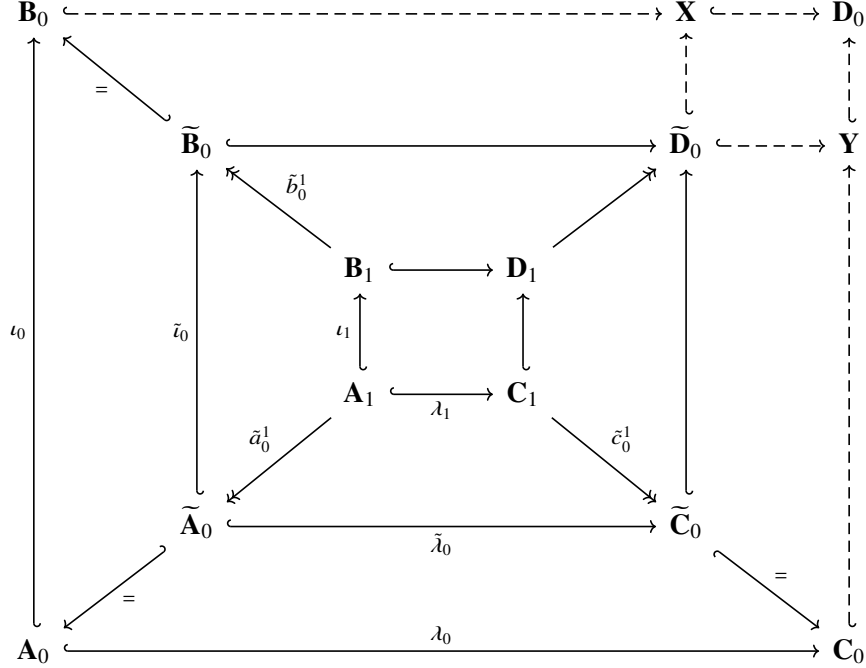
\[
	\begin{tikzcd}
	\struc{B}_0 \ar[rrrr,hook,dashed]& & & & \struc{X}\ar[r,hook,dashed]& \struc{D}_0\\
		&\widetilde{\struc{B}}_0 \ar[ul,hook',"="]\ar[rrr,hook]& & & \widetilde{\struc{D}}_0\ar[u,hook',dashed]\ar[r,hook,dashed] &\struc{Y}\ar[u,hook',dashed]\\
		& &\struc{B}_1 \ar[r,hook]\ar[ul,two heads,"\tilde{b}^1_0"']& \struc{D}_1\ar[ur,two heads]\\
		& &\struc{A}_1 \ar[r,hook,"\lambda_1"']\ar[u,hook',"\iota_1"]\ar[dl,two heads,"\tilde{a}^1_0"']& \struc{C}_1\ar[u,hook']\ar[dr,two heads,"\tilde{c}^1_0"]\\
		& \widetilde{\struc{A}}_0 \ar[dl,hook',"="]\ar[uuu,hook',"\tilde\iota_0"]\ar[rrr,hook,"\tilde\lambda_0"']& & & \widetilde{\struc{C}}_0\ar[uuu,hook']\ar[dr,hook',"="]\\
		\struc{A}_0 \ar[uuuuu,hook',"\iota_0"]\ar[rrrrr,hook,"\lambda_0"]& & & & & \struc{C}_0\ar[uuuu,hook',dashed]
	\end{tikzcd}
	\]
	This shows that $(\func{F}\comma\func{F})$ has the \AP. Thus, as already noted above, $\class{C}$ has the \AEP.
\end{proof}

\subsection{Proof of Theorem~\ref{secmainthm}}

For convenience, let us repeat the formulation of Theorem~\ref{secmainthm}:
	\begin{secmainthm}
		Let $\Sigma$ be a finite and purely relational signature, and let $\class{C}\subseteq\class{S}_\Sigma$  be an age with the \AP, the \AEP, and the \HAP. Let $\mstr{U}$ be the unique universal homogeneous metric structure in $\pi\sigma\class{C}$ postulated by Theorem~\ref{mainthm}. Then every isomorphism between finite substructures of its underlying structure $\struc{U}$ extends to an automorphism of\/ $\struc{U}$. Moreover, this extension may be chosen to be an isometry with respect to the metric of $\mstr{U}$ on every ball of radius $< 2^{-l}$ for some $l>0$. 
	\end{secmainthm}
	\label{proofsecmain} Recall that in the proof of Theorem~\ref{mainthm} a universal homogeneous metric structure in $\pi\sigma\class{C}$ is obtained as $\mlim\cochain{U}$, where $\cochain{U}=((\struc{U}_i)_{i<\omega}, (\Omega_i^j)_{i\le j<\omega})$ has very specific properties. In particular, $\struc{U}_0$ is a \Fraisse limit of $\class{C}$, and for each $i<\omega$ we have that $\Omega_i^{i+1}\colon\struc{U}_{i+1}\epito\struc{U}_i$ is a universal skew-homogeneous homomorphism (cf.~``Proof of existence.'' on page~\pageref{poe}).   
	
	More concretely, it was shown above that $\Omega_i^{i+1}$ may be chosen in such a way that $(\struc{U}_{i+1},\Omega^{i+1}_i,\struc{U}_i)$ is a universal homogeneous object in the comma-category $(\func{F}\comma \func{G}_{\struc{U}_i})$ (cf. the proof of Corollary~\ref{critunivskewhom}). 
	
	Next we are going to revise the definition of the cochain $\cochain{U}$ retaining all its crucial properties. Let $\struc{U}$ be the \Fraisse limit of $\class{C}$. Then, using that  $\class{C}$ has the \AEP and the \HAP, together with \cite[Theorem 6.8]{PecPec18b}, we conclude that $\struc{U}$ has a universal homogeneous \emph{endomorphism} $\Omega$. In other words  $(\struc{U},\Omega,\struc{U})$ is a universal homogeneous object in $(\func{F}\comma \func{G}_{\struc{U}})$. Using the same argument as in the proof of Corollary~\ref{critunivskewhom} we conclude that $\Omega$ is universal and skew-homogeneous, too. Thus, we may consider the cochain  $((\struc{U}_i)_{i<\omega}, (\Omega_i^j)_{i\le j<\omega})$ defined for each $i<\omega$ by $\struc{U}_i\coloneqq\struc{U}$, and $\Omega^{i+1}_i\coloneqq\Omega$. By all what we know we have that $\mstr{U}_\infty\coloneqq\mlim\cochain{U}$ is a universal homogeneous metric structure in $\pi\sigma\class{C}$. Let $\struc{U}_\infty$ be its underlying structure. 
	
	On $\struc{U}_\infty$ we may introduce two shift-operators $T_L$ and $T_R$ given by
	\begin{align*}
		T_L\colon & (x_0,x_1,x_2,\dots)\mapsto (x_1,x_2,x_3,\dots),\text{ and}\\
		T_R\colon & (x_0,x_1,x_2,\dots)\mapsto (\Omega(x_0),x_0,x_1,\dots). 
	\end{align*}
	
	\begin{lemma}
		$T_L$ and $T_R$ are mutually inverse bi-Lipschitz automorphisms of\/ $\struc{U}_\infty$ with respect to the metric $\delta$ of $\mstr{U}_\infty$. 
	\end{lemma}
	\begin{proof}
		Clearly, $T_L$ and $T_R$ are well defined self-mappings of $U_\infty$. 
	It is also clear that $T_L$ is a homomorphism. Moreover, $T_R$ is a homomorphism because $\Omega$ is.
	
	$T_L\circ T_R=1_{\struc{U}_\infty}= T_R\circ T_L$ follows from the definition of $T_L$ and $T_R$. 
		
	It remains to observe the bi-Lipschitz property. Suppose that $\delta(\tup{x},\tup{y})=2^{-n}$. If $n>0$, then we have that $\delta(T_R(\tup{x}),T_R(\tup{y}))= 2^{-(n+1)}$  and  $\delta(T_L(\tup{x}),T_L(\tup{y}))= 2^{-(n-1)}$. If we suppose that $n=0$, then we have at least that $1/2\le \delta(T_R(\tup{x}),T_R(\tup{y}))\le  1$,  and $\delta(T_L(\tup{x}),T_L(\tup{y}))= 1$. 
	It follows that $T_L$ and $T_R$ are bi-Lipschitz (with Lipschitz constant $L=2$)
 	\end{proof}
 
\begin{lemma}\label{shiftmet}
	Let  $\struc{A},\struc{B}$ be finite substructures of $\struc{U}_\infty$ and let  $\alpha\colon\struc{A}\to\struc{B}$ be an isomorphism. Then there exists an $l<\omega$, such that $T_L^l\circ\alpha\circ T_R^l\colon \langle T_L^l(A)\rangle_{\mstr{U}_\infty}\to \langle T_L^l(B)\rangle_{\mstr{U}_\infty}$ is a metric isomorphism.
\end{lemma}
\begin{proof}
	As $A$ is finite, for every relational symbol $\varrho$ of arity $n$  there are just finitely many $n$-tuples $\ttup{x}\in A^n$, such that $\ttup{x}\notin \varrho^\struc{A}$. 
	Let $m_\varrho<\omega$ be minimal, such that $\Omega^\infty_{m_\varrho}(\ttup{x})\notin \varrho^{\struc{U}_{m_\varrho}}$, for all $\ttup{x}\notin \varrho^\struc{A}$. 
	Let $m$ be the maximum over all $m_\varrho$.  Further, let $n<\omega$ be minimal, such that $\Omega^\infty_n\restr_{A\cup B}$ is one to one. Let $l=\max\{m,n\}$. Then for all $\tup{a},\tup{b}\in A$ we have 
	\[
	\delta(T_L^l(\tup{a}),T_L^l(\tup{b})) = 1= \delta(T_L^l(\alpha(\tup{a})),T_L^l(\alpha(\tup{b}))). 
	\]
	From this the claim follows at once.
\end{proof}

\begin{proof}[Proof of Theorem~\ref{secmainthm}.]
	Let $\struc{A},\struc{B}$ be finite substructures of $\struc{U}_\infty$. Let $l<\omega$ be such that $T_L^l\circ\alpha\circ T_R^l\colon \langle T_L^l(A)\rangle_{\mstr{U}_\infty}\to \langle T_L^l(B)\rangle_{\mstr{U}_\infty}$ is a metric isomorphism (exists by Lemma~\ref{shiftmet}). By homogeneity there exists a metric automorphism $h$ of $\mstr{U}_\infty$ that extends $T_L^l\circ\alpha\circ T_R^l$. Considered as automorphism of $\struc{U}_\infty$, $h$ is an isometry. Consider now the mapping $\tilde{h}\coloneqq T_L^l\circ h\circ T_R^l$. 
	We claim that $\tilde{h}$ extends $\alpha$. Most easily this is seen through the following commuting diagram 
	\[
	\begin{tikzcd}[column sep = small]
		\struc{U}_\infty\ar[rr,"h"] & &\struc{U}_\infty\\
		\struc{U}_\infty\ar[rr,"\tilde{h}",dashed]\ar[u,"T_R^l","\cong"'] & &\struc{U}_\infty\ar[u,"T_R^l","\cong"']\\
		 & \struc{A}.\ar[ul,"\le",hook'] \ar[ur,hook',"\alpha"']
	\end{tikzcd}
	\] 
	Observe now that  whenever $\delta(\tup{x},\tup{y})< 1$ then $\delta(T_R^l(\tup{x}), T_R^l(\tup{y}))= 2^{-l}\cdot \delta(\tup{x},\tup{y})$.  
	
	Consider $\tup{u}$ and $\tup{v}$ such that $\delta(\tup{u},\tup{v})< 2^{-l}$. Then   $\delta(T_L^l(\tup{u}), T_L^l(\tup{v})) =  2^{l}\cdot \delta(\tup{u},\tup{v}) <1$. On the other hand  $\delta(T_L^l(\tup{u}), T_L^l(\tup{v}))= \delta(h(T_L^l(\tup{u})), h(T_L^l(\tup{v})))$. Thus by the previous observation we obtain that 
	\begin{align*}
		\delta(\tilde{h}(\tup{u}),\tilde{h}(\tup{v})) &= \delta(T_R^l(h(T_L^l(\tup{u}))), T_R^l(h(T_L^l(\tup{v})))) = 2^{-l}\cdot  \delta(h(T_L^l(\tup{u})), h(T_L^l(\tup{v}))) \\
		&= 2^{-l}\cdot \delta(T_L^l(\tup{u}), T_L^l(\tup{v})) = 2^{-l}\cdot 2^l \delta(\tup{u},\tup{v})\\
		&= \delta(\tup{u},\tup{v}).
	\end{align*}
	Thus $\tilde{h}$ acts as an isometry on all metric balls of radius $<2^{-l}$.
\end{proof}

\section{Examples and concluding remarks}

Now that we have a criterion for $\pi\sigma\class{C}$ to contain a universal homogeneous ultrametric structure, it is high time to give a couple of examples. Very many \Fraisse classes are known and we need to decide which of the ``known suspects'' have the \AEP. Unfortunately, checking the \AEP directly  is not a pleasant task. Fortunately, Proposition~\ref{propertieseq} gives us some hints. To our mind spring \Fraisse classes that have in one way or the other canonical amalgams. Here ``canonical'' can be understood as a uniform method for the construction of amalgams. In the language of category theory this amounts to looking for \Fraisse classes in which for every span $\struc{B}\stackrel{\iota}{\injfrom}\struc{A}\stackrel{\kappa}{\injto}\struc{C}$ in $(\cat{C},\injto)$ there exists a cospan $\struc{B}\stackrel{\lambda}{\injto}\struc{D}\stackrel{\nu}{\injfrom}\struc{C}$ in $(\cat{C},\injto)$, such that the square:
\[
\begin{tikzcd}
	\struc{B} \ar[r,hook,"\lambda"]& \struc{D}\\
	\struc{A} \ar[u,hook',"\iota"]\ar[r,hook,"\kappa"']& \struc{C}\ar[u,hook',"\nu"']
\end{tikzcd}
\]
is a pushout-square in $\cat{C}$. This means, for all $\struc{X}\in\class{C}$ and for all homomorphisms $f\colon\struc{B}\to\struc{X}$, $g\colon\struc{C}\to\struc{X}$, such that $f\circ\iota=g\circ\kappa$ there exists a unique homomorphism $h\colon\struc{D}\to\struc{X}$ that makes the following diagram commutative:
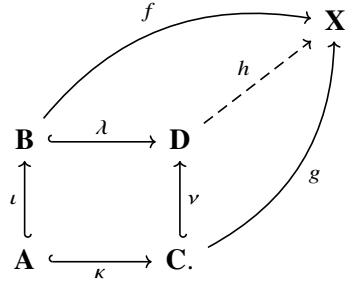
\[
\begin{tikzcd}
 & & \struc{X}\\
	\struc{B} \ar[urr,bend left,"f"]\ar[r,hook,"\lambda"]& \struc{D}\ar[ur,"h",dashed]\\
	\struc{A} \ar[u,hook',"\iota"]\ar[r,hook,"\kappa"']& \struc{C}\ar[uur,bend right,"g"']\ar[u,hook',"\nu"'].
\end{tikzcd}
\]
\Fraisse classes with this property are sometimes called \emph{strict} (cf.~\cite[Page 638]{DolMas12b}). Bearing in mind Proposition~\ref{propertieseq} it is quite obvious that any strict \Fraisse class has the \AEP. In particular every free amalgamation class has the strict amalgamation property. Some examples of strict \Fraisse classes are given by the classes of finite
\begin{itemize}
	\item simple graphs,
	\item $K_n$-free graphs, for every $n\ge 3$,
	\item non-strict posets,
	\item rational metric spaces,
	\item semilattices,
	\item distributive lattices,
	\item Boolean algebras,
	\item \dots.
\end{itemize}

An example of an age that has the \AEP but that does not have the strict amalgamation property is given by the class of finite total orders (with strict or with reflexive order relation).

Finding natural examples of \Fraisse classes that fail to have the \AEP appears do be more difficult. Here we give a somewhat artificial example just to make the point:
\begin{example}
	Consider the signature $\Sigma=(\Phi,\Rho)$, where $\Phi$ is empty and where $\Rho$ consists of one binary relational symbol $\varrho$ and of two unary symbols $P$ and $Q$. We may see $\Sigma$-structures as some kind of vertex-colored directed graphs. Consider the $\Sigma$-structure $\struc{\Gamma}$ on the set $\{x_1,x_2,x_3,x_4\}$ given in the figure below:

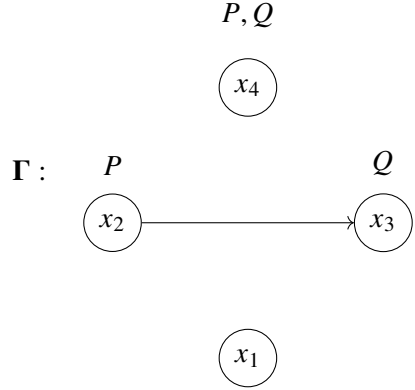
\[	
\struc{\Gamma}:\quad
\begin{tikzpicture}[baseline={(current bounding box.center)},scale=.9,auto=center,every node/.style={draw=black,circle,fill=blue!0}]
	\node (u1) at (0,0) {$x_1$};
	\node[label={$P$}] (u2) at (-2,2) {$x_2$};
	\node[label={$Q$}] (u3) at (2,2) {$x_3$};
	\node[label={$P,Q$}] (u4) at (0,4) {$x_4$};
	\draw[->] (u2) to (u3);
\end{tikzpicture}
\]
	It is clear, that $\struc{\Gamma}$ has no non-trivial local isomorphisms. Therefore it is homogeneous. 
	Let 
	\begin{align*}
		\struc{A}&\coloneqq\langle x_1\rangle_{\struc{\Gamma}},		& \struc{B}_1&\coloneqq\langle x_1,x_2\rangle_{\struc{\Gamma}}, & \struc{B}_2&\coloneqq\langle x_1,x_3\rangle_{\struc{\Gamma}},&\struc{T}&\coloneqq\langle x_4\rangle_{\struc{\Gamma}}.
	\end{align*}
	  For each $i\in\{1,2\}$ let $h_i\colon\struc{B}_i\to\struc{T}$ be the unique mapping. Clearly, $h_1$ and $h_2$ are homomorphisms. Moreover, the following diagram is commutative:
	  
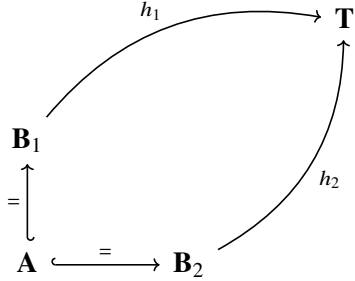
\[
	  \begin{tikzcd}
	  	& & \struc{T}\\
	  	\struc{B}_1\ar[rru,bend left,"h_1"] \\
	  	\struc{A}\ar[u,hook',"="]\ar[r,hook,"="] & \struc{B}_2\ar[uur,bend right,"h_2"']
	  \end{tikzcd}
	  \]
	  The unique minimal amalgam of $\struc{B}_1$ and $\struc{B}_2$ with respect to $\struc{A}$ in $\struc{\Gamma}$ is $\struc{C}=\langle x_1,x_2,x_3\rangle_{\struc{\Gamma}}$. Any joint extension of $h_1$ and $h_2$ to $\struc{C}$ within $\struc{\Gamma}$ maps $x_2$ and $x_3$ to $x_4$. However, $(x_2,x_3)\in\varrho^{\struc{\Gamma}}$, while $(x_4,x_4)\notin\varrho^{\struc{\Gamma}}$. This shows that $\Age(\struc{\Gamma})$ does not have the \AEP.
	\end{example}

	At the end of this paper let us  develop some  examples with a bit greater detail.
	As the first example let us consider finite linear orders:
	\begin{example}
		In the following let $\class{C}$ denote the class of finite linear orders with reflexive order relations (so we look onto structures of the shape $(A,\le)$, rather than the more common $(A,<)$; while this distinction is of no importance from the point of view of \Fraisse theory, it  makes a big difference in our setting). Being the age of the chain of rationals, $\class{C}$ has the \AP. It was shown in \cite[Proposition 3.23]{Kub13}, that $\class{C}$ has the \AEP and the \HAP. Arguing just like in the proof of Theorem~\ref{secmainthm} (see page~\pageref{proofsecmain}), the universal homogeneous ultrametric structure in $\pi\sigma\class{C}$ can be obtained as  $\mlim\cochain{U}$, where $\cochain{U}$ is of the shape:
		\[
		\begin{tikzcd}
			\struc{Q} & \ar[l, two heads, "\Omega"'] \struc{Q} & \ar[l, two heads, "\Omega"'] \struc{Q}  &  \ar[l, two heads, "\Omega"'] \cdots
		\end{tikzcd}
		\]
		and where $\Omega$ is a universal homogeneous endomorphism of $\struc{Q}=(\mathbb{Q},\le)$, i.e., $(\struc{Q},\Omega,\struc{Q})$ is a universal homogeneous object in $(\func{F}\comma \func{G}_{\struc{Q}})$. An realization of $\Omega$ was given in \cite[Remark on page 818]{PecPec18b}. In particular, $\Omega$ may be taken to be any locally constant surjective endomorphism of $\struc{Q}$. Next let us describe a concrete model of the $\omega$-cochain $\cochain{U}$: For every $n\ge 1$  Consider the lexicographic power  $\struc{Q}^n_\lex=(\mathbb{Q}^n,\le_\lex)$ of $\struc{Q}$, where $(a_i)_{i<n} <_\lex (b_i)_{i<n})$ if there exists some $i<n$ such that $a_i< b_i$ and such that for all $j<i$ we have $a_i=b_i$. 
		Being a countable dense chain with no end points it follows from Cantor's theorem that $\struc{Q}^n_\lex$ is isomorphic to $\struc{Q}$. Now if we define $\Omega^{n+1}_n\colon \struc{Q}^{n+1}_\lex\epito\struc{Q}^n_\lex$ according to $\Omega^{n+1}_n\colon (a_0,\dots,a_n)\mapsto (a_0,\dots,a_{n-1})$, then $\Omega^{n+1}_n$ is order preserving, surjective, and locally constant. In other words, it is a universal homogeneous homomorphism. Hence, instead of considering the $\omega$-cochain $\cochain{U}$ we may switch our attention to the naturally isomorphic $\omega$-cochain $\cochain{V}$ given by:
		\[
		\begin{tikzcd}
			\struc{Q}= \struc{Q}^1_\lex & \ar[l, two heads, "\Omega^2_1"'] \struc{Q}^2_\lex & \ar[l, two heads, "\Omega^3_2"'] \struc{Q}^3_\lex  &  \ar[l, two heads, "\Omega^4_3"'] \cdots.
		\end{tikzcd}
		\]
		A limit of $\cochain{V}$ is given by the lexicographic power $\struc{Q}^\omega_\lex$, together with the limiting cone $\Omega^\infty_i\colon \struc{Q}^\omega_\lex\to\struc{Q}^{i+1}_\lex$ given by $\tup{a}\mapsto (a_0,\dots,a_i)$. This can be completed to an ultrametric structure $\mstr{Q}$ by augmenting it with the standard ultrametric $\delta$ on $\mathbb{Q}^\omega$ and the multivalued predicate $\le^{\mstr{Q}}$ defined by
		\[
		\le^{\mstr{Q}}(\tup{a},\tup{b})\coloneqq
		\begin{cases}
			0 &  \tup{a}\le_\lex\tup{b} \\
			\delta(\tup{a},\tup{b}) & \text{else.}
		\end{cases}
		\]
		The resulting ultrametric structure $\mstr{Q}$ is universal and homogeneous in $\pi\sigma\class{C}$. At this point we make the following observation:
		\begin{obs}
			The poset $\struc{Q}^\omega_\lex$ is isomorphic to the chain of irrational numbers with their natural order.
		\end{obs}	
		\begin{proof}
			In the following we denote $\struc{R}=(\mathbb{R},\le)$, and $\struc{I}=(\mathbb{I},\le)$. Here $\mathbb{R}$ and $\mathbb{I}$ are the real and the irrational numbers, respectively.
			
			Our goal is to construct a sequence $(\sigma_i)_{i<\omega}$ such that:
			\begin{enumerate}[label=\arabic*), ref=\arabic*)]
				\item \label{c1} $\forall i<\omega: \sigma_i\subseteq\mathbb{I}^2$ is an equivalence relation with the following properties:
				\begin{enumerate}[label=(\alph*), ref=\arabic{enumi}\alph*)]
					\item \label{c1a}$\forall x\in\mathbb{I}:[x]_{\sigma_i}$ is an open interval with rational bounds in $\mathbb{I}$,
					\item \label{c1b}$\struc{I}/\sigma_i\cong\struc{Q}$,
				\end{enumerate} 
				\item \label{c2}$\forall a\in\mathbb{Q}\,\exists i<\omega, x\in\mathbb{I} : a = \sup [x]_{\sigma_i} \text{ or }  a=\inf [x]_{\sigma_i}$,
				\item \label{c3}$\forall i,j<\omega: i\le j \implies \sigma_j\subseteq\sigma_i$, 
				\item \label{c4}$\forall i<\omega\,\forall x\in\mathbb{I}: \langle [y]_{\sigma_{i+1}}\mid y\in [x]_{\sigma_i}\rangle_{\struc{I}/\sigma_{i+1}}\cong\struc{Q}$.  
			\end{enumerate}
			Before starting with the construction of the sequence $(\sigma_i)_{i<\omega}$, let us fix an enumeration $(q_i)_{i<\omega}$ of $\mathbb{Q}$. 
			
			Let $\Omega\colon\struc{Q}\epito\struc{Q}$ be a locally constant, surjective endomorphism (in other words, $\Omega$ is universal and homogeneous). Define $L\coloneqq \{\inf[x]_{\ker\Omega}\mid x\in\mathbb{Q}\}$ and $U\coloneqq \{\sup[x]_{\ker\Omega}\mid x\in\mathbb{Q}\}$. Let $\struc{B}\coloneqq \langle L\cup U\rangle_{\struc{R}}$. Then $\struc{B}$ is a countably infinite unbounded chain. By Cantor's theorem there exists an embedding $\iota$ of $\struc{B}$ into $\struc{Q}$. Note that $\iota$ may be chosen in such a way that its image in $\mathbb{Q}$ is unbounded, too. 
			
			Next, for $x,y\in\mathbb{I}$ define $x\mathop{\sigma_0} y :\Leftrightarrow (x,y)\cap \iota(B) = \emptyset$. By construction, $\sigma_0$ satisfies condition  \ref{c1}.
			
			Let $X$ be any class of $\sigma_0$ and let $q\coloneqq\inf X$. Suppose that $\sigma_0,\dots,\sigma_n$ are already constructed. Let $(X_i)_i<\omega$  be an enumeration of the equivalence classes of $\sigma_n$. For each $i<\omega$ let $l_i\coloneqq\inf X_i$ and $u_i\coloneqq\sup X_i$ in $\struc{R}$, and let $j_i$ be the smallest number such that $q_{j_i}\in(l_i,u_i)$. Let $\psi_i\colon\struc{Q}\to\langle(l_i,u_i)\cap\mathbb{Q}\rangle_{\struc{Q}}$ be an order isomorphism that maps $q$ to $q_{j_i}$ (this exists by Cantor's theorem and by the homogeneity of $\struc{Q}$). Then $\psi_i$ uniquely extends to an isomorphism $\widehat\varphi\colon\struc{R}\to\langle(l_i,u_i)\rangle_{\struc{R}}$. Finally, let $\varphi\colon\struc{I}\to \langle (l_i,u_i)\cap\mathbb{I}\rangle_{\struc{I}}$ be the restriction of $\widehat\varphi$ to $\struc{I}$. Clearly, $\varphi_i$ is an order isomorphism. Define $\sigma_{n+1,i}:=\varphi_i(\sigma_0)$. Then $\varphi_i(X)$ is an equivalence class of $\sigma_{n+1,i}$ and $q_{j_i}=\inf \varphi_i(X)$. Finally, define $\sigma_{n+1}\coloneqq \bigcup_{i<\omega}\sigma_{n+1,i}$. Thus the sequence $(\sigma_i)_{i<\omega}$ is completely defined. By design it has properties \ref{c1}, \ref{c2}, \ref{c3}, and \ref{c4}. 
			
			Now for each $i<\omega$ fine $\struc{I}_i\coloneqq \struc{I}/\sigma_i$, and $\pi^{i+1}_i\colon\struc{I}_{i+1}\epito\struc{I}_i$ according to $\pi^{i+1}_i\colon [x]_{\sigma_{i+1}}\mapsto [x]_{\sigma_i}$, for all $x\in \mathbb{I}$. Observe that for every $i<\omega$ we have that $\struc{I}_i\cong\struc{Q}$ and that $\pi^{i+1}_i$ is a universal homogeneous homomorphism. Let us denote by $\cochain{I}$  the $\omega$-cochain given by 
			\[
			\begin{tikzcd}
				\struc{I}_0 & \ar[l, two heads, "\pi^1_0"'] \struc{I}_1 & \ar[l, two heads, "\pi^2_1"'] \struc{I}_2  &  \ar[l, two heads, "\pi^3_2"'] \cdots.
			\end{tikzcd}
			\]
			Let $\struc{\tilde{I}}\coloneqq \plim \cochain{I}$ (in particular, $\struc{\tilde{I}}\cong\struc{Q}^\omega_\lex$). Define $\pi\colon \struc{I}\to\struc{\tilde{I}}$ according to $\pi\colon x\mapsto ([x]_{\sigma_i})_{i<\omega}$. We claim that $\pi$ is an isomorphism. Clearly, $\pi$ is order preserving. It remains to show that it is also bijective. To this end it suffices to prove that for all $(X_i)_{i<\omega}\in \tilde{\mathbb{I}}$ we have 
			\begin{equation}\label{clm}
			\bigcap_{i<\omega} X_i = \{x\}, \text{for some $x\in\mathbb{I}$.} 
			\end{equation}
			The rest of the proof will be dedicated to the proof of \eqref{clm}.
			
			For each $i<\omega$ chose some $x_i$ from $X_i$. Since $X_0$ is a bounded subset of $\mathbb{R}$, the sequence $(x_i)_{i<\omega}$ has an accumulation point in $\mathbb{R}$, say, $x$. 
			\begin{enumerate}[label=\textbf{Claim \arabic*:}, ref=\arabic*,nosep,align=left,leftmargin=0em,labelindent=0em,itemindent=1em,labelsep=0.5em,labelwidth=!]
				\item \label{clm1}$x\in\mathbb{I}$. Suppose on the contrary that $x\in\mathbb{Q}$. Then by \ref{c2} there exists some $y\in\mathbb{I}$, such that $x=\inf [y]_{\sigma_j}$ or $x=\sup [y]_{\sigma_j}$. We will handle only case that $x=\sup [y]_{\sigma_j}$, as the other case goes analogously. 
					\begin{enumerate}[label=\textbf{Claim \theenumi\alph*:}, ref=\theenumi\alph*,align=left,leftmargin=0em,labelindent=0em,itemindent=1em,labelsep=0.5em,labelwidth=!]
					  \item\label{clm1a} $X_j=[y]_{\sigma_j}$. If $x$ was a lower bound of $X_j$, then $\inf X_j>x$, by  \ref{c1b}. However, then $|x-x_k|>|\inf X_j-x_k|$, for all $k\ge j$, which is in contradiction with $x$ being an accumulation point of $(x_n)_{n<\omega}$. Thus $x$ is an upper bound of $X_j$. If $X_j < [y]_{\sigma_j}$, then $|x_k-x|\ge\varepsilon$, for all $k\ge j$, where $\varepsilon$ is the length of $[y]_{\sigma_j}$, in contradiction with $x$ being an accumulation point of $(x_n)_{n<\omega}$. Thus Claim \ref{clm1a} holds.
					\end{enumerate}
					Consider now $\hat{x}\coloneqq\sup X_{j+1}$. By \ref{c4} we have $\hat{x}<x$. However, then for all $k\ge j+1$ we have $|x_k-x|>|\hat{x}-x|$, and this too is a contradiction with $x$ being an accumulation point of $(x_n)_{n<\omega}$. It follows that the assumption that $x$ is rational was wrong and Claim \ref{clm1} is proven.
				\item\label{clm2} $x\in\bigcap_{i<\omega} X_i$. Suppose on the contrary that $x\notin \bigcap_{i<\omega} X_i$. Let $j<\omega$ be minimal with the property that $x\notin X_j$. Suppose that $x$ is an upper bound of $X_j$. Since $x$ is irrational, we have that $x>\sup X_j$. But then $|x-x_k|>|x-\sup X_j|$, for all $k\ge j$, in contradiction with $x$ being an accumulation point of $(x_n)_{n<\omega}$. The case that $x$ is a lower bound of $X_j$ is eliminated analogously. Thus Claim \ref{clm2} is proven. 
				\item\label{clm3} $\bigcap_{i<\omega} X_i = \{x\}$. Suppose on the contrary that there exists some $y\neq x$ such that $y\in\bigcap_{i<\omega} X_i$. Then inbetween $x$ and $y$ there exists some $q\in\mathbb{Q}$. By \ref{c2} there exists some $j<\omega$ and some $X\in\mathbb{I}/\sigma_j$, such that $q=\inf X$ or $q=\sup X$. However, then $x$ and $y$ can not belong to the same equivalence class of $\sigma_j$, a contradiction. Thus Claim \ref{clm3} is proven.\qedhere
			\end{enumerate}
		\end{proof}
		Thus we arrived at another description of the universal homogeneous ultrametric structure in $\pi\sigma\class{C}$: $\mstr{A} = (\mathbb{I}, \delta_\mstr{A},\le^{\mstr{A}})$, where
		\begin{align*}
			\delta_\mstr{A}(x,y) &= \begin{cases}
				0 & x=y \\
				2^{-i} & i = \min\{j<\omega \mid [x]_{\sigma_j}\neq [y]_{\sigma_j}\},
			\end{cases}\\
			\le^{\mstr{A}}(x,y) &= \begin{cases}
				0 & x\le y \\
				2^{-i} & i = \min\{j<\omega\mid [x]_{\sigma_j} \neq [y]_{\sigma_j}\}.
			\end{cases}
		\end{align*}

	\end{example}

	As the second example let us have a look onto the class of finite graphs:
	\begin{example}
		In the following, let $\class{C}$ denote the class of finite graphs in which each vertex carries a loop. It is well-known that $\class{C}$ has the free amalgamation property. As was noted above, this entails directly that $\class{C}$ has the \AP and the \AEP. The \Fraisse limit of $\class{C}$ is just the Rado graph with all loops added; let us denote this graph by $\struc{R}$. It is easy to see that $\class{C}$ has the \HAP. Appealing once more to the arguments in the proof of Theorem~\ref{secmainthm} (see page~\pageref{proofsecmain}), the universal homogeneous ultrametric structure in $\pi\sigma\class{C}$ can be obtained as  $\mstr{R}=\mlim\cochain{U}$, where $\cochain{U}$ is of the shape:
		\[\begin{tikzcd}
			\struc{R} & \ar[l, two heads, "\Omega"'] \struc{R} & \ar[l, two heads, "\Omega"'] \struc{R} \ar[l, two heads, "\Omega"'] &  \ar[l, two heads, "\Omega"'] \cdots
		\end{tikzcd}
		\]
		and where $\Omega$ is a universal homogeneous endomorphism of $\struc{R}$ (i.e., a universal homogeneous object in $(\func{F}\comma \func{G}_{\struc{R}})$). In order to grasp the structure of $\mstr{R}$, let us have a closer look onto $\Omega$: By \cite[Proposition~2.2]{DroGoe92} a homomorphism $\Omega\colon\struc{R}\to\struc{\widetilde{R}}$ is universal homogeneous in $(\func{F}\comma\func{G}_{\struc{\widetilde{R}}})$ if and only if for all $\struc{A},\struc{B}\in\class{C}$, for all commuting diagrams of the shape 
		
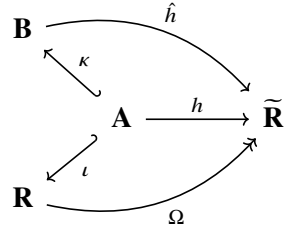
\[
		\begin{tikzcd}[sep=scriptsize]
			\struc{B}\ar[rrrd,bend left,"\hat{h}"]\\
			& \struc{A} \ar[rr,"h"]\ar[ul,hook',"\kappa"']\ar[dl,hook',"\iota"]& & \struc{\widetilde{R}}\\
			\struc{R}	\ar[rrru,bend right,two heads,"\Omega"']		
		\end{tikzcd}
		\]
		there exists $\hat\iota\colon\struc{B}\injto\struc{R}$ that makes the following diagramm commutative:
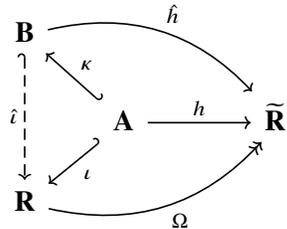
		\[
		\begin{tikzcd}[sep=scriptsize]
			\struc{B}\ar[dd,hook',dashed,"\hat\iota"']\ar[rrrd,bend left,"\hat{h}"]\\
			& \struc{A} \ar[rr,"h"]\ar[ul,hook',"\kappa"']\ar[dl,hook',"\iota"]& & \struc{\widetilde{R}}.\\
			\struc{R}	\ar[rrru,bend right,two heads,"\Omega"']		
		\end{tikzcd}
		\]
		On closer inspection this is equivalent to	the condition that for all finite disjoint subsets $A$ and $B$ of $R$ and for every $c\in \widetilde{R}$ that is connected to each vertex of $\Omega(A)$ there exists some $x\in R$ with $\Omega(x)=c$, such that $x$ is connected by an edge to each vertex from $A$ and to none from $B$. With this in mind, it becomes clear that for every $c\in R$, the subgraph of $\struc{R}$ induced by $\Omega^{-1}(c)$ is isomorphic to $\struc{R}$. Similarly it can be observed that for all edges $(c,d)$ of $\struc{R}$ the edges between $\Omega^{-1}(c)$ and $\Omega^{-1}(d)$ in $\struc{R}$ induce a random bipartite graph. 
		
		Having noted all this, we may envision $\mstr{R}$ as kind of a fractal structure. From very far away $\mstr{R}$ looks just like the Rado graph. When slowly closing up to  $\mstr{R}$, the apparent vertices start looking like clouds of vertices inducing again  Rado graphs and the edges resolve into copies of the random bipartite graph. Going even closer this repeats ad infinitum. 
		
		Let us round off this example by describing a concrete representation of a universal homogeneous homomorphism $\Omega$ between two different concrete copies of the Rado graph:
		As the codomain $\struc{\widetilde{R}}$ of $\Omega$ we take the ``standard'' representation  of the Rado graph with all loops added, where $\widetilde{R}=\omega$ and where for all $n\le m<\omega$ an edge is put between $n$ and $m$ whenever either $n=m$ or  the $n$-th digit in the binary expansion of $m$ is equal to $1$. Let us denote the edge relation of $\struc{R}$ by $\varrho$. 
		
		Let us now define the domain of $\Omega$. Consider the set $\{0,1\}^\ast$ consisting of all finite words in the alphabet $\{0,1\}$. As usually, by $\varepsilon$ we denote the empty word, and by $|w|$ we denote the number of letters of the word $w$. By $\#(w)$ we are going to denote the index of $w$ in the lexicographic enumeration of $\{0,1\}^\ast$. Here $v<w$ means that either $|v|<|w|$ or $|v|=|w|$ and $v<_{\operatorname{lex}} w$. Thus, we get, e.g., 
		\[
		\begin{array}{r|c|c|c|c|c|c|c|c}
			w & \varepsilon & 0 & 1 & 00 & 01 & 10 & 11 & \dots\\\hline
			\#(w) & 0 & 1 & 2 & 3 & 4 & 5 & 6 & \dots.
		\end{array}
		\]
		Moreover, for $w=a_0\dots a_n$ and for $i<\omega$ we define 
		\[
		w_i\coloneqq\begin{cases}
			a_i & i\le n\\
			0 & \text{else.}
		\end{cases}
		\]
		Next, let $\varphi\colon \omega\epito\omega$ be such that for all $n<\omega$ we have that $\varphi^{-1}(n)$ is infinite. E.g., we may take the sequence \texttt{A002262} from the online encyclopedia of integer sequences (see \cite{OEIS23}) given by:
		\[
		\begin{array}{r|c|c|c|c|c|c|c|c|c|c|c}
			n & 0 & 1 & 2 & 3 & 4 & 5 & 6 & 7 & 8 & 9 &\dots \\\hline
			\varphi(n) & 0 & 0 & 1 & 0 & 1 & 2 & 0 & 1 & 2 & 3 & \dots
		\end{array}
		\]
		Now, the domain of $\Omega$ is the graph $\struc{R}$ with vertex set $\{0,1\}^\ast$,  where for any $v,w\in\{0,1\}^\ast$ with $\#(v)\le\#(w)$ we put an edge between $v$ and $w$ if either $v=w$ or if $(\varphi(|v|),\varphi(|w|))\in\varrho$ and $w_{\#(v)}=1$. 

		Finally we define $\Omega\colon\struc{R}\to\struc{\widetilde{R}}$ according to 
		\[
		\Omega\colon w\mapsto \varphi(|w|).
		\]
		In order to see that $\Omega$ is indeed a universal homogeneous homomorphism, let $A$ and $B$ be finite disjoint subsets of $\{0,1\}^\ast$, and let $C:=\Omega(A)$. Let $c$ be any vertex of $\struc{\widetilde{R}}$ that is connected to all vertices from $C$. Let $n$ be such that $\varphi(n)=c$ and such that $n>\max\{\#(x)\mid x\in A\cup B\}$. Let $u\in\{0,1\}^\ast$, such that $|u|=n$ and such that for all $v\in A$ we have $u_{\#(v)}=1$ and for all $w\in B$ we have $u_{\#(w)}=0$. It is not hard to see that with this choice we have that $\Omega(u)=c$ and that $u$ is connected in $\struc{R}$ to all elements of $A$ and to none of $B$. By the remarks above, this implies that $\Omega$ is indeed universal and homogeneous.  
	\end{example}

\bibliographystyle{abbrv}
\bibliography{homMetStruc}

\begin{thebibliography}{10}

\bibitem{Ban70}
B.~Banaschewski.
\newblock Injectivity and essential extensions in equational classes of
  algebras.
\newblock In {\em Proc. {Conf}. on {Universal} {Algebra} ({Queen}'s {Univ}.,
  {Kingston}, {Ont}., 1969)}, pages 131--147. Queen's Univ., Kingston, Ont.,
  1970.

\bibitem{BKK17}
C.~Bargetz, J.~K{\k a}kol, and W.~Kubi{\'s}.
\newblock A separable {F}r\'{e}chet space of almost universal disposition.
\newblock {\em J. Funct. Anal.}, 272(5):1876--1891, 2017.

\bibitem{Yaa15}
I.~Ben~Yaacov.
\newblock Fra{\"\i}ss{\'e} limits of metric structures.
\newblock {\em The Journal of Symbolic Logic}, 80(1):100--115, 2015.

\bibitem{YaaBerHenUsv09}
I.~Ben~Yaacov, A.~Berenstein, C.~W. Henson, and A.~Usvyatsov.
\newblock Model theory for metric structures.
\newblock In {\em Model theory with applications to algebra and analysis.
  {Vol}. 2}, volume 350 of {\em London math. {Soc}. {Lecture} note ser.}, pages
  315--427. Cambridge Univ. Press, Cambridge, 2008.

\bibitem{Dol11}
I.~Dolinka.
\newblock The {Bergman} property for endomorphism monoids of some
  {Fra{\"\i}ss{\'e}} limits.
\newblock {\em Forum Mathematicum}, pages~--, Dec. 2011.

\bibitem{DolMas12b}
I.~Dolinka and D.~Ma{\v s}ulovi{\'c}.
\newblock A universality result for endomorphism monoids of some
  ultrahomogeneous structures.
\newblock {\em Proceedings of the Edinburgh Mathematical Society. Series II},
  55(3):635--656, 2012.

\bibitem{DroGoe90}
M.~Droste and R.~G{\"o}bel.
\newblock Universal domains in the theory of denotational semantics of
  programming languages.
\newblock In {\em Fifth {Annual} {IEEE} {Symposium} on {Logic} in {Computer}
  {Science} ({Philadelphia}, {PA}, 1990)}, pages 19--34. IEEE Comput. Soc.
  Press, Los Alamitos, CA, 1990.

\bibitem{DroGoe92}
M.~Droste and R.~G{\"o}bel.
\newblock A categorical theorem on universal objects and its application in
  {Abelian} group theory and computer science.
\newblock In L.~A. Bokut', Y.~L. Ershov, and A.~I. Kostrikin, editors, {\em
  Proc. {Int}. {Conf}. {Memory} {A}. {I}. {Mal}'cev, {Novosibirsk}/{USSR}
  1989}, volume 131(3) of {\em Contemp. {Math}.}, pages 49--74, Providence, RI,
  1992. AMS.

\bibitem{Fra53}
R.~Fra{\"\i}ss{\'e}.
\newblock Sur certaines relations qui g{\'e}n{\'e}ralisent l'ordre des nombres
  rationnels.
\newblock {\em C. R. Acad. Sci. Paris}, 237:540--542, 1953.

\bibitem{Fra53a}
R.~Fra{\"\i}ss{\'e}.
\newblock Sur l'extension aux relations de quelques propri{\'e}t{\'e}s connues
  des ordres.
\newblock {\em C. R. Acad. Sci. Paris}, 237:508--510, 1953.

\bibitem{GarKub15}
J.~Garbuli\'{n}ska-W\c{e}grzyn and W.~Kubi\'{s}.
\newblock A universal operator on the {G}urari\u{\i} space.
\newblock {\em J. Operator Theory}, 73(1):143--158, 2015.

\bibitem{IrvSol06}
T.~Irwin and S.~Solecki.
\newblock Projective {Fra{\"\i}ss{\'e}} limits and the pseudo-arc.
\newblock {\em Transactions of the American Mathematical Society},
  358(7):3077--3096, 2006.

\bibitem{Kan58}
D.~M. Kan.
\newblock Adjoint functors.
\newblock {\em Transactions of the American Mathematical Society}, 87:294--329,
  1958.

\bibitem{KraKub21}
A.~Krawczyk and W.~Kubi{\'s}.
\newblock Games with finitely generated structures.
\newblock {\em Annals of Pure and Applied Logic}, 172(10):Paper No. 103016, 13,
  2021.

\bibitem{K41}
W.~Kubi{\'s}.
\newblock Metric-enriched categories and approximate {F}ra{\"\i}ss{\'e} limits.
\newblock {\em arXiv:1210.6506}, Oct. 2012.

\bibitem{Kub13}
W.~Kubi{\'s}.
\newblock Injective objects and retracts of {Fra{\"\i}ss{\'e}} limits.
\newblock {\em Forum Mathematicum}, pages~--, Jan. 2013.

\bibitem{Kub14}
W.~Kubi{\'s}.
\newblock Fra{\"\i}ss{\'e} sequences: category-theoretic approach to universal
  homogeneous structures.
\newblock {\em Annals of Pure and Applied Logic}, 165(11):1755--1811, 2014.

\bibitem{Kub15}
W.~Kubi{\'s}.
\newblock Injective objects and retracts of {Fra{\"\i}ss{\'e}} limits.
\newblock {\em Forum Mathematicum}, 27(2):807--842, 2015.

\bibitem{Mac98}
S.~Mac~Lane.
\newblock {\em Categories for the working mathematician}, volume~5 of {\em
  Graduate texts in mathematics}.
\newblock Springer-Verlag, New York, 2 edition, 1998.

\bibitem{OEIS23}
{OEIS Foundation Inc.}
\newblock The {O}n-{L}ine {E}ncyclopedia of {I}nteger {S}equences, 2023.
\newblock Published electronically at \url{http://oeis.org}.

\bibitem{PecPec18b}
C.~Pech and M.~Pech.
\newblock Fra{\"\i}ss{\'e} limits in comma categories.
\newblock {\em Applied Categorical Structures}, Feb. 2018.

\end{thebibliography}

\end{document}